\documentclass[a4paper,11pt]{amsart}

\pdfoutput=1

\usepackage[text={400pt,660pt},centering]{geometry}

\usepackage{amsthm, amssymb, amsmath, amsfonts, mathrsfs}
\usepackage{mathtools}
\usepackage{scalerel} 
\usepackage[colorlinks=true, pdfstartview=FitV, linkcolor=blue, citecolor=blue, urlcolor=blue,pagebackref=false]{hyperref}

\usepackage{esint} 
\usepackage{MnSymbol} 
\usepackage{bbm}

\usepackage{appendix}

\usepackage{color}

\usepackage{microtype}

\definecolor{labelkey}{gray}{.8}
\definecolor{refkey}{gray}{.8}

\definecolor{darkgreen}{rgb}{0,0.5,0}
\definecolor{darkblue}{rgb}{0,0,0.7}
\definecolor{darkred}{rgb}{0.9,0.1,0.1}

\newtheorem{proposition}{Proposition}
\newtheorem{theorem}[proposition]{Theorem}
\newtheorem{lemma}[proposition]{Lemma}
\newtheorem{corollary}[proposition]{Corollary}

\theoremstyle{remark}
\newtheorem{remark}[proposition]{Remark}

\theoremstyle{definition}

\newtheorem{hypothesis}[proposition]{Hypothesis}

\numberwithin{equation}{section}
\numberwithin{proposition}{section}

\renewcommand{\leq}{\leqslant}
\renewcommand{\geq}{\geqslant}

\renewcommand{\subset}{\subseteq}
\newcommand{\mcl}{\mathcal}

\newcommand{\F}{\mathcal{F}}
\newcommand{\G}{\mathcal{G}}
\newcommand{\K}{\mathcal{K}}

\renewcommand{\S}{\mathsf{S}}
\newcommand{\DD}{\mathbf{D}} 
\newcommand{\DDb}{{\overbracket[1pt][-1pt]{\DD}}} 
\newcommand{\cm}{\mathbf{\Sigma}} 
\newcommand{\D}{{\mathsf{D}}} 
\newcommand{\api}{\overrightarrow{\pi}} 
\newcommand{\GG}{\tilde G^{(1)}}
\renewcommand{\c}{\mathbf{c}}
\newcommand{\cc}{\bar{\c}}

\newcommand{\E}{\mathbb{E}}
\newcommand{\Er}{\mathbb{E}_{\rho}}
\renewcommand{\Pr}{\mathbb{P}_{\rho}}

\newcommand{\expec}[1]{\left\langle #1 \right\rangle}
\renewcommand{\L}{\mathcal{L}}
\newcommand{\Lb}{{\overbracket[1pt][-1pt]{\L}}}
\newcommand{\Pb}{\bar P}

\newcommand{\La}{\Lambda}

\newcommand{\N}{\mathbb{N}}

\newcommand{\Ll}{\left}
\newcommand{\Rr}{\right}
\newcommand{\lhs}{left-hand side}
\newcommand{\rhs}{right-hand side}
\newcommand{\1}{\mathbf{1}}
\newcommand{\R}{\mathbb{R}}

\newcommand{\Z}{\mathcal{Z}}
\newcommand{\Zd}{{\mathbb{Z}^d}}

\renewcommand{\P}{\mathbb{P}}
\newcommand{\ov}{\overline}
\renewcommand{\bar}{\overline}

\renewcommand{\tilde}{\widetilde}

\newcommand{\ep}{\varepsilon}
\renewcommand{\d}{{\mathrm{d}}}
\newcommand{\var}{\mathbb{V}\!\mathrm{ar}}

\renewcommand{\epsilon}{\varepsilon}
\newcommand{\T}{\mathsf{T}}

\newcommand{\X}{\mathcal{X}} 

\newcommand{\g}{\mathbf{g}} 

\newcommand{\HH}{\mathcal{H}}

\newcommand{\cu}{{\scaleobj{1.2}{\square}}}

\newcommand{\Rd}{{\mathbb{R}^d}}

\renewcommand{\r}{\mathbf{r}}
\newcommand{\nub}{\bar{\nu}}

\newcommand{\fil}{\mathscr{F}}

\newcommand{\x}{\boldsymbol x}
\newcommand{\y}{\boldsymbol y}

\newcommand{\Ent}{\text{Ent}}
\newcommand{\sE}{\mathscr{E}} 
\newcommand{\sEb}{\bar{\mathscr{E}}}

\DeclareMathOperator{\supp}{supp}
\DeclareMathOperator{\diam}{diam}

\newcommand{\dH}{\dot{H}}  

\newcommand{\Ind}[1]{\mathbf{1}_{\left\{#1\right\}}}
\newcommand{\id}{\mathsf{Id}}
\newcommand{\norm}[1]{\left\Vert{#1}\right\Vert}
\newcommand{\bracket}[1]{\left\langle{#1}\right\rangle}

\usepackage{hyperref}

\title[Relaxation of non-gradient exclusion processes]{Relaxation to equilibrium of conservative dynamics II: non-gradient exclusion processes}

\author[C. Gu]{Chenlin Gu} 
\address[Chenlin Gu]{Yau Mathematical Sciences Center, Tsinghua University, China}
\email{gclmath@tsinghua.edu.cn}

\author[L. Yang]{Linzhi Yang} 

\address[Linzhi Yang]{Qiuzhen College, Tsinghua University, China}
\email{ylz24@mails.tsinghua.edu.cn}

\begin{document}

	\begin{abstract}
		For the speed-change exclusion process on $\Zd$ reversible with respect to the product Bernoulli measure, we prove that its semigroup $P_t$ satisfies a variance decay  $\var[P_t u] = C_u t^{-\frac{d}{2}} + o(t^{-\frac{d+\delta}{2}})$ for every local function $u$, with the constant $C_u$ explicitly characterized. This extends the result of Janvresse, Landim, Quastel and Yau in [Ann. Probab. 27(1) 325--360, 1999] to a non-gradient model. The proof combines the regularization argument in the previous work, and the chaos expansion in [Markov Process. Related Fields, 5(2) 125--162, 1999] by Bertini and Zegarlinski, via a new input from the homogenization theory.
		
		\bigskip
		
		
		\medskip
		
		
	\end{abstract}
	\maketitle
	
	\setcounter{tocdepth}{1}
	\tableofcontents
	
	\newpage
	
	%
	%
	%
	%
	%
	%
	%
	%
	\section{Introduction}

	The decay of diffusive semigroup is a fundamental topic in the research of partial differential equations and probability. Especially, different from the case in finite domain, the diffusion in the infinite volume space does not have spectral gap, thus the decay is not of exponential type. For the standard heat equation in $\Rd$, its semigroup has Gaussian distribution, so a polynomial decay of type $t^{-\frac{d}{2}}$ can be obtained by calculation. A less solvable example is the parabolic equation of divergence form, which requires the Aronson--Nash type estimate \cite{Aronson67, Nash58}. 
	
	A similar question is posed to conservative interacting particle systems of diffusive universality, and the Gaussian decay is also expected under invariant measure, i.e.
	\begin{align}\label{eq.GenralDecay}
		\var[P_t u] \leq C t^{-\frac{d}{2}},
	\end{align}
	with $P_t$ as the semigroup, and $u$ as a local function on particle configuration. This result is known as \emph{the relaxation to equilibrium}. One solvable example is the linear statistic of independent simple random walks on $\Zd$: let $\eta = (\eta_x)_{x \in \Zd} \in \N^{\Zd}$ stand for the configuration, and let $u$ be a function of type 
	\begin{align}\label{eq.u_lienar}
		u(\eta) = \sum_{x \in \Zd} f(x) \eta_x.
	\end{align}
	If every particle runs a simple symmetric random walk, then its generator $\L$ satisfies
	\begin{align}\label{eq.generator}
		(\L u)(\eta) = \sum_{x \in \Zd} \frac{1}{2d} \sum_{y \sim x} \Ll(f(y)-f(x) \Rr)\eta_x.
	\end{align}
	We thus have $P_t u (\eta) = \sum_{x \in \Zd} f_t(x) \eta_x$, with $f_t$ satisfying the discrete heat equation $\partial_t f_t(x) = \frac{1}{2d} \sum_{y \sim x} \Ll(f_t(y)-f_t(x) \Rr)$, and \eqref{eq.GenralDecay} follows like PDE setting.
	
	The argument above is quite limited for two reasons. Firstly, a general particle system has interactions. Secondly, there are more general functions $u$ than linear statistics. Actually, if $u$ is a linear statistic of form \eqref{eq.u_lienar}, then the study of variance can be reduced to \emph{the equilibrium fluctuation}, which is a well-understood topic for a large family of particle systems in \cite{broxrost,spo86,mpsw86,cha94,chayau92,lu94,cha96,fun96}; it is also extended to the non-equilibrium setting in \cite{fpv88,chayau92, jarmen18}; see also Chapter~11 of the monograph \cite{kipnis1998scaling} by Kipnis and Landim. Hence, we do focus on the general nonlinear functions $u$.

	One progress is \eqref{eq.GenralDecay} for the simple symmetric exclusion (SSEP), in which every site is occupied by at most one particle, i.e. $\eta \in \{0,1\}^{\Zd}$. Bertini and Zegarlinski proved this result in \cite{berzeg} via a \emph{generalized Nash estimate} for exclusion system 
	\begin{align}\label{eq.NashGeneralized}
		\var[u] \leq C \E[u(-\L u)]^{\frac{d}{d+2}} ||| u |||^{\frac{4}{d+2}},
	\end{align}
	where $\E[u(-\L u)]$ is the Dirichlet form associated to the generator of SSEP, and $||| \cdot |||$ is a specific norm. This estimate helps prove \eqref{eq.GenralDecay}, provided the contraction of $t \mapsto ||| P_t u |||$ holds.  Bertini and Zegarlinski also obtained \eqref{eq.NashGeneralized} for the exclusion under Gibbs measure in \cite{berzeg2}. The convergence to equilibrium under $L^p$ distance was later extended by  Ferrari, Galves, and Landim in \cite{FGL} using coupling argument.
	
	An unforeseen circumstance is that, the contraction of $t \mapsto ||| P_t u |||$ is only verified for symmetric exclusion, but unknown for other models. This is a major difference from the diffusion in PDEs, and really poses a challenge in the extensions to general particle systems. For this reason, Janvresse, Landim, Quastel, and Yau developed another approach in \cite{jlqy}, and derived the sharp asymptotic decay for the zero-range model that 
	\begin{align}\label{eq.JLQY}
		\var[u] = C_u t^{-\frac{d}{2}} + o(t^{-\frac{d}{2}}).
	\end{align}
	The meaning of the remainder is that ${\lim_{t \to \infty}t^{\frac{d}{2}}o(t^{-\frac{d}{2}}) = 0}$. The constant $C_u$ is also characterized explicitly for every local function $u$
	\begin{align}\label{eq.Cu}
		C_u = \frac{(\tilde u')^2 \chi}{\sqrt{(8\pi)^{d} \det [\DD]}}.
	\end{align}
	Here $\tilde u = \E[u]$ and $\tilde u'$ is the derivative in function of the density. The quantity $\chi$ is \emph{the compressibility}, and $\DD$ is \emph{the diffusion matrix}, and $\mathrm{det}[\DD]$ stands for its determinant.  
	Different from the generalized Nash inequality, the proof in \cite{jlqy} relies on the cutoff of dynamic, the regularization argument, and the estimate of entropy. Since this approach does not require a contraction of norm $||| \cdot |||$ mentioned above, it was adapted to several other situations: Landim and Yau obtained \eqref{eq.GenralDecay} (with a logarithm correction) for a Ginzburg--Landau model in \cite{lanyau}. For a family of exclusion processes under mixing condition, Cancrini, Cesi, and Roberto proved the following upper bound in \cite{ccr} 
	\begin{equation}\label{eq.ccr}
		\var[P_t u]  \leq C(\epsilon, u) t^{- (\frac{d-\epsilon}{2})}.
	\end{equation}
	Here $\epsilon > 0$ can be arbitrarily small, and $C(\epsilon, u)$ only depends on $\epsilon$ and $u$. This model is quite general, but the sharp convergence order was not attained; see the discussion in the first paragraph \cite[Page~219]{ccr}.  The argument was also partially applied in \cite{giunti2019heat} to study the heat kernel of the tagged particle in exclusion, by Giunti, Yu Gu, and Mourrat. In the work \cite{gu2020decay}, the first author of present paper adapted the same approach to derive a bound like \eqref{eq.GenralDecay} (with a logarithm correction) for a particle system in continuum configuration space.

	The object of present paper is a sharp estimate like \eqref{eq.JLQY} in non-gradient exclusion processes. We believe the sharp Gaussian upper bound is universal. Furthermore, we also wonder if it is the case for the leading constant $C_u$ in \eqref{eq.Cu}. As far as we know, the leading constant was only clarified for two examples, the zero-range model in \cite{jlqy} and SSEP in \cite{berzeg}, so a lot of work still remains to be done. The present paper restarts the exploration along this direction, and we choose the non-gradient exclusion process with the product Bernoulli invariant measure as an example. This is a canonical model and was studied in \cite{fuy, fun96, PatriciaJara} for the hydrodynamic limit, the equilibrium fluctuation, and the KPZ fluctuation under a weak drift. The present work justifies its sharp Gaussian bound as \eqref{eq.JLQY}, and confirms the leading constant with the same expression as \eqref{eq.Cu}. The new input in the proof comes from the recent progress in quantitative homogenization, especially the adaptation to the exclusion process \cite{funaki2024quantitative}. Notably, the homogenization method builds a bridge between the two existing approaches in \cite{berzeg} and \cite{jlqy}, so one can utilize the advantages inherent to each method.

	\subsection{Main result}
	Let $\Zd$ be the Euclidean lattice and $\X := \{0,1\}^{\Zd}$ stand for the space of the configuration of the exclusion process. We denote by $\eta = \{\eta_x: x \in \Zd \}$ the canonical element in $\X$. Here $\eta_x = 0$ means the site $x$ is vacant and $\eta_x = 1$ means the site is occupied. We denote by $y \sim x$ the neighbor for $x,y \in \Zd$, i.e. $\vert x - y\vert = 1$. Then $\{x,y\}$ is called an (undirected) bond. For every $\Lambda \subset \Zd$, we denote by $\Lambda^*$ the bond in $\Lambda$ that 
	\begin{align}\label{eq.defBond}
		\Lambda^* := \{ \{x,y\}: x,y \in \Lambda, x \sim y\}.
	\end{align} 
	
	For $x,y \in \Zd$, the exchange operator $\eta^{x,y}$ is defined as 
	\begin{align}\label{eq.exchange}
		(\eta^{x,y})_z := \Ll\{\begin{array}{ll}
			\eta_z, & \qquad z \neq x,y; \\
			\eta_y, & \qquad z = x; \\
			\eta_x, & \qquad z = y.
		\end{array}\Rr.
	\end{align}
	Especially, when $b = \{x,y\}$ is a bond, we also write $\eta^b$ instead of $\eta^{x,y}$, and define the Kawasaki operator $\pi_b \equiv \pi_{x,y}$ 
	\begin{align}\label{eq.Kawasaki}
		\pi_b F(\eta) := F(\eta^b) - F(\eta).
	\end{align} 
	For every $x \in \Zd$, the translation operator $\tau_x$ is defined as 
	\begin{align}\label{eq.translation1}
		(\tau_x \eta)_{y} := \eta_{x+y},
	\end{align}
	and given a function $F$ on $\X$, we also define $\tau_x F$ as 
	\begin{align}\label{eq.translation2}
		(\tau_x F)(\eta) := F(\tau_x \eta).
	\end{align}
	
	The \emph{non-gradient exclusion process} on $\Zd$ is defined by the generator below
	\begin{align}\label{eq.Generator}
		\L := \sum_{b \in (\Zd)^*} c_b(\eta) \pi_b = \frac{1}{2}\sum_{x,y \in \Zd: \vert x - y\vert = 1} c_{x,y}(\eta) \pi_{x,y},
	\end{align}
	where the family of functions
	\begin{align}
		\{c_b(\eta) \equiv c_{x,y}(\eta) = c_{y,x}(\eta); b=\{x,y\} \in (\Zd)^*\},
	\end{align}
	determine the jump rate of particles on the nearest bonds. This model is also called \emph{the speed-change Kawasaki dynamics} or \emph{the lattice gas} in the literature. 
	
	We suppose the following conditions for the jump rate throughout the paper without specific explanation.
	\begin{hypothesis}\label{hyp} The following conditions are supposed for  $\{c_b\}_{b \in (\Zd)^*}$.
		\begin{enumerate}
			\item Non-degenerate and local: $c_{x,y}(\eta)$ depends only on $\{\eta_z: \vert z - x\vert \leq \r\}$ for some integer $\r > 0$, and is bounded on two sides $1 \leq c_{x,y}(\eta) \leq \lambda$.
			\item Spatially homogeneous: for all $\{x,y\} \in (\Zd)^*$, $c_{x,y} = \tau_x c_{0,y-x}$.
			\item Detailed balance under Bernoulli measures: $c_{x,y}(\eta)$ is independent of $\{\eta_x, \eta_y\}$.
		\end{enumerate}
	\end{hypothesis}
	This model is known of non-gradient type, i.e. one cannot find functions $\{h_{i,j}\}_{1\leq i,j\leq d}$ such that $c_{0,e_i} (\eta)(\eta_{e_i} - \eta_0 ) = \sum_{j=1}^d \Ll((\tau_{e_j} h_{i,j})(\eta) - h_{i,j}(\eta)\Rr)$ for general $\{c_{b}\}_{b \in (\Zd)^*}$, with $\{e_i\}_{1 \leq i \leq d}$ the canonical basis of $\Zd$. 
	
	For the non-gradient model, its long-time behavior is governed by \emph{the diffusion matrix}.  We refer to \cite[Part II, Proposition~2.2]{spohn2012large} and \cite[(1.5)]{fuy} for the background and the definition. The diffusion matrix $\DD : (0,1) \to \R^{d \times d}$ is defined by
	\begin{align}\label{eq.Einstein}
		\DD(\rho) := \frac{\c(\rho)}{2 \chi(\rho)},
	\end{align}
	where $\chi(\rho)$ is \emph{the compressibility}
	\begin{align}\label{eq.defCompress}
		\chi(\rho) := \rho (1-\rho),
	\end{align}  
	and  $\c(\rho)$ is \emph{the effective conductivity} defined as follows. We construct a quadratic form with respect to the function $F \in \F_0^d$ 
	\begin{align}\label{eq.defQuadra}
		\xi \cdot \c(\rho; F) \xi = \frac{1}{2} \sum_{\vert x\vert = 1} \bracket{c_{0,x}\Ll(\xi \cdot \Ll\{ x(\eta_x - \eta_0) - \pi_{0,x}(\sum_{y \in \Zd} \tau_y F)\Rr\}\Rr)^2}_{\rho},
	\end{align}
	where $\F_0$ is the local function space on $\X$ and $\F_0^d := (\F_0)^d$, and $\bracket{\cdot}_{\rho}$ stands for the expectation under Bernoulli product measure of density $\rho \in (0,1)$. Then $\c(\rho)$ is the minimization of $\c(\rho; F)$
	\begin{align}\label{eq.defC}
		\xi \cdot \c(\rho) \xi := \inf_{F \in \F_0^d}\xi \cdot \c(\rho; F) \xi.
	\end{align}
	
	Let $P_t := e^{t \L}$ be the semigroup associated to \eqref{eq.Generator}, and we study its convergence to equilibrium. Our main result is the following one.
	\begin{theorem}\label{thm.main}
		There exists a positive exponent $\delta(d,\r, \lambda) > 0$, such that for every local function $u$, we have 
		\begin{align}\label{eq.main}
			\var_{\rho}[P_t u] = \frac{\tilde u' (\rho)^2 \chi(\rho)}{\sqrt{(8\pi t)^{d} \det [\DD(\rho)]}} + o(t^{-\frac{d+\delta}{2}}).
		\end{align}
		Here the function $\tilde u $ is defined as $\tilde u (\rho) := \bracket{u}_\rho$, and $\tilde u' (\rho)$ is the derivative of the mapping $\rho \mapsto \tilde u (\rho)$. The remainder depends on $u$ and satisfies ${\lim_{t \to \infty}t^{\frac{d+\delta}{2}}o(t^{-\frac{d+\delta}{2}}) = 0}$.
	\end{theorem}
	
	As discussed in the introduction, this result generalizes \cite[Theorem~1.1]{jlqy} in the non-gradient exclusion processes, with the leading order constant of the same form.

	\subsection{Sketch of the proof}
	
	The proof contains three ingredients, which can be summarized as ``regularization--homogenization--chaos expansion'' illustrated as follows
	\begin{align}\label{eq.KeyDecomposition}
		\boxed{P_t u =  \underbrace{(P_t u -  P_t R_{K(t)}u)}_{\text{regularization \cite{jlqy}}} + \underbrace{(P_t - \Pb_t) R_{K(t)}u}_{\text{homogenization}} + \underbrace{\Pb_t R_{K(t)}u}_{\text{chaos expansion \cite{berzeg}}}.}
	\end{align}
	The notations will be clarified in the following paragraphs. Among them, the chaos expansion appeared in \cite{berzeg} by Bertini and Zegarlinski, and the regularization was developed by Janvresse, Landim, Quastel, and Yau in \cite{jlqy}. We will review them and explain how the two approaches are linked via the homogenization method.  
	
	\subsubsection{Chaos expansion}
	We consider at first a symmetric exclusion process (SEP) with constant jump rate, whose generator $\Lb$ can be written as
	\begin{align}\label{eq.GeneratorEffective}
		\Lb := \frac{1}{2}\sum_{x \in \Zd} \sum_{y \in \Zd} Q_{y-x} \pi_{x,y},
	\end{align}
	Here $Q : \Zd \to \R_+$ is a symmetric jump rate of compact support. We denote by  $\Pb_t := e^{t\Lb}$ its semigroup and $\DDb$ its diffusion matrix. 	Although the leading order constant for SEP was not stated explicitly in \cite{berzeg}, most ingredients have already been included there. The main tool is the Wiener--It\^o analysis on Bernoulli random variables. The chaos expansion yields
	\begin{align}\label{eq.Fock}
		L^2(\X, \fil, \Pr) = \bigoplus_{n=0}^{\infty}\HH_n,
	\end{align}
	where $\HH_n$ is the subspace expanded by the normalized cylinder function $\prod_{i=1}^n (\eta_{x_i} - \rho)$. Every local function $u$ can be written as a sum in $L^2$
	\begin{align}\label{eq.chaos}
		u = \sum_{n=0}^\infty \Pi_{n} u,
	\end{align}
	with the projection $\Pi_{n} u \in \HH_n$. One nice property of SEP says that $\Pb_t$ is closed on $\HH_{n}$, thus it commutes with the projection operator
	\begin{align}\label{eq.commute}
		\Pb_t \Pi_{n} u = \Pi_{n} \Pb_t u.
	\end{align}	
	The evolution of $\Pb_t \Pi_{n} u$ is quite similar as the discrete heat equations, and the following two observations conclude a version \eqref{eq.main} for SEP
	\begin{align}\label{eq.varPb}
		\var_{\rho}[\Pb_t u] = \frac{\tilde u' (\rho)^2 \chi(\rho)}{\sqrt{(8\pi t)^{d} \det [\DDb]}} + o(t^{-\frac{d}{2}}).
	\end{align}
	\begin{enumerate}
		\item $\Pb_t \Pi_{1} u$ is the linear statistic and follows exactly the discrete heat equation as the solvable case in \eqref{eq.generator}. Then local CLT entails the asymptotic decay \eqref{eq.main} with the correct leading order constant.
		\item $\Pb_t \Pi_{n} u$ for $n \geq 2$ evolves like a discrete heat equation in $(\Zd)^n$, so it has a faster decay of order $t^{-\frac{nd}{2}}$ and is negligible in \eqref{eq.main}. Due to the exclusion rule, its rigorous proof is non-trivial and requires efforts, but has already been presented in \cite[Section~6, Theorem~17]{berzeg}.
	\end{enumerate}
	
	\subsubsection{Homogenization}
	The semigroup defined by \eqref{eq.Generator} is generally not closed under chaos expansion, so \eqref{eq.commute} does not hold for $P_t$. The homogenization aims to reduce $P_t$ to $\Pb_t$: recall the diffusion matrix $\DD(\rho)$ defined in \eqref{eq.Einstein}, we can then find a transition matrix $Q$ such that SEP defined in \eqref{eq.GeneratorEffective} has the same diffusion matrix $\DDb \equiv \DD(\rho)$. This type of construction is not unique, but every semigroup $\Pb_t = e^{t\Lb}$ has the similar long-time behavior as $P_t$. For this reason, we consider such $\Pb_t$ as a \emph{homogenized semigroup} and expect
	\begin{align}\label{eq.homoRough}
		P_t \simeq \Pb_t.
	\end{align}
	Especially, we need a quantitative estimate with respect to $t$.
	
	There are numerous references on the homogenization of PDEs and here we just list some of them. A classical reference is \cite{bensoussan1979boundary} by Bensoussan, Lions, and Papanicolaou. In the periodic coefficient setting, the homogenization of parabolic semigroup was derived by Zhikov and Pastukhova in \cite{zhikov2006estimates} with a sharp explicit rate. The quantitative homogenization in stochastic setting can be found in the early work \cite{NS} by Naddaf and Spencer, and a lot of results emerge in the last decade, especially since the work \cite{gloria2011optimal} by Gloria and Otto. We refer to the monograph \cite{AKMbook} by Armstrong, Kuusi, and Mourrat; see also another more recent monograph \cite{armstrong2022elliptic} and an informal introduction \cite{informal}. These results are also applied to a lot of models in probability and statistical mechanics; see \emph{``the historical remarks and further reading''} in \cite[Chapter~5]{armstrong2022elliptic}.
	
	The idea of homogenization is not new in particle systems, as the link between the two topics was revealed in the work \cite{varadhanII} of Varadhan. In the literature, it is usually mentioned as \emph{Varadhan's argument}, and was largely applied to the non-gradient models; see \cite{quastel, KLO94, varadhan1997mixing, fuy} for examples. In the last decade, \emph{the exclusion process in random environment} attracted attention, and the homogenization theory was utilized in \cite{quastel2006disorder, faggionato2003disorder, gonccalves2008scaling, jaranonhomogeneous, jara2006, faggionato2008,faggionato2022}. The Hodge decomposition, as a key ingredient in Varadhan's argument, was also discussed in \cite{bannai2024topological, bannai2021varadhan, BS25}, and can be extended to a large family of particle systems.
	
	Under the framework of \cite{AKMbook}, the quantitative homogenization theory for particle system was at first carried on a continuum space \cite{bulk, giunti2021smoothness, gu2024quantitative} by Giunti, Mourrat, Nitzschner, and the first author. Recently, in \cite{funaki2024quantitative}, Funaki, Wang and the first author extended the theory to the non-gradient exclusion process \eqref{eq.Generator} and proved the quantitative hydrodynamic limit. This work overcame the difficulty from the exclusion rule, thus paved way for the present paper. 
	
	Recall the object \eqref{eq.homoRough} for a general local function $u$. Combining the basis in \cite{funaki2024quantitative} and the techniques in \cite{gu2024quantitative}, the present paper implements a two-scale expansion in Wiener--It\^o analysis, and obtains that
	\begin{align}\label{eq.homoMixed}
		\norm{(P_t - \Pb_t) u}_{L^2}^2 \leq C \Ll(t^{-2\beta} \norm{u}_{L^2}^2 + \sum_{n=2}^{\infty}t^{-\frac{3nd}{8}}||| \Pi_n u |||^2_n\Rr).
	\end{align} 
	Here $\beta > 0$ is a fixed rate, and $L^2$ is a shorthand notation of $L^2(\X, \fil, \Pr)$,  and $||| \cdot |||_n$ is the triple norm on $\HH_n$. Viewing \eqref{eq.homoMixed}, the two semigroups are close for large $t$.
	
	\subsubsection{Regularization}
	Unfortunately, compared to the leading order in \eqref{eq.main}, the homogenization of rate $t^{-\beta}$ in \eqref{eq.homoMixed} is not necessarily small enough. A refined estimate is needed, but the solution turns out to be the spatial regularization in \cite{jlqy}.
	
	Let us review the approach in \cite{jlqy}. The four authors proposed a regularized version $P_t R_{K(t)}u$ with the following decomposition
	\begin{align}\label{eq.KeyDecomposition_JLQY}
		P_t u =  (P_t u -  P_t R_{K(t)}u) + P_t R_{K(t)}u.
	\end{align}
	Precisely, the regularization operator is defined as 
	\begin{align}\label{eq.RKt}
		R_{K(t)}u := \frac{1}{\vert \La_{K(t)} \vert}\sum_{x \in \La_{K(t)}}\tau_x u,
	\end{align}
	with $\La_{K(t)}$ a box of side length $K(t) \simeq t^{\frac{1-\epsilon}{2}}$. The choice $\epsilon > 0$ can be arbitrarily close to $0$ but strictly positive, so $K(t)$ is always mesoscopical compared to the diffusive scale $t^{\frac{1}{2}}$. This explains $P_t u \simeq P_t R_{K(t)}u$, and one key estimate in \cite{jlqy} stated
	\begin{align}\label{eq.regularization_JLQY}
		\var[(P_t u -  P_t R_{K(t)}u)] = o(t^{-\frac{d}{2}}).
	\end{align}

	The proof of \eqref{eq.regularization_JLQY} is quite robust and can be adapted in a lot of situations, including \cite{lanyau, ccr, gu2020decay}. The bottleneck to the sharp Gaussian decay is actually $\var[P_t R_{K(t)}u]$.  In the zero-range model, \eqref{eq.JLQY} was reached using the Boltzmann--Gibbs principle. For other models, it is less good, but $P_t R_{K(t)}u$ is easier to treat than $P_t u$, because the translation in \eqref{eq.RKt} creates spatial independence. Recall that $u$ is a local function, then a naive bound yields
	\begin{align}\label{eq.varRKt}
		\var[P_t R_{K(t)}u] \leq \var[R_{K(t)}u]  \simeq C t^{-\frac{(1-\epsilon)d}{2}}.
	\end{align}
	This upper bound is only suboptimal for an arbitrarily small exponent $\epsilon$.

	Now, we check the new decomposition \eqref{eq.KeyDecomposition} and notice:
	\begin{itemize}
		\item The term $(P_t u -  P_t R_{K(t)}u)$ is as good as \eqref{eq.regularization_JLQY}.
		\item The term $ \Pb_t R_{K(t)}u$ yields \eqref{eq.varPb} since the semigroup is of SEP.
		\item Whatever $\beta > 0$ in \eqref{eq.homoMixed} is,  it fills the last gap $\epsilon$ above for $(P_t - \Pb_t) R_{K(t)}u$.
	\end{itemize}
	Therefore, we get the desired result \eqref{eq.main}.

	\subsection{Organization of paper}
	All the claims and heuristics will be verified in the rest of paper. We will review some facts about the discrete heat equation in Section~\ref{sec.pre}. Afterward, the details of the three ingredients will be explained by order in Sections~\ref{sec.SEP},~\ref{sec.HomoLinear},~\ref{sec.Reg}: Section~\ref{sec.HomoLinear} involves the homogenization argument, which is new and will be the most technical part of the paper. Some arguments in Sections~\ref{sec.SEP},~\ref{sec.Reg} appeared in the previous work \cite{berzeg,jlqy}, but we still recap them to make the proof self-contained.


	\section{Preliminaries}\label{sec.pre}
	
	\subsection{Notations}\label{subsec.notation}
	\subsubsection{Probability space}
	For every $\Lambda \subset \Zd$, we denote by $\fil_{\Lambda}$ the $\sigma$-algebra generated by $(\eta_x)_{x \in \Lambda}$ and we write $\fil$ short for $\fil_{\Zd}$. Given $\rho \in (0,1)$ as the density of particle, let $\Pr = \operatorname{Bernoulli}(\rho)^{\otimes \Zd}$ stand for the Bernoulli product measure on $\X$, thus $(\X, \fil, \Pr)$ is the triplet of probability space most used in this paper. For the expectation under $\Pr$, we use the notation $\bracket{ \cdot }_{\rho}$ or $\Er[\cdot]$. We make use of $\P_{\rho, \Lambda}, \bracket{\cdot}_{\rho, \Lambda}$ when we restrict our measure on $(\eta_x)_{x \in \Lambda}$. We also denote by $\P_{\Lambda, N, \zeta}$ and $\bracket{\cdot}_{\Lambda, N, \zeta}$ for the probability and expectation under the canonical ensemble, i.e. $N$ particles distributed uniformly on different sites of $\Lambda$ with the configuration $\zeta$ on $\Lambda^c$. We usually omit $\zeta$ and just write them as  $\P_{\Lambda, N}$ and $\bracket{\cdot}_{\Lambda, N}$. 
	
	\subsubsection{Function spaces and norms}
	Given $p \geq 1$, the norm of the space $L^p(\X, \fil, \Pr)$ is usually written as $L^p$. Meanwhile, we use $L^p(\Rd)$ to highlight the $L^p$-norm with respect to the usual Lebesgue measure. Concerning the function on a countable set $V$, we define the norm
	\begin{align}\label{eq.lp}
		\norm{f}_{\ell^p(V)} := \Ll(\sum_{x \in V} \vert f(x)\vert^p\Rr)^{\frac{1}{p}}.
	\end{align} 
	For example, $\ell^p(\Zd)$ stands for the $\ell^p$-norm of functions defined on $\Zd$. Specifically, we denote by $\vert \cdot \vert_p$ as the $\ell^p$-distance for the vector in $\Rd$, and keep $\vert \cdot \vert$ for the Euclidean distance on $\Rd$.
	
	A function $f$ defined on $\Zd$ is of support $\Lambda$, then $f = 0$ on $\Lambda^c$, and we use $C_c(\Lambda)$ to represent the set of such functions. Meanwhile, concerning a function $F$ on $(\X, \fil, \Pr)$, we say it is of support $\Lambda \subset \Zd$ if $F$ is $\fil_{\Lambda}$-measurable, and we denote these functions by $\F_0(\Lambda)$. The support of a function is written as $\supp(\cdot)$.

	\subsubsection{Geometry}
	For every finite set $\Lambda \subset \Zd$, we define its volume and diameter respectively as
	\begin{align}\label{eq.defVolumeDiameter}
		\vert \Lambda\vert := \#\{x : x\in \Lambda\}, \qquad \diam(\Lambda) := \max\{\vert x - y\vert:x,y\in \Lambda\}.
	\end{align}
	We denote by $\La_L(x)$ the lattice cube centered at $x$ of side length $2L+1$ 
	\begin{align}\label{eq.defLa}
		\La_L(x):=x+\{-L,\cdots,L\}^d,
	\end{align}
	and we follow the convention $\La_L \equiv \La_L(0)$.

	\subsubsection{Conventions of constant}
	Given $u$ a local function on configuration space, we define the constant $\ell_u$ as 
	\begin{align}\label{eq.deflu}
		\ell_u := \min\{L \in \N_+: u \in \F_0(\Lambda_L)\}.
	\end{align}
	For every $\alpha > 0$, the notations $O(t^{-\alpha})$ and $o(t^{-\alpha})$ stand for the remainder in the relaxation to equilibrium. They are defined as
	\begin{align}\label{eq.defOo}
		\limsup_{t \to \infty} t^\alpha \vert O(t^{-\alpha}) \vert < \infty, \qquad \lim_{t \to \infty} t^\alpha o(t^{-\alpha}) = 0.
	\end{align}
	These remainders can depend on $d, \lambda, \r, \rho$ and the local function $u$ in the concrete statement.
	
	\subsection{Nash estimate for discrete heat equation}
	We recall some classical results about Nash estimate in the discrete setting. Let $Q$ be the transition matrix associated to a continuous-time symmetric random walk $(S_t)_{t \geq 0}$ on $\Zd$, then it satisfies
	\begin{enumerate}
		\item $Q_y \geq 0$;
		\item $\sum_{y \in \Zd} Q_y = 1$;
		\item $Q_{y} = Q_{-y}$ for all $y \in \Zd$.
	\end{enumerate}
	It also defines a discrete Laplacian $\frac{1}{2}\Delta_Q$, which is the generator of $(S_t)_{t \geq 0}$
	\begin{align}\label{eq.defDeltaQ}
		\Ll(\frac{1}{2}\Delta_Q f\Rr)(x):= \sum_{y\in \Zd}Q_{y-x}(f(y) - f(x)).
	\end{align}
	When the random walk $(S_t)_{t \geq 0}$ has finite second moment increment, its covariance matrix $\cm$ is well-defined as 
	\begin{align}\label{eq.defCov}
		\cm_{ij} := \sum_{y \in \Zd} Q_y y_i y_j.
	\end{align}
	The following estimate is classical for the associated discrete heat equation.
	\begin{lemma}\label{lem.NashZd}
		Suppose that we are given a symmetric transition matrix $Q$ with finite range. There exists a finite positive constant $C(d)$, such that for every function $f$, the solution of the equation  
		\begin{equation}\label{eq.heatZd}
			\left\{
			\begin{aligned}
				\partial_t f_t &=  \Ll(\frac{1}{2}\Delta_Q \Rr)f_t, \qquad t>0,\\
				f_0 &=f,
			\end{aligned}
			\right.
		\end{equation}
		satisfies the following estimate:
		\begin{equation}\label{eq.NashZd}
			\Ll|\norm{f_t}_{\ell^2(\Zd)}-\frac{\vert \mathbf{m}_{f} \vert}{\Ll((4\pi t)^d \det [\cm]\Rr)^{\frac{1}{4}}}\Rr|\leq C t^{-\frac{d+2}{4}} \sum_{x \in \Zd} \vert x f(x)\vert, \qquad \forall t>0,
		\end{equation}
		where $\mathbf{m}_{f} := \sum_{x\in\Zd}f(x)$.
	\end{lemma}	
	\begin{proof}
		Throughout the proof, we denote respectively by $\ast$ and $\circledast$ the convolution operator in $\Zd$ and $\Rd$
		\begin{align}\label{eq.defConvolution}
			(g \ast h)(x) := \sum_{y \in \Zd} g(x-y)h(y), \qquad (g \circledast h)(x) := \int_{\Rd}  g(x-y)h(y)\, \d y.
		\end{align}
		Let $\bar{p}_t$ be the semigroup $e^{\Ll(\frac{1}{2}\Delta_Q \Rr)t}$, then $f_t = \bar{p}_t \ast f$. Its behavior of $\ell^2$-norm can be studied in three steps.
		
		\textit{Step~1: local CLT.} 
		The classical Nash estimate applies to \eqref{eq.heatZd} and yields
		\begin{align*}
			\norm{f_t}_{\ell^p(\Zd)} \leq C(d,p)t^{-\frac{d}{2}(1-\frac{1}{p})}\norm{f}_{\ell^1(\Zd)}. 
		\end{align*}
		A decay of type $t^{-\frac{d}{4}}$ then appears when $p = 2$. Because of the local CLT, the large-scale behavior of $\bar{p}_t$ is close to the Gaussian distribution characterized by the covariance matrix $\cm$
		\begin{equation*}
			\Psi_t(x):=\frac{1}{\sqrt{(2\pi t)^d \det[\cm]}}e^{-\frac{x\cdot\cm^{-1} x}{2t}},\qquad x\in\Rd.
		\end{equation*}
		Then we apply the local CLT \cite[Theorem 2.1.3]{lawler2010random} to $(\bar{p}_t - \Psi_t)$, which will bring another factor $t^{-\frac{1}{2}}$. We combine the observations above and obtain that
		\begin{align}\label{eq.NashPass1}
			\norm{\bar{p}_t \ast f-\Psi_t \ast f}_{\ell^2(\Zd)} \leq C t^{-\frac{d+2}{4}}\norm{f}_{\ell^1(\Zd)}.
		\end{align}
		
		\textit{Step~2: $\Rd$ extension.}	The term $\Psi_t \ast f$ is naturally close to the convolution in continuous space. Let $[f]$ be the constant extension in $\Rd$ defined as 
		\begin{align*}
			\forall x\in \Zd, y \in x +\Ll[-\frac{1}{2}, \frac{1}{2}\Rr)^d, \qquad [f](y) := f(x).
		\end{align*}
		The convolution $\Psi_t \circledast [f]$ in $\Rd$ can be expressed as  
		\begin{align*}
			(\Psi_t \circledast [f])(x) = \int_{\Rd} \Psi_t(x-y)[f](y) \, \d y = \sum_{z \in \Zd} \Ll(\int_{z +\Ll[-\frac{1}{2}, \frac{1}{2}\Rr)^d} \Psi_t(x-y) \, \d y\Rr) f(z).
		\end{align*}
		Comparing it with \eqref{eq.defConvolution},  and using the regularity of $\Psi_t$, we have 
		\begin{align*}
			\vert (\Psi_t \circledast [f])(x) - (\Psi_t \ast f)(x)\vert &= \Ll\vert \sum_{z \in \Zd} \Ll(\int_{z +\Ll[-\frac{1}{2}, \frac{1}{2}\Rr)^d} \Psi_t(x-y) - \Psi_t(x-z)  \, \d y\Rr) f(z)\Rr\vert \\
			&\leq \frac{C}{t} \sum_{z \in \Zd}  \Psi_t(x-z)  \vert (x-z) f(z)\vert.\\
			&=\frac{C}{\sqrt{t}}\Ll(\tilde{\Psi}_t\ast |f|\Rr)(x),
		\end{align*}
		where $\tilde{\Psi}_t(x)= \frac{|x|}{\sqrt{t}}\cdot\Psi_t(x)$.
		By Young's convolution inequality, we have 
		\begin{align}\label{eq.NashPass2}
			\norm{\Psi_t \ast f-\Psi_t \circledast [f]}_{\ell^2(\Zd)} \leq \frac{C}{\sqrt{t}}\norm{\tilde{\Psi}_t}_{\ell^2(\Zd)}\norm{f}_{\ell^1(\Zd)}\leq Ct^{-\frac{d+2}{4}}\norm{f}_{\ell^1(\Zd)}.
		\end{align}
		The behavior of $\norm{\Psi_t \circledast [f]}_{\ell^2(\Zd)}$ close to $\norm{\Psi_t \circledast [f]}_{L^2(\Rd)}$ can be derived through similar estimates,
		\begin{equation}\label{eq.NashPass3}
			\Ll|\norm{\Psi_t \circledast [f]}_{\ell^2(\Zd)}-\norm{\Psi_t \circledast [f]}_{L^2(\Rd)}\Rr|\leq Ct^{-\frac{d+2}{4}}\norm{f}_{\ell^1(\Zd)}.
		\end{equation}
		
		\textit{Step~3: Gaussian expansion.} The quantity $\mathbf{m}_f$ can be considered as the mass of the initial condition, then the long-time behavior of $\norm{\Psi_t \circledast [f]}_{L^2(\Rd)}$ is close to $\norm{\mathbf{m}_{f} \Psi_t }_{L^2(\Rd)}$: using \cite[Theorem 4]{duoandikoetxea1992moments} with $k=0$, $q=2$, $p=1$ as the parameters there, we have
		\begin{align}\label{eq.NashPass4}
			\norm{\Psi_t \circledast [f]-\mathbf{m}_{f} \Psi_t }_{L^2(\Rd)}\leq Ct^{-\frac{d+2}{4}}\sum_{x \in \Zd} \vert x f(x)\vert.
		\end{align}
		We combine \eqref{eq.NashPass1}, \eqref{eq.NashPass2}, \eqref{eq.NashPass3} and \eqref{eq.NashPass4} to arrive at the desired result \eqref{eq.NashZd}.
	\end{proof}

	\subsection{SEP associated to the diffusion matrix}
	
	The object of this paragraph is to construct a SEP with the same diffusion matrix $\DD(\rho)$. Recall that $\DD(\rho)$ is a positive definite matrix, but not necessarily diagonal.
	
	We start from a more basic example. The following lemma asserts that for every positive definite matrix $\cm$, we can construct a continuous-time symmetric random walk on $\Zd$ with $\cm$ as its covariance matrix.
	
	\begin{lemma}\label{lem.CovRW}
		Given a matrix $\cm\in \mathbb{R}_{sym}^{d\times d}$, satisfying $\id \leq \cm \leq  C \id$, then there exists a transition matrix $(Q_x)_{x \in \Zd}$ such that $\cm$ is the covariance matrix of a continuous-time symmetric random walk on $\Zd$ associated to $Q$, i.e. 
		\begin{equation}\label{eq.CovRW}
			\forall 1 \leq i,j \leq d, \qquad	\cm_{ij} = \sum_{x \in \Zd }Q_x x_i x_j .
		\end{equation}
		Moreover, following two conditions hold for $Q$.
		\begin{enumerate}
			\item For the canonical basis $\{e_i\}_{1 \leq i \leq d}$, we have $Q_{e_i} \geq 
			\frac{1}{4}$.
			\item The support of the transition matrix satisfies $\supp(Q) \subset \Lambda_{4Cd^2}$
		\end{enumerate} 
	\end{lemma}
	\begin{proof}	
		Following \eqref{eq.defCov}, it suffices to prove the existence of a transition matrix $(\tilde{Q}_x)_{x \in \Zd}$ such that
		\begin{align}\label{eq.CovRWPre}
			\cm = \sum_{x \in \Zd } \tilde{Q}_x x x^{\T},
		\end{align}
		where we treat $x \in \Zd$ as a $d \times 1$ matrix, with $x^{\T}$ as its transpose. We then use the natural symmetrisation $Q_x := \frac{\tilde{Q}_x + \tilde{Q}_{-x}}{2}$, and obtain the desired symmetric transition matrix. The construction of \eqref{eq.CovRWPre} can be divided into two steps.

		\textit{Step~1: approximation of diagonalization.} Actually, \eqref{eq.CovRW} is obvious using diagonalization if the random walk is defined on $\Rd$. However, we need a random walk on $\Zd$, so a good approximation is needed. 
		Since $\cm\geq \id$, we decompose $\cm$ into
		\begin{equation*}
			\cm=\id+\tilde{\cm},
		\end{equation*}
		such that $\tilde \cm$ is also a positive semi-definite symmetric matrix. Then there exists an orthogonormal matrix $\mathbf P$ such that 
		\begin{equation*}
			\mathbf P^{\T} \tilde{\cm} \mathbf P=\mathbf \Lambda,
		\end{equation*}
		where $\mathbf \Lambda= \mathrm{diag}(\lambda_1,\cdots, \lambda_d)$ is a diagonal matrix with $0 \leq \lambda_i \leq  C$. Denoting by $\mathbf P=(p_1,p_2,\cdots, p_d)$, then we have the decomposition
		\begin{equation*}
			\tilde \cm=\sum_{i=1}^{d}\lambda_i p_i p_i^{\T}.
		\end{equation*}
		Notice that $ p_i$ is a vector in $\Rd$ instead of $\Zd$, thus we need some further approximation. We pick a very large $N \in \N$ and define 
		\begin{equation}\label{eq.SigmaApproximationP}
			\tilde { p}_i := \lfloor N  p_i \rfloor, \qquad  \tilde \lambda_i :=\frac{\lambda_i}{N^2},
		\end{equation}
		where $\lfloor y \rfloor=(\lfloor y_1 \rfloor, \lfloor y_2\rfloor,\cdots, \lfloor y_d \rfloor)$. Then the matrix $\tilde {\tilde \cm}$ defined by
		\begin{equation}\label{eq.SigmaApproximation}
			\tilde {\tilde \cm} :=\sum_{i=1}^{d}\tilde\lambda_i\tilde { p}_i\tilde { p}_i^{\T},
		\end{equation}
		is a good approximation of $\tilde \cm$ with error
		\begin{equation}\label{eq.SigmaError}
			\epsilon := \max_{1\leq i,j \leq d}\vert \tilde \cm_{ij}-\tilde {\tilde \cm}_{ij}\vert \leq \sum_{i=1}^{d}\lambda_i \Ll\vert p_i p_i^{\T} - \frac{\lfloor N  p_i \rfloor}{N} \frac{\lfloor N  p_j \rfloor}{N} \Rr\vert_{\infty} \leq \frac{2C d}{N},
		\end{equation}
		where $\vert\cdot\vert_\infty$ is defined as $\vert\cm|_\infty:=max_{1\leq i,j\leq d}|\cm_{ij}|$ for a matrix $\cm$. The value $N$ is determined later.
		
		\textit{Step~2: diagonalization of remainder.} We need to treat the remainder
		\begin{align}\label{eq.SigmaRemainder}
			\mathbf R:= \cm -\tilde{\tilde \cm} = \id +\tilde \cm-\tilde{\tilde \cm} .
		\end{align}
		Viewing \eqref{eq.SigmaError}, it satisfies the bound 
		\begin{equation}\label{eq.SigmaRemainderBound}
			\begin{split}
				\forall 1\leq i \neq j\leq d, \qquad \vert \mathbf R_{ij} \vert &\leq \ep,\\
				\forall 1\leq i\leq d, \qquad \vert \mathbf R_{ii}-1\vert &\leq \ep.
			\end{split}
		\end{equation}
		The symmetric matrix $\mathbf R$ can be decomposed to 
		\begin{equation}\label{eq.CovDecom2Sum}
			\mathbf R=\sum_{1\leq i < j\leq d}\vert \mathbf R_{ij}\vert(  e_i+ \mathrm{sgn}(\mathbf R_{ij})  e_j) (  e_i+ \mathrm{sgn}(\mathbf R_{ij})  e_j)^{\T}+\tilde {\mathbf R},
		\end{equation}
		where $\{  e_i\}_{1\leq i\leq d}$ is the canonical basis of $\Rd$ and $\tilde {\mathbf R}$ is a diagonal matrix defined as 
		\begin{align*}
			\tilde {\mathbf R}_{ii} := \mathbf R_{ii} - \sum_{j \neq i} \vert \mathbf R_{ij}\vert.
		\end{align*}
		We recall the bound in \eqref{eq.SigmaRemainderBound}. When $\ep\leq \frac{1}{2d}$, which corresponds to $N \geq 4Cd^2$ in \eqref{eq.SigmaError}, we have
		\begin{align}\label{eq.RttBound}
			\forall 1\leq i\leq d, \qquad \tilde {\mathbf R}_{ii} = \mathbf R_{ii} - \sum_{j \neq i} \vert \mathbf R_{ij} \vert \geq 1 - d \epsilon \geq \frac{1}{2}.
		\end{align}
		We just take $N = \lfloor 4Cd^2 \rfloor + 1$, then we can write
		\begin{equation}\label{eq.CovDecom3Sum}
			\tilde {\mathbf R} =\sum_{i=1}^d \tilde {\mathbf R}_{ii}  e_i  e_i^{\T}.
		\end{equation}
		Combining \eqref{eq.SigmaApproximation}, \eqref{eq.SigmaRemainder}, \eqref{eq.CovDecom2Sum}, \eqref{eq.CovDecom3Sum}, we obtain 
		\begin{align*}
			\cm = \sum_{i=1}^{d}\tilde\lambda_i\tilde { p}_i\tilde { p}_i^{\T} + \sum_{1\leq i < j\leq d}\vert \mathbf R_{ij}\vert(  e_i+ \mathrm{sgn}(\mathbf R_{ij})  e_j) (  e_i+ \mathrm{sgn}(\mathbf R_{ij})  e_j)^{\T}+\sum_{i=1}^{d}\tilde {\mathbf R}_{ii}  e_i  e_i^{\T},
		\end{align*}
		which is \eqref{eq.CovRWPre} and concludes \eqref{eq.CovRW}. The domination $Q_{e_i} \geq \frac{1}{4}$ comes from \eqref{eq.RttBound}. The support of transition matrix should be bounded by $\tilde { p}_i$ in \eqref{eq.SigmaApproximationP}, which gives  $\supp(Q) \subset \Lambda_{4Cd^2}$.
	\end{proof}
	
	We can derive the corresponding version in SEP.
	\begin{corollary}\label{cor.CovSEP}
		Given a matrix $\DD(\rho) \in \mathbb{R}_{sym}^{d\times d}$, satisfying $\id \leq \DD(\rho) \leq  \lambda \id$, then there exists a SEP such that $\DD(\rho)$ is its diffusion matrix. Moreover, there exists a positive constant $C_{\ref{cor.CovSEP}}(d, \lambda)$ such that for every $F\in L^2(\X,\F,\P_\rho)$, the following estimates hold:
		\begin{equation}\label{eq.EnergyLbarLower}
			\expec{F(-\bar{\L})F}_\rho\geq \frac{1}{8}\sum_{b\in(\Zd)^*}\expec{(\pi_b F)^2}_\rho,
		\end{equation}
		\begin{equation}\label{eq.EnergyLbarUpper}
			\expec{F(-\bar{\L})F}_\rho\leq C_{\ref{cor.CovSEP}}\expec{F(-\L)F}_\rho. 
		\end{equation}
	\end{corollary}
	\begin{proof}
		Let $\cm := 2 \DD(\rho)$ and we pick a symmetric transition matrix $Q$ associated to the covariance matrix $\cm$ in Lemma~\ref{lem.CovRW}, and follow the convention in \eqref{eq.GeneratorEffective} to define
		\begin{align*}
			\Lb = \frac{1}{2}\sum_{x \in \Zd} \sum_{y \in \Zd} Q_{y-x} \pi_{x,y}.
		\end{align*}
		Then we should notice that (see \eqref{eq.defDeltaQ})
		\begin{equation}\label{eq.Delta_Q_Closed}
			\begin{split}
				\Lb \Ll(\sum_{x \in \Zd} f(x) \eta_x\Rr) &= \frac{1}{2} \sum_{x \in \Zd} \sum_{y \in \Zd} Q_{y-x} \Ll(f(y) - f(x)\Rr) (\eta_x - \eta_y) \\
				&= \sum_{x \in \Zd} \sum_{y \in \Zd} Q_{y-x} \Ll(f(y) - f(x)\Rr) \eta_x \\
				&= \sum_{x \in \Zd} \Ll(\frac{1}{2}\Delta_Q f\Rr)(x)\eta_x. 
			\end{split}
		\end{equation}
		The associated diffusion matrix is $\frac{1}{2} \cm = \DD(\rho)$.
		
		For the constructed SEP and its generator $\bar{\L}$, we have an explicit expression for $\expec{F(-\bar{\L})F}_\rho$:
		\begin{equation*}
			\expec{F(-\bar{\L})F}_\rho=\frac{1}{4}\sum_{x\in\Zd}\sum_{y\in\Zd}Q_{y}\expec{(\pi_{x,x+y}F)^2}_\rho.
		\end{equation*}
		\eqref{eq.EnergyLbarLower} then follows from the property $(1)$ in Lemma~\ref{lem.CovRW}.
		For any $y\in\Zd$, we consider a geodesic path in $\ell^1$ distance on lattice $\Zd$ connecting $x$ and $x+y$:
		\begin{equation*}
			x_0=x\rightarrow x_1\rightarrow\cdots\rightarrow x_{|y|_1}=x+y.
		\end{equation*}
		Using Cauchy-Schwarz inequality and the symmetry of measure $\P_\rho$, we can obtain the following moving-particle inequality:
		\begin{equation}\label{eq.Movingparticle}
			\expec{(\pi_{x,x+y}F)^2}_\rho\leq |y|_1^2\sum_{i=1}^{|y|_1}\expec{(\pi_{x_{i-1},x_i}F)^2}_\rho
		\end{equation}
		Combining \eqref{eq.Movingparticle}, the property $(2)$ in Lemma~\ref{lem.CovRW}, and the uniform ellipticity of $c_b$ we obtain \eqref{eq.EnergyLbarUpper}:
		\begin{equation*}
			\expec{F(\bar{\L})F}_\rho\leq \frac{1}{2}C_{\ref{cor.CovSEP}}\sum_{b\in(\Zd)^*}\expec{(\pi_b F)^2}_\rho\leq \frac{1}{2}C_{\ref{cor.CovSEP}}\sum_{b\in(\Zd)^*}\expec{c_b(\eta)(\pi_b F)^2}_\rho=C_{\ref{cor.CovSEP}}\expec{F(-\L)F}_\rho.
		\end{equation*}
	\end{proof}
	\begin{remark}
		The identity $\frac{1}{2} \cm = \DD$ follows the convention in the literature. Take the isotropic Brownian motion of covariance matrix $\cm =\sigma^2 \id$ as an example. Its generator is $\frac{\sigma^2}{2} \Delta$, which corresponds to the diffusion matrix $\DD = \frac{\sigma^2 \id}{2}$. 
	\end{remark}

	\section{Convergence to equilibrium of SEP}\label{sec.SEP}
	In the following sections, we keep the conventions that $Q$ is a symmetric transition matrix defining SEP with the diffusion matrix $\DD(\rho)$ as Corollary~\ref{cor.CovSEP}. We also let $\Lb, \Pb_t, \cm(\rho)$ respectively stand for its generator, semigroup, and covariance matrix. One should keep in mind the relation $\cm(\rho) = 2 \DD(\rho)$.

	We are interested in the long-time behavior of this SEP. The main result in this section is the following statement.
	\begin{proposition}\label{prop.AvePbarDecay}
		For every local function $u$, we have that
		\begin{align}\label{eq.AvePbarDecay}
			\var_{\rho}[\Pb_t u] = \frac{\tilde u' (\rho)^2 \chi(\rho)}{\sqrt{(8\pi t)^{d} \det [\DD(\rho)]}} + O(t^{-\frac{d}{2}-\frac{1}{2}}).
		\end{align}
		Here  $\tilde u' $ follows its definition in Theorem~\ref{thm.main}.
	\end{proposition}

	This section is organized as follows. Section~\ref{subsec.Fock} recalls the facts of the Fock space for SEP, and presents a perturbation formula for the leading order constant. In Section~\ref{subsec.LBarEvolu}, we study the evolution of semigroup $\bar{P}_t$ on the Fock space. In Section~\ref{subsec.DecayPbarFock}, we follow \cite{berzeg} to establish the generalized Nash estimate, and then prove Proposition~\ref{prop.AvePbarDecay}.
	
	\subsection{Wiener--It\^o analysis of Bernoulli variables}\label{subsec.Fock} 
	We follow \cite{Privault2008} and recap some facts about the Fock space of Bernoulli variables. The centered variable $\bar{\eta}_x$ under $(\X, \fil, \Pr)$ are defined as
	\begin{align}\label{eq.defeta_center}
		\forall x \in \Zd, \qquad \bar{\eta}_x:=\eta_x-\rho.
	\end{align}
	For every $\Lambda \subset \Zd$, we denote by  $\K_n(\Lambda)$ and $\K(\Lambda)$ 
	\begin{align}\label{eq.defKn}
		\K_n(\Lambda) := \{Y \subset \Lambda: \vert Y\vert = n\}, \qquad 	\K(\Lambda) := \bigcup_{n=0}^\infty \K_n(\Lambda).
	\end{align}
	We keep $\K_n, \K$ respectively the shorthand notations for $\K_n(\Zd)$ and $\K(\Zd)$.
	
	
	Let $\{e_x\}_{x \in \Zd}$ be the canonical basis of $\X = \{0,1\}^{\Zd}$, where $e_x$ means only the site $x$ is occupied. Then we define $\eta^{x}_{+}, \eta^{x}_{-}$ as
	\begin{align}\label{eq.eta_plus_minus}
		\eta^{x}_{+}:=\eta+(1-\eta_x)e_x, \qquad \eta^{x}_{-}:=\eta-\eta_x e_x.
	\end{align}
	That is to say, $\eta^{x}_{+}$ (\emph{resp.} $\eta^{x}_{-}$) assigns the value $1$ (\emph{resp.} $0$) at $x$. We define a Glauber-type derivative for $F:\X\rightarrow \R$
	\begin{equation}\label{eq.Glauber} 
		D_x F(\eta):=F(\eta^{x}_{+})-F(\eta^{x}_{-}).
	\end{equation}
	For $Y \subset \Zd$, we define the higher-order derivative as
	\begin{equation*}
		D_Y:=\prod_{x\in Y}D_x,
	\end{equation*}
	and we define
	\begin{equation*}
		\bar{\eta}_Y:=\prod_{x\in Y}\bar{\eta}_x.
	\end{equation*}
	We also have the integration by parts formula on Bernoulli variables.
	\begin{lemma}
		The following identity holds for all $Y \in \K$
		\begin{align}\label{eq.IPP_Ber}
			\Er\Ll[F(\eta)\bar{\eta}_{Y}\Rr] = \chi(\rho)^{|Y|}\Er[D_{Y}F(\eta)].
		\end{align}
	\end{lemma}
	\begin{proof}
		Using the definition \eqref{eq.eta_plus_minus}, we have
		\begin{align*}
			\Er[F(\eta)\bar{\eta}_{x}]
			&=\Pr(\eta_x=1)\Er[F(\eta_{+}^{x})(1-\rho)]+\Pr(\eta_x=0)\Er[F(\eta_{-}^{x})(-\rho)]\\
			&=\chi(\rho)\Er[F(\eta_{+}^{x})-F(\eta_{-}^{x})]\\
			&=\chi(\rho)\Er[D_x F(\eta)].
		\end{align*}
		Then we can prove \eqref{eq.IPP_Ber} by recurrence.
	\end{proof}
	For convenience, we define a shorthand notation for the {\rhs} of \eqref{eq.IPP_Ber}
	\begin{equation}\label{eq.DefTn}
		\forall n \in \N_+,  Y\in\K_n, \qquad	T_n F(Y):=\Er[D_Y F],
	\end{equation}
	and we set $T_0 F:=\Er[F]$ by convention.

	The {\lhs} of \eqref{eq.IPP_Ber} is just the inner product between the cylinder functions, so \eqref{eq.IPP_Ber} allows us to describe better the space $L^2(\X,\fil,P_\rho)$. We define at first the discrete multiple stochastic integral $I_n$ as
	\begin{equation}\label{eq.DefIn}
		\forall n \geq 1, \qquad I_n(f_n):=\sum_{Y\in\K_n}f_n(Y)\bar{\eta}_Y.
	\end{equation}
	It is well-defined for $f_n \in \ell^1(K_n)$, but then can be extended to $f_n \in \ell^2(K_n)$ viewing the following orthogonal property
	\begin{equation}\label{eq.Ortho}
		\forall Y_1,Y_2\in\K, \qquad \Er[\bar{\eta}_{Y_1}\bar{\eta}_{Y_2}]= \chi(\rho)^{|Y_1|}\1_{\{Y_1=Y_2\}},
	\end{equation}
	and the isometric property
	\begin{equation}\label{eq.ItoIso}
		\Er[I_n(f_n)^2]= \chi(\rho)^n \norm{f_n}^2_{\ell^2(\K_n)}.
	\end{equation}
	Afterwards, we can define the Fock space of order $n$ 
	\begin{equation*}
		\forall n \in \N_+, \qquad \HH_n:=\{I_n(f_n):f_n\in\ell^2(\K_n)\},
	\end{equation*}
	and we keep the convention $\HH_0:=\R$.
	
	
		

	Let $\mcl{S}$ denote the linear space spanned by the multiple stochastic integrals:
	\begin{equation*}
		\mcl{S}:=\text{vect}\Ll\{\bigcup_{n=0}^{\infty}\mcl{H}_n\Rr\}=\left\{\sum_{k=0}^{n}I_k(f_k):f_k\in\ell^2(\K_k),n\in\N\right\},
	\end{equation*}
	and we denote by $\bigoplus_{n=0}^{\infty}\HH_n$ the completion of $\mcl{S}$ in $L^2(\X,\fil,P_\rho)$.


	\smallskip

	One important property is the chaos expansion.

	\begin{lemma}[Chaos expansion]\label{lem.Chaos}
		The following identity holds
		\begin{align}\label{eq.Fock2}
			L^2(\X, \fil, \Pr) = \bigoplus_{n=0}^{\infty}\HH_n.
		\end{align}
		More precisely, for every $F\in L^{2}(\X,\fil,\P_\rho)$, it can be decomposed as
		\begin{equation}\label{eq.chaos2} 
			F(\eta)=\sum_{n=0}^{\infty}I_n(T_n F),
		\end{equation}
		where the equality makes sense in $L^{2}(\X,\fil,\P_{\rho})$.
	\end{lemma}
	
	\begin{proof}
		The proof can be divided into 2 steps and can be found in \cite[Proposition~6.7]{Privault2008}.
		
		\textit{Step~1: orthogonal projection.} It is clear that the {\rhs} of \eqref{eq.Fock2} is a sum of orthogonal subspaces thanks to \eqref{eq.Ortho}. The inclusion
		\begin{align*}
			L^2(\X, \fil, \Pr) \supset \bigoplus_{n=0}^{\infty}\HH_n,
		\end{align*} 
		is obvious, and it suffices to verify the inclusion of another direction. Recall that $\Lambda_n$ is of side length $(2n+1)$, then the projection
		\begin{align*}
			F_n := \Er[F \vert \fil_{\Lambda_n}],
		\end{align*}
		forms a closed martingale and we have $\lim_{n \to \infty} F_n = F$ in $L^2$. Meanwhile, it is clear that 
		\begin{align*}
			L^2(\X, \fil_{\Lambda_n}, \Pr) \subset \bigoplus_{k=0}^{(2n+1)^d}\HH_k,
		\end{align*}
		because the {\lhs} is a finite dimensional linear space. This concludes \eqref{eq.Fock2}.

		\textit{Step~2: identification of coefficients.} For every $Y\in\K$, the projection on $\bar{\eta}_{Y}$ is
		\begin{equation*}
			\frac{\Er[F(\eta)\bar{\eta}_{Y}]}{\Er[(\bar{\eta}_{Y})^2]}=\frac{\Er[F(\eta)\bar{\eta}_{Y}]}{\chi(\rho)^{|Y|}} = \frac{\chi(\rho)^{|Y|}\Er[D_{Y}F(\eta)]}{\chi(\rho)^{|Y|}} = T_{|Y|} F(Y).
		\end{equation*}
		The first equality comes from the product Bernoulli measure, and the second equality comes from \eqref{eq.IPP_Ber}. Then we use the definition \eqref{eq.DefTn} for the third equality and conclude \eqref{eq.chaos2}.	
	\end{proof}
	
	Now we present a perturbation formula.
	\begin{lemma}[Perturbation formula]\label{lem.PertFormu}
		Given a local function $F$, we have the identity
		\begin{equation*}
			\frac{\d^n}{\d\rho^n}\Er[F]=n!\sum_{Y\in\K_n}T_n F(Y).\label{eq.PertFormu}
		\end{equation*}
	\end{lemma}
	
	\begin{proof}
		Let $\boldsymbol{\rho}\in(0,1)^{\otimes\Zd}$ be a generalized density on every site, and we define a measure $\P_{\boldsymbol{\rho}}:=\otimes_{x\in\Zd} \operatorname{Bernoulli}(\boldsymbol{\rho}_x)$ as the product Bernoulli measure. For every function $F:\X\rightarrow\R$, we can define its expectation function $	E(F;\boldsymbol{\rho})$ by
		\begin{equation*}
			E(F;\boldsymbol{\rho})=\int_{\X}F(\eta)\d\P_{\boldsymbol{\rho}}.
		\end{equation*}
		Recall that $\{e_x\}_{x \in \Zd}$ is the canonical basis of $\X = \{0,1\}^{\Zd}$. Because $\boldsymbol{\rho}+te_x$ only modifies the density at $x$, for $t$ small enough, we still have $\boldsymbol{\rho}+te_x\in(0,1)^{\otimes\Zd}$. We can consider the derivative of $E(F;\boldsymbol{\rho})$ along the direction $e_x$ with $x\in\Zd$.
		\begin{align*}
			E(F;\boldsymbol{\rho}+te_x)
			&=\P_{\boldsymbol{\rho}+te_x}(\eta_x=1)\E_{\P_{\boldsymbol{\rho}}}[F(\eta_{+}^{x})]+\P_{\boldsymbol{\rho}+te_x}(\eta_x=0)\E_{\P_{\boldsymbol{\rho}}}[F(\eta_{-}^{x})]\\
			&=(\boldsymbol{\rho}_x+t)\E_{\P_{\boldsymbol{\rho}}}[F(\eta_{+}^{x})]+(1-\boldsymbol{\rho}_x-t)\E_{\P_{\boldsymbol{\rho}}}[F(\eta_{-}^{x})]\\
			&=E(F;\boldsymbol{\rho})+t(\E_{\P_{\boldsymbol{\rho}}}[F(\eta_{+}^{x})]-\E_{\P_{\boldsymbol{\rho}}}[F(\eta_{-}^{x})])\\
			&=E(F;\boldsymbol{\rho})+t\E_{\P_{\boldsymbol{\rho}}}[D_x F].
		\end{align*}
		Sending $t\rightarrow0$, we get
		\begin{equation*}
			\frac{\partial}{\partial \boldsymbol{\rho}_x} E(F;\boldsymbol{\rho})=\E_{\P_{\boldsymbol{\rho}}}[D_x F].
		\end{equation*}
		
		Given a local function $F$ which is $\fil_{\Lambda_L}$-measurable, then we can study $\frac{\d}{\d\rho}\Er[F]$ using the chain rule with the reference measure $\boldsymbol{\rho}_x \equiv \rho$ for all $x \in \Zd$
		\begin{equation*}
			\frac{\d}{\d\rho}\Er[F]=\sum_{x\in\La_L}\frac{\partial}{\partial \boldsymbol{\rho}_x} E(F;\boldsymbol{\rho})=\sum_{x\in\La_L}\Er[D_x F]=\sum_{x\in\Zd}\Er[D_x F],
		\end{equation*}
		since $D_x F = 0$ for all $x\notin\La_L$.

		Iterating this process, we derive the following equality by induction:
		\begin{equation*}
			\frac{\d^n}{\d\rho^n}\Er[F]=\sum_{(x_1,\cdots,x_n)\in(\Zd)^{\otimes n}}\Er[D_{x_1}\cdots D_{x_n}F].
		\end{equation*}
		We can always exchange derivative and summation because there are at most $|\La_L|^n$ nonzero terms in the $n$-th order derivative. 
		
		Notice the fact $D_{x}D_{x} F\equiv0$ for all $x\in\Zd$, the expression above can be simplified as
		\begin{equation*}
			\frac{\d^n}{\d\rho^n}\Er[F]=\sum_{(x_1,\cdots,x_n)\in(\Zd)^{\otimes n}}\Er[D_{x_1}\cdots D_{x_n}F]=n!\sum_{Y\in\K_n}T_n F(Y).
		\end{equation*}
		This concludes the desired result.
	\end{proof}

	\subsection{Kolmogorov equation in Fock space}\label{subsec.LBarEvolu}
	In this subsection, we study the evolution of semigroup on the Fock space. 
	\begin{lemma}\label{lem.eqKolmogrov}
		The following equation holds for all $f_n\in\ell^2(\K_n)$
		\begin{align*}
			\Lb(I_n(f_n)) = I_n\left(\frac{1}{2}\Delta^{(n)}_Qf_n\right),
		\end{align*}
		where the discrete Laplace operator is defined as 
		\begin{equation}\label{eq.defDelta_n}
			\forall Y\in\K_n, \qquad \Ll[\frac{1}{2}\Delta^{(n)}_{Q}f_n\Rr](Y):=\sum_{x\in Y,y\in \Zd \backslash Y}Q_{y-x}(f_n(Y\cup\{y\}\backslash\{x\})-f_n(Y)).
		\end{equation}
		Moreover, we have $\frac{1}{2}\Delta^{(n)}_Q f_n\in\ell^2(\K_n)$ and $\Lb(I_n(f_n))\in\HH_n$.
	\end{lemma}
	\begin{proof}
		We fix $Y\in\K_n$ and consider one term $\pi_{x,y}f_n(Y)\bar{\eta}_{Y}$. This term will be canceled if $x,y\in Y$ or $x,y \notin Y$, so we only need to consider the case $\{x\in Y,y\notin Y\}$ and $\{x\notin Y,y\in Y\}$. We take the former as an example, then $Y$ can be written as $Y=Y_{n-1} \cup \{x\}$ with $Y_{n-1} \in \K_{n-1}(\Zd \setminus \{x,y\})$, and we obtain 
		\begin{multline*}
			\pi_{x,y}I_n(f_n)=\\
			\sum_{Y_{n-1}\in\K_{n-1}(\Zd\backslash\{x,y\})}\Ll(f_n(Y_{n-1}\cup\{y\}-f_n(Y_{n-1}\cup\{x\}))(\bar{\eta}_{Y_{n-1}\cup\{x\}}-\bar{\eta}_{Y_{n-1}\cup\{y\}}\Rr).
		\end{multline*}
		Since $\Lb=\frac{1}{2}\sum_{x\in\Zd}\sum_{y\in\Zd}Q_{y-x}\pi_{x,y}$, we can calculate $\Lb(I_n(f_n))$ by linearity, with the coefficient of the term $\bar{\eta}_{Y}$:
		\begin{multline*}
			\frac{1}{2}\sum_{x\in Y,y\in\Zd\backslash Y}Q_{y-x}(f_n(Y\cup\{y\}\backslash\{x\})-f_n(Y))\\
			-\frac{1}{2}\sum_{y\in Y,x\in\Zd\backslash Y}Q_{y-x}(f_n(Y)-f_n(Y\cup\{x\}\backslash\{y\})).
		\end{multline*}
		The two parts in the above expression are the same because of the symmetry  $Q_{y-x}=Q_{x-y}$. Therefore, we have
		\begin{equation*}
			\Lb(I_n(f_n))=\sum_{Y\in\K_n}\sum_{x\in Y,y\in \Zd \backslash Y}Q_{y-x}(f_n(Y\cup\{y\}\backslash\{x\})-f_n(Y))\bar{\eta}_Y=I_n\Ll(\frac{1}{2}\Delta^{(n)}_Q f_n\Rr).
		\end{equation*}
		This yields \eqref{eq.defDelta_n}. Since the support of $Q$ is a finite set, we can show 
		\begin{equation*}
			\norm{\frac{1}{2}\Delta^{(n)}_{Q}f_n}_{\ell^2(\K_n)}\leq C ||f_n||_{\ell^2(\K_n)},
		\end{equation*}
		where $C$ is a constant depending on $Q$. Using the estimate above, we have $\frac{1}{2}\Delta^{(n)}_Q f_n\in\ell^2(\K_n)$ and $\Lb(I_n(f_n))\in\HH_n$.
	\end{proof}

	\subsection{Faster heat kernel decay}\label{subsec.DecayPbarFock}
	In this subsection, we recall the generalized Nash estimate  and  the variance decay in \cite{berzeg}.
	
	For $n\in\N$, we consider the following seminorm for $F\in L^2(\X,\fil,\P_{\rho})$:
	\begin{equation}\label{eq.defTripleNorm}
		|||F|||_n:=\sum_{Y\in\K_n}||D_Y F||_{L^\infty}.
	\end{equation}
	One can identify the seminorm $|||\cdot|||_n$ as $\mathbf{V}_{\infty,1}^n$ in \cite[(5.2)]{berzeg}. 	For an element $I_n(f_n)$ in $\HH_n$, its seminorm $|||I_n(f_n)|||_n$ can be computed as:
	\begin{equation*}
		|||I_n(f_n)|||_n=\sum_{Y\in\K_n}|f_n(Y)|=\norm{f_n}_{\ell^1(\K_n)}\geq \norm{f_n}_{\ell^2(\K_n)}.
	\end{equation*}
	We denote $\tilde{\HH}_n$ the collection of all elements in $\HH_n$ with finite extended norm $|||\cdot|||_n$:
	\begin{equation}\label{eq.DeftildeHn}
		\tilde{\HH}_n:=\{I_n(f_n): \norm{f_n}_{\ell^1(\K_n)}<\infty\}\subset\HH_n.
	\end{equation}

	One major achievement in \cite{berzeg} is the following generalized Nash inequality. We reformulate its proof in Section~\ref{sec.Nash} for convenience. In the statement, the Dirichlet form associated to SEP is given by
	\begin{equation*}
		\bar{\sE}_{\rho}(F):= \bracket{F(-\Lb F)}_{\rho}. 
	\end{equation*}
	
	\begin{proposition}[Generalized Nash inequality]\label{prop.Nash}
		For every $F\in\tilde{\HH}_n$, the following estimate holds
		\begin{equation*}
			\bracket{F^2}_{\rho}\leq C\chi(\rho)^{n(1-\alpha_n)}\bar{\sE}_\rho(F)^{\alpha_n}|||F|||^{2(1-\alpha_n)}_n,
		\end{equation*}
		where $\alpha_n :=\frac{nd}{2+nd}$ and the constant $C$ only depends on $n$ and $d$.
	\end{proposition}
	
	The generalized Nash inequality has some extensions; see for example \cite[Appendix~A, B]{jlqy}. However, the following property, proved in \cite[Theorem 10]{berzeg}, is specific for SEP, (see also \cite[Chapter 8]{liggett2012interacting})
	\begin{equation}\label{eq.ContraPbart}
		\forall n \in \N, \qquad |||\bar{P}_t F|||_n\leq |||F|||_n .
	\end{equation}	
	
	Combing both Proposition~\ref{prop.Nash} and \eqref{eq.ContraPbart}, we derive the faster decay result \cite[Theorem~17]{berzeg}.
	\begin{proposition}\label{prop.fasterDecay}
		For every $n\in\N_+$, there exists a constant $C$ only depending on $n,d$ and $\rho$ such that the following estimate holds for all $F\in\tilde{\HH}_n$:  
		\begin{align}\label{eq.fasterDecay}
			\bracket{(\Pb_t F)^2}_\rho \leq C t^{-\frac{nd}{2}} |||F|||^2_n.
		\end{align}
	\end{proposition}
	
	\begin{proof}
		For any $F\in\tilde{\HH}_n$, we have
		\begin{align*}
			\partial_t\expec{(\bar{P}_t F)^2}_\rho
			&=-2 \bar{\sE}_\rho(\bar{P}_tF)\\
			&\leq -C\expec{(\bar{P}_t F)^2}^{\frac{2+nd}{nd}}_\rho|||\bar{P}_t F|||^{-\frac{4}{nd}}_n\\
			&\leq -C\expec{(\bar{P}_t F)^2}^{\frac{2+nd}{nd}}_\rho|||F|||^{-\frac{4}{nd}}_n.
		\end{align*}
		The passage from the first line to the second line comes from Proposition~\ref{prop.Nash}, and we apply \eqref{eq.ContraPbart} from the second line to the third line. This implies
		\begin{equation}\label{eq.PbarDecay}
			\partial_t\Ll(\expec{(\bar{P}_t F)^2}_\rho^{-\frac{2}{nd}}\Rr)\geq C|||F|||_n^{-\frac{4}{nd}}.
		\end{equation}
		We can obtain the desired result by integrating  \eqref{eq.PbarDecay}.
	\end{proof}
	
	
	Now we can prove the main result in this section.
	\begin{proof}[Proof of Proposition~\ref{prop.AvePbarDecay}]
		The proof can be structured in three steps.
		
		\textit{Step~1: chaos expansion.}
		Since $u$ is a local function, there exists $N\in\N$ such that 
		\begin{equation*}
			u\in\bigoplus_{n=0}^{N}\HH_n.
		\end{equation*}
		Using the chaos expansion in Lemma~$\ref{lem.Chaos}$, we have the following identity
		\begin{equation*}
			u=\sum_{n=0}^{N}I_n(T_nu),
		\end{equation*}
		where $I_n(T_nu)\in\HH_n$ is the projection of $u$ on $\HH_n$. Since $u$ is a local function, we have $I_n(T_n u)\in\tilde{\HH}_n$. By the linearity of the semigroup $\bar{P}_t$ we obtain
		\begin{equation*}
			\bar{P}_tu=\sum_{n=0}^{N}\bar{P}_t I_n(T_nu).
		\end{equation*}
		Because of Lemma~\ref{lem.eqKolmogrov}, $\bar{P}_t I_n(T_nu)\in\HH_n$ and we have
		\begin{equation*}
			\bracket{\bar{P}_tu}_\rho =\bar{P}_t I_0(T_0u).
		\end{equation*}
		Using the orthogonal decomposition over $(\HH_n)_{n \in \N}$,  we have
		\begin{equation}\label{eq.decomL2Pbt'}
			\var_\rho[\bar{P}_tu]= \bracket{ \left(\sum_{n=1}^{N}\bar{P}_t I_n(T_nu)\right)^2}_\rho =\sum_{n=1}^{N} \bracket{\left(\bar{P}_t I_n(T_nu)\right)^2}_\rho.
		\end{equation}
		\smallskip
		
		\textit{Step~2: faster decay of higher-order terms.}
		Proposition~\ref{prop.fasterDecay} implies the faster decay of the heat kernel in higher dimension :
		\begin{align}\label{eq.fasterL2Pbt'}
			\forall n \geq 2, \qquad \bracket{\left(\bar{P}_t I_n (T_nu)\right)^2}_\rho \leq  Ct^{-\frac{nd}{2}}|||I_n(T_nu)|||^2_{n}.
		\end{align}
		Thus the main contribution for very large $t$ in \eqref{eq.decomL2Pbt'} is the case $n=1$.
		\smallskip
		
		\textit{Step~3: identification of the leading order.}
		For the case $n=1$ in \eqref{eq.decomL2Pbt'}, Lemma~\ref{lem.eqKolmogrov} suggests a function $f_t \in \ell^2(\Zd)$ such that  
		\begin{align*}
			\bar{P}_t I_1 (T_1u) = I_1(f_t).
		\end{align*}
		Furthermore, the calculation in \eqref{eq.Delta_Q_Closed} yields the following equation
		\begin{align*}
			I_1(\partial_t f_t) = \partial_t \bar{P}_t I_1 (T_1u) = \Lb \bar{P}_t I_1 (T_1u) = \Lb I_t(f_t) = I_1\Ll(\frac{1}{2}\Delta_Q f_t\Rr).
		\end{align*}
		Therefore, $\partial_t f_t = \frac{1}{2}\Delta_Q f_t$ with $f_0 = T_1 u$, 	which implies the explicit solution
		\begin{equation*}
			\bar{P}_t I_1(T_1u)=I_1(e^{t\Delta_Q/2}T_1u)=I_1(\bar{p}_t\ast(T_1u)).
		\end{equation*}
		Here the function $\bar{p}_t = e^{t\Delta_Q/2}$ and $\ast$ is the discrete convolution defined in \eqref{eq.defConvolution}.
		The isometric property \eqref{eq.ItoIso} then yields
		\begin{align}\label{eq.ItoL2Pbt'}
			\bracket{\left(\bar{P}_t I_1(T_1 u)\right)^2}_\rho = \bracket{\left(I_1(\bar{p}_t\ast(T_1u))\right)^2}_\rho = \chi(\rho) \norm{\bar{p}_t\ast(T_1u)}_{\ell^2(\Zd)}^2.
		\end{align}
		For the last term about the discrete heat kernel, Lemma~\ref{lem.NashZd} implies
		\begin{align}\label{eq.Order1L2Pbt'}
			\norm{\bar{p}_t\ast(T_1u)}_{\ell^2(\Zd)}^2 = \frac{(\sum_{x \in \Zd} T_1u(x) )^2}{\sqrt{(4\pi t)^d \det [\cm(\rho)]}}+ O( t^{-\frac{d}{2}-\frac{1}{2}}).
		\end{align}
		Since $u$ is a local function, we can show that $T_1 u$ is also a local function. Then from the perturbation formula in Lemma~\ref{lem.PertFormu}, we have
		\begin{equation}\label{eq.coeffL2Pbt'}
			\sum_{x\in\Zd}T_1u(x)=\frac{\d}{\d\rho} \E_\rho[u] =\tilde{u}'(\rho).
		\end{equation}
		Here we use the notation $\tilde{u}(\rho) = \bracket{u}_\rho$ defined in Theorem~\ref{thm.main}.
		
		Combing \eqref{eq.ItoL2Pbt'}, \eqref{eq.Order1L2Pbt'} and \eqref{eq.coeffL2Pbt'} and the definition $\cm(\rho) = 2 \DD(\rho)$, we conclude
		\begin{align}\label{eq.leadingL2Pbt'}
			\bracket{\left(\bar{P}_t I_1(T_1u)\right)^2}_\rho =  \frac{\tilde{u}'(\rho)^2  \chi(\rho)}{\sqrt{(8\pi t)^d \det [\DD(\rho)]}}+ O( t^{-\frac{d}{2}-\frac{1}{2}}).
		\end{align}
		The estimates \eqref{eq.decomL2Pbt'}, \eqref{eq.fasterL2Pbt'} and \eqref{eq.leadingL2Pbt'} complete the proof of Proposition~\ref{prop.AvePbarDecay}.
	\end{proof}

	\section{Homogenization of semigroup}\label{sec.HomoLinear}
	We are interested in the approximation between $P_t$ and $\bar{P}_t$. This is a classical topic in homogenization theory, and we develop its counterpart now in exclusion process. The main result in this section is the following statement.
	\begin{proposition}\label{prop.MixL2HkEs}
		There exist two finite positive constants $C(d,\lambda, \r,\rho),\beta(d,\lambda, \r)$, such that for every $F\in \oplus_{k=1}^{N}\tilde{\HH}_k$ with $N \in \N$, the following estimate holds for all $t > 0$:
		\begin{equation}\label{eq.MixL2HkEs}
			\norm{(P_t-\bar{P}_t)F}^2_{L^2} \leq C\Ll(t^{-2\beta} \norm{F}^2_{L^2}+\sum_{k=2}^{N}t^{-\frac{3kd}{8}}||| I_k(T_k F) |||^2_k\Rr).
		\end{equation}
	\end{proposition}
	\begin{remark}
		A similar estimate for particles in continuum configuration space was developed in \cite[Proposition~4.1]{gu2024quantitative}, but \eqref{prop.MixL2HkEs} is stronger as it covers more functions.
	\end{remark}

	The proof makes use of \emph{the two-scale expansion ansatz}, which shares the same spirit of the \emph{gradient replacement} in Varadhan's argument for non-gradient models. Our proof follows the work \cite{zhikov2006estimates} and \cite{gu2024quantitative}, and the Wiener--It\^o analysis in Section~\ref{sec.SEP} also involves.  We will recall some basic results about the correctors in Section~\ref{subsec.corrector}, then present the proof in Section~\ref{subsec.twoscale} and ~\ref{subsec.reg}.

	\medskip
	
	The following notations will only be used in this section. The domain with various boundary conditions $\Lambda^-, \overline{\Lambda^*}, \Lambda^+$ are defined as follows. We define $\partial \Lambda$ as the boundary set of $\Lambda$, and denote by $\Lambda^-$ its interior that 
	\begin{align}\label{eq.defBoundary}
		\partial \Lambda := \{ x \in \Lambda : \exists y \notin \Lambda, y \sim x\}, \qquad \Lambda^- := \Lambda \setminus \partial \Lambda.
	\end{align}
	Recall that the bonds set of $\Lambda$ is defined as $\Lambda^*$ in \eqref{eq.defBond}. We  define its enlarged version 
	\begin{align}\label{eq.defBondLarge}
		\overline{\Lambda^*} := \{\{x,y\}: x \in \Lambda, y = x+e_i, i=1,2, \cdots, d\},
	\end{align}
	where $e_i \in \Zd$ is the $i$-th directed unit vector. We also denote by $\Lambda^+$ the vertices concerned in \eqref{eq.defBondLarge}
	\begin{align} \label{eq.defVertexLarge}
		\Lambda^+:=\Lambda\cup\bigcup_{i=1}^d(\Lambda+e_i).
	\end{align}

	For every $\Lambda \subset \Zd$ and $\xi \in \Rd$, the affine function defined as follows will be used 
	\begin{align}\label{eq.defAffine}
		\ell_{\xi,\Lambda} := \sum_{x \in \Lambda} (\xi \cdot x) \eta_x .
	\end{align}
	
	For all $m \in \N$, let $\cu_m$ stand for the lattice cube of side length $3^m$ 
	\begin{align}\label{eq.defcu}
		\cu_m := \Ll(- \frac{3^m}{2}, \frac{3^m}{2}\Rr)^d \bigcap \Zd.
	\end{align}
	We also use the notation $\Z_m := 3^m \Zd$ as the centers of triadic cubes.
	
	\subsection{Local corrector and centered flux}\label{subsec.corrector}
	To understand how $P_t$ is close to $\Pb_t$, a similar but more fundamental question is how the local functions in  \eqref{eq.defQuadra} can approximate the diffusion matrix $\DD(\rho)$ and the conductivity $\c(\rho)$; see their definitions in \eqref{eq.Einstein} and \eqref{eq.defC}. The previous work \cite{funaki2024quantitative} was devoted to this question, and its main result (\cite[Theorem~1.3]{funaki2024quantitative}) confirms a quantitative version of \eqref{eq.defQuadra} that 
	\begin{align}\label{eq.thm_main_FGW24}
		\inf_{F_L \in \F^d_0(\Lambda_L)} \sup_{\rho \in [0,1]} 	\vert \c(\rho; F_L) - \c(\rho) \vert \leq C L^{-\gamma}.
	\end{align}
	The finite positive constants $C, \gamma$ here only depend on $d, \r, \lambda$. The object of this part is to extract the key ingredients in \eqref{eq.thm_main_FGW24} to study $P_t$ and $\Pb_t$.

	The key ingredient is the function to minimize \eqref{eq.thm_main_FGW24}. Precisely, in \cite[(1.20),(4.1)]{funaki2024quantitative}, for every vector $p \in \Rd$, we introduce a variational formula
	\begin{equation}\label{eq.defnu}
		\begin{split}
			\nub(\rho,\cu_m,p) &:= \inf_{v\in\ell_{p,\cu_m^+} +\F_0(\cu_m^-)} \Ll\{ \frac{1}{2 \chi(\rho)\vert\cu_m\vert} \sum_{b\in\ov{\cu_m^*}} \bracket{ \frac{1}{2}c_b(\pi_b v)^2}_\rho \Rr\},\\
			&= \frac{1}{2} p \cdot \DDb(\rho, \cu_m) p.
		\end{split}
	\end{equation}
	Based on the approximated diffusion matrix $\DDb(\rho,\cu_m)$, we define the approximated conductivity as
	\begin{align}\label{eq.defconduct_local}
		\cc(\rho,\cu_m):=2\chi(\rho)\DDb(\rho,\cu_m).
	\end{align}
	The convergences $\DDb(\rho,\cu_m) \xrightarrow{m \to \infty} \DD(\rho)$ and $\cc(\rho,\cu_m) \xrightarrow{m \to \infty} \c(\rho)$ are proved in \cite[Proposition~5.1]{funaki2024quantitative}. They are the basis to study \eqref{eq.thm_main_FGW24}.
	
	Concerning the minimizer, we denote by $v(\cdot,\rho,\cu_m,p)$ the unique minimizer of $\nub(\rho,\cu_m,p)$ satisfying $\bracket{v-\ell_{p,\cu_m^+}}_\rho = 0$. If $c_b \equiv 1$ for all $b \in (\Zd)^*$, i.e. the exclusion process is SSEP, then we observe that ${v(\cdot,\rho,\cu_m,p) =\ell_{p,\cu_m^+}}$. Therefore, the local function part $\F_0(\cu_m^-)$ in \eqref{eq.defnu} aims to reduce the perturbation from the rate $c_b$.  In the homogenization theory, this part is called corrector and is usually much smaller compared to the affine part. In our context, the rigorous definition of \emph{the local corrector} is that
	\begin{equation}\label{eq.defCorrector}
		\begin{split}
			\phi^0_{m, e_i} &:= v(\cdot,\rho,\cu_m,e_i) - \ell_{e_i,\cu_m^+}, \\
			\phi^z_{m, e_i} &:= \tau_z  \phi^0_{m, e_i}.
		\end{split}
	\end{equation}
	The local corrector then defines \emph{the centered flux} 
	\begin{align}\label{eq.defFlux}
		\g_{m,e_i, b}^z := c_b  \pi_b (\ell_{e_i} + \phi^z_{m, e_i} ) - \pi_b \ell_{\DD(\rho) e_i},
	\end{align}
	and $\g_{m,e_i, b}^z$ can create spatial cancellation, which plays a similar role as Varadhan's gradient replacement (see \cite[Chapter~7.1]{kipnis1998scaling}). 
	
	Viewing the discussions above, we summarize some important properties about $\phi^z_{m, e_i}$ and $\g_{m,e_i, b}^z$ from \cite{funaki2024quantitative}.
	\begin{proposition}\label{prop.phiz}
		We have the following properties for $\phi^z_{m, e_i}$ and $\g_{m,e_i, b}^z$.
		\begin{enumerate}
			\item \emph{(Locality)}: the local corrector $\phi^z_{m, e_i}$ is a local function and  $\fil_{z+\cu_m^-}$-measurable. 
			\item \emph{(Mean)}: the local corrector $\phi^z_{m, e_i}$ is centered  
			\begin{align}\label{eq.phiz_mean}
				\bracket{ \phi^z_{m, e_i}}_{\rho, z+\cu_m} = 0.
			\end{align}
			\item \emph{($L^2$-bound)}: the local corrector $\phi^z_{m, e_i}$ satisfies a uniform bound: for all ${\rho \in (0,1)}, m \in \N$ and $z \in \Zd$, we have 
			\begin{align}\label{eq.L2phiz}
				\bracket{\Ll(\phi^z_{m, e_i}\Rr)^2 }_{\rho} \leq 16 \lambda \chi(\rho) 3^{(d+2)m}.
			\end{align}
			\item \emph{(Spatial cancellation)}: there exists an exponent $\alpha(d,\lambda, \r) > 0$ and a positive constant $C(d, \lambda, \r) < \infty$, such that for every $v : \X \to \R$, we have 
			\begin{align}\label{eq.FluxCancel}
				\Ll\vert \frac{1}{\vert \cu_m \vert} \sum_{b \in \ov{(z+\cu_m)^*} }\bracket{(\pi_b v) \g_{m,e_i, b}^z}_{\rho} \Rr\vert \leq C 3^{-\alpha m}\Ll(\frac{1}{\vert \cu_m \vert} \sum_{b \in \ov{(z+\cu_m)^*} }\bracket{(\pi_b v)^2}_{\rho}\Rr)^{\frac{1}{2}}.
			\end{align}
		\end{enumerate}
	\end{proposition}
	\begin{proof}
		The properties (1) and (2) just follow the definition of $\phi^z_{m, e_i}$. Concerning the $L^2$ moment in (3), we utilize the bound of $\nub(\rho,\cu_m,e_i) \leq 2 \lambda$ (see \cite[(4.2)]{funaki2024quantitative}), which implies that
		\begin{align*}
			\sum_{b\in\ov{\cu_m^*}} \bracket{ (\pi_b \phi^0_{m, e_i})^2}_\rho &\leq \sum_{b\in\ov{\cu_m^*}}\bracket{ c_b(\pi_b \phi^0_{m, e_i})^2}_\rho \\
			& \leq 2\Ll( 2\chi(\rho) \vert \cu_m\vert \nub(\rho,\cu_m,e_i) + \sum_{b\in\ov{\cu_m^*}}\bracket{ c_b(\pi_b \ell_{e_i})^2}_\rho \Rr) \\
			& \leq 16 \lambda \chi(\rho) 3^{dm}.
		\end{align*}
		Here the uniform ellipticity of $c_b$ is used in the first line. Then we apply the spectral gap inequality in \cite[Lemma~2.4]{funaki2024quantitative}, and obtain
		\begin{align*}
			\bracket{\Ll(\phi^0_{m, e_i}\Rr)^2 }_{\rho}  \leq 3^{2m}  \sum_{b\in\ov{\cu_m^*}} \bracket{ (\pi_b \phi^0_{m, e_i})^2}_\rho \leq 16 \lambda \chi(\rho) 3^{(d+2)m}.
		\end{align*}
		
		\smallskip
		
		The quantitative convergence in (4) is non-trivial, since it is deduced from the intermediate steps of the proof to \eqref{eq.thm_main_FGW24}. In previous work \cite{funaki2024quantitative}, we proposed a dual quantity
		\begin{align*}
			\nub_*(\rho,\cu_m,q) &:=\sup_{v \in  \F_0} \Ll\{ \frac{1}{2 \chi(\rho)\vert \cu_m \vert}\sum_{b\in\ov{\cu_m^*}}  \bracket{ (\pi_b \ell_{q})(\pi_b v) - \frac{1}{2} c_b(\pi_b v)^2}_{\rho}\Rr\} \\
			&= \frac{1}{2} q \cdot \DDb_*^{-1}(\rho,\cu_m)q.
		\end{align*}
		Because $q \mapsto \nub_*(\rho, \cu_m, q)$ is quadratic, we have the expression in the second line above, where $\DDb_*^{-1}(\rho,\cu_m)$ is a positive definite matrix. We denote by $u(\cu_m, q)$ the maximizer of $\nub_*(\rho, \cu_m,q)$. Then for every function $v : \X \to \R$, the variational formula yields
		\begin{align}\label{eq.u_variational}
			\sum_{b \in \ov{\cu_m^*} }\bracket{c_b(\pi_b u(\cu_m, q))(\pi_b v)}_\rho = \sum_{b \in \ov{\cu_m^*} }\bracket{(\pi_b\ell_{q})(\pi_b v)}_\rho.
		\end{align} 
		Denoting by
		\begin{align*}
			\DDb_m(\rho) := \DDb_*(\rho,\cu_m),
		\end{align*}
		and admitting a heuristic
		\begin{align}\label{eq.uv_heuristic}
			u(\cu_m, \DDb_m(\rho)e_i) \simeq  \ell_{e_i,\cu_m^+} + \phi^0_{m, e_i}, \qquad \DDb_m(\rho) \simeq \DD(\rho).
		\end{align}
		then  \eqref{eq.u_variational} is nearly as the {\lhs} of \eqref{eq.FluxCancel} when $q = \DDb_m(\rho)e_i$. 
		
		Hence, we have the following decomposition
		\begin{align*}
			\frac{1}{\vert \cu_m \vert} \sum_{b \in \ov{\cu_m^*} }\bracket{\g_{m,e_i, b}^0 (\pi_b v) }_{\rho}  &= \mathbf{I} + \mathbf{II} + \mathbf{III},
		\end{align*}
		with the 3 terms defined as follows. We use the shorthand notation $u_m \equiv u(\cu_m, \DDb_m(\rho)e_i)$. 
		\begin{align*}
			\mathbf{I} &:= \frac{1}{\vert \cu_m \vert} \sum_{b \in \ov{\cu_m^*} }\bracket{c_b\pi_b (\ell_{e_i} + \phi^0_{m, e_i} - u_m)(\pi_b v) }_{\rho},\\
			\mathbf{II} &:= \frac{1}{\vert \cu_m \vert} \sum_{b \in \ov{\cu_m^*} }\bracket{\Ll(c_b\pi_b u_m - \pi_b \ell_{\DDb_m(\rho) e_i}\Rr)(\pi_b v) }_{\rho},\\
			\mathbf{III} &:= \frac{1}{\vert \cu_m \vert} \sum_{b \in \ov{\cu_m^*} }\bracket{\Ll(\pi_b \ell_{\DDb_m(\rho) e_i} - \pi_b \ell_{\DD(\rho) e_i}\Rr)(\pi_b v) }_{\rho}.\\
		\end{align*}
		Here $\mathbf{II} = 0$ thanks to \eqref{eq.u_variational} with $q = \DDb_m(\rho) e_i$. Thus we only need to treat $\mathbf{I}$ and $\mathbf{III}$, and use Cauchy--Schwarz inequality to obtain a bound that
		\begin{multline}\label{eq.J_1}
			\Ll\vert \frac{1}{\vert \cu_m \vert} \sum_{b \in \ov{\cu_m^*} }\bracket{\g_{m,e_i, b}^0 (\pi_b v) }_{\rho} \Rr\vert 
			\leq \Ll(\frac{1}{\vert \cu_m \vert} \sum_{b \in \ov{\cu_m^*} }\bracket{(\pi_b v)^2 }_{\rho} \Rr)^{\frac{1}{2}}\\
			\times \Ll(\Ll(\frac{1}{\vert \cu_m \vert} \sum_{b \in \ov{\cu_m^*} }\bracket{c_b(\pi_b (\ell_{e_i} + \phi^0_{m, e_i} - u_m))^2 }_{\rho} \Rr)^{\frac{1}{2}} + \chi(\rho)^{\frac{1}{2}} \vert \DDb_m(\rho) - \DD(\rho) \vert\Rr).
		\end{multline}
		The error in the second line requires a precise quantitative estimate of the heuristic in \eqref{eq.uv_heuristic}. They are related to the master quantity $J(\rho, \cu_m, p, q)$ in \cite[Proposition~4.7]{funaki2024quantitative}. Especially, Lemma~4.6 and  Proposition~4.7 in  \cite{funaki2024quantitative} respectively yield
		\begin{equation}\label{eq.J_2}
			\begin{split}
				\Ll(\frac{1}{\vert \cu_m \vert} \sum_{b \in \ov{\cu_m^*} }\bracket{c_b(\pi_b (\ell_{e_i} + \phi^0_{m, e_i} - u_m))^2 }_{\rho} \Rr)^{\frac{1}{2}} &\leq \chi(\rho)^{\frac{1}{2}} J(\rho, \cu_m, e_i, \DDb_m e_i)^{\frac{1}{2}}, \\
				\chi(\rho)^{\frac{1}{2}} \vert \DDb_m(\rho) - \DD(\rho) \vert &\leq \chi(\rho)^{\frac{1}{2}} J(\rho, \cu_m, e_i, \DDb_m e_i)^{\frac{1}{2}}.
			\end{split}
		\end{equation}
		The last term $J(\rho, \cu_m, e_i, \DDb_m e_i)$ is positive and has a uniform bound (see \cite[(4.15)]{funaki2024quantitative} and the bound of $\nub, \nub_*$)
		\begin{align*}
			0 \leq J(\rho, \cu_m, e_i, \DDb_m e_i) \leq 12 \lambda.
		\end{align*}
		It also has a polynomial decay (see \cite[Lemma~5.5]{funaki2024quantitative}): there exist two finite positive constants $C, \gamma_1$ only depending on $d,\lambda, \r$, such that for all $\rho \in (0,1)$ and $m \in \N$, we have 
		\begin{align*}
			\chi(\rho)^2 J(\rho, \cu_m, e_i, \DDb_m e_i)  \leq C 3^{-\gamma_1 m}.
		\end{align*}
		We make an interpolation between the two results above, then put it back to \eqref{eq.J_1} and \eqref{eq.J_2} to conclude the desired result.
	\end{proof}

	\subsection{Two-scale expansion of linear statistic}\label{subsec.twoscale}
	Throughout this section, we fix a function $G \in L^2(\X, \fil, \Pr)$ and denote by 
	\begin{align}\label{eq.defGt}
		G_t := P_t G, \qquad \bar G_t := \Pb_t G.
	\end{align}
	We study the quantitative homogenization of semigroup. We recall the notations \eqref{eq.DefTn} and \eqref{eq.DefIn} in Fock space, and also introduce the projection operator  
	\begin{align}\label{eq.defProjection}
		\forall n \in \N, \qquad \Pi_{n} G := I_n(T_n G), \qquad \Pi_{\geq n} G := \sum_{k \geq n} \Pi_k G.
	\end{align}
	We are especially interested in the leading order term, which is denoted by
	\begin{align}\label{eq.defgt}
		g := T_1 G, \qquad g_t := T_1 \bar{G}_t,
	\end{align}
	and $\Pi_1 \bar{G}_t$ is simplified as  
	\begin{align*}
		\Pi_1 \bar{G}_t = I_1 (T_1 \bar{G}_t) = I_1(g_t).
	\end{align*}
	Then $g$ and $g_t$ can be treated as functions defined on $\Zd$. Using $\bar{p}_t = e^{\Ll(\frac{1}{2}\Delta_Q \Rr)t}$ and \eqref{eq.Delta_Q_Closed} and Step~3 in the proof of Proposition~\ref{prop.AvePbarDecay}, $g_t$ satisfies the expression that 
	\begin{align*}
		g_t = \bar{p}_t \ast g.
	\end{align*}

	We propose \emph{the two-scale expansion} for linear statistic as
	\begin{align}\label{eq.defTwoScaleLinear}
		\GG_t := \Pi_1 \bar{G}_t +  \sum_{z \in \Z_m} \sum_{i = 1}^d (\D_{e_i} g_t)_{z+\cu_m} \phi^z_{m, e_i}.
	\end{align}
	Here $\Z_m$ is defined in \eqref{eq.defcu} and $\D_{h}$ is the finite difference operator  on $\Zd$ that 
	\begin{align}\label{eq.defD}
		\forall h \in \Zd, x \in \Zd, \qquad (\D_h g)(x) := g(x+h) - g(x),
	\end{align}
	and $(\D_{e_i} g_t)_{z+\cu_m}$ is the local average of $\D_{e_i} g_t$ in $z+\cu_m$
	\begin{align}\label{eq.DgLocalAvg}
		(\D_{e_i} g_t)_{z+\cu_m} := \frac{1}{\vert \cu_m \vert} \sum_{x \in z + \cu_m } \D_{e_i} g_t(x).
	\end{align}
	
	A similar two-scale expansion can be found in \cite[eq.(4.4)]{gu2024quantitative}. Let us explain the intuition of such expansion in exclusion model. As we know from \eqref{eq.defnu}, the corrector $\phi^z_{m, e_i}$ is the local correction of the function $\ell_{e_i}$. If we hope to correct $\bar{G}_t$, we need to express local increment as a linear combination of $\{\ell_{e_i}\}_{1 \leq i \leq d}$. Then, a natural candidate of the slope along direction $e_i$ is 
	\begin{align}\label{eq.defSlope}
		\frac{\pi_{x, x+e_i} \bar{G}_t}{\pi_{x, x+e_i} \ell_{e_i}}.
	\end{align}
	Here we take the convention $\frac{0}{0} = 0$. This slope is still complicated generally, but its version for the projection in $\HH_1$ is simple
	\begin{align}\label{eq.defSlopeH1}
		\frac{\pi_{x, x+e_i} \Pi_1 \bar{G}_t}{\pi_{x, x+e_i} \ell_{e_i}} = \frac{\pi_{x, x+e_i} I_1(g_t)}{\pi_{x, x+e_i} \ell_{e_i}} = \frac{(g_t(x+e_i) - g_t(x))(\eta_x - \eta_{x+e_i})}{\eta_x - \eta_{x+e_i}} = \D_{e_i}g_t(x).
	\end{align}
	This is just a deterministic function, and $(\D_{e_i} g_t)_{z+\cu_m}$ is its local average by the definition in \eqref{eq.DgLocalAvg}. For this reason, \eqref{eq.defTwoScaleLinear} corrects the term in $\HH_1$, and we consider $\GG_t$ as the two-scale expansion of the linear statistic. $\GG_t$ does not correct the other terms  in $\oplus_{n \geq 2} \HH_n$, but it is already a good approximation of $G_t$, because the projection in higher-order space has a faster decay.

	\smallskip

	We define the Sobolev semi-norm to simplify the notation
	\begin{align*}
		\norm{F}^2_{\dH^k} := \bracket{F (-\Lb)^k F}_\rho.
	\end{align*}
	In particular, when $k = 0$, this is just $L^2$-norm.
	We should also keep in mind, this semi-norm is associated to the specific Dirichlet form $\Lb$. The following lemma gathers several useful estimates. Their proofs are elementary and can be found in Appendix~\ref{sec.Sobolev}.
	\begin{lemma}\label{lem.elementary}
		For $\bar{G}_t = \Pb_t G$ and $0 < \tau < t < \infty$, we have the estimates 
		\begin{align}\label{eq.GtHkDecay}
			\forall k \in \N, \qquad \norm{\bar{G}_t}^2_{\dH^{k}} + 2 \int_{\tau}^t \norm{\bar{G}_s}^2_{\dH^{k+1}} \, \d s = \norm{\bar{G}_\tau}^2_{\dH^{k}},
		\end{align}
		and 
		\begin{equation}\label{eq.GtHkDecayd}
			\forall k \in \N, \qquad \norm{\bar{G}_t}_{\dH^{k+1}}\leq\frac{1}{(t-\tau)^{\frac{1}{2}}}\norm{\bar{G}_\tau}_{\dH^{k}}.
		\end{equation}
		\begin{align}
			\sum_{i=1}^d\chi(\rho) \norm{\D_{e_i} g_t}^2_{\ell^2(\Zd)} &\leq 16 \norm{\Pi_{1} \bar{G}_t}^2_{\dH^{1}},  \label{eq.GgH1} \\
			\sum_{i,j=1}^d\chi(\rho) \norm{\D_{e_i} \D_{e_j} g_t}^2_{\ell^2(\Zd)} &\leq 64 \norm{\Pi_{1} \bar{G}_t}^2_{\dH^{2}}.  \label{eq.GgH2}
		\end{align}
	\end{lemma}

	Our first result shows that,  the two-scale expansion $\tilde G_t^{(1)}$ approximates $\bar G_t$ in $L^2$.
	\begin{lemma}\label{lem.L2linearStats}
		There exists a finite positive constant $C(d, \lambda)$ such that two-scale expansion for linear statistic satisfies
		\begin{equation}\label{eq.L2LinearStats}
			\norm{\tilde G_t^{(1)} - \bar G_t}^2_{L^2}\leq C\Ll(3^{2m}\norm{\Pi_1\bar{G}_t}^2_{\dH^1}+\norm{\Pi_{\geq 2}\bar{G}_t}^2_{L^2}\Rr).
		\end{equation}
	\end{lemma}
	\begin{proof}
		Let us  first calculate $\norm{\tilde G_t^{(1)} - \bar G_t}^2_{L^2}$:
		\begin{equation}\label{eq.L2LinearStats_Step1}
			\begin{split}
				\norm{\tilde G_t^{(1)} - \bar G_t}^2_{L^2}
				&\leq 2\norm{\tilde G_t^{(1)}-\Pi_1\bar{G_t}}^2_{L^2}+2\norm{\Pi_{\geq 2}\bar{G}_t}^2_{L^2}\\
				&\leq 2d \sum_{i=1}^{d}\norm{\sum_{z\in\Z_m}(\D_{e_i}g_t)_{z+\cu_m}\phi^z_{m,e_i}}^2_{L^2}+2\norm{\Pi_{\geq 2}\bar{G}_t}^2_{L^2}.
			\end{split}
		\end{equation}
		Then we focus on one term $\norm{\sum_{z\in\Z_m}(\D_{e_i}g_t)_{z+\cu_m}\phi^z_{m,e_i}}^2_{L^2}$	
		\begin{equation}\label{eq.L2LinearStats_Step2}
			\begin{split}
				&\norm{\sum_{z\in\Z_m}(\D_{e_i}g_t)_{z+\cu_m}\phi^z_{m,e_i}}^2_{L^2}\\
				&=\sum_{z,z'\in\Z_m}(\D_{e_i}g_t)_{z+\cu_m} (\D_{e_i}g_t)_{z'+\cu_m}\bracket{\phi^z_{m,e_i}\phi^{z'}_{m,e_i}}_\rho\\
				&=\sum_{z\in\Z_m}(\D_{e_i}g_t)^2_{z+\cu_m}\norm{\phi^z_{m,e_i}}^2_{L^2}.
			\end{split}
		\end{equation}
		In the second line, we use the fact that $(\D_{e_i}g_t)_{z+\cu_m}$ is a deterministic constant. 
		In the passage from the second line to the third line, the only contribution comes from the term $z=z'$, thanks to the independence between the local correctors $\{\phi^z_{m,e_i}\}_{z \in \Z_m}$ and the property $\bracket{\phi^z_{m,e_i}}_\rho = 0$ (see Proposition~\ref{prop.phiz}). 
		
		Then we insert the estimates \eqref{eq.L2phiz}
		\begin{equation}\label{eq.L2LinearStats_Step3}
			\begin{split}
				&\sum_{z\in\Z_m}(\D_{e_i}g_t)^2_{z+\cu_m}\norm{\phi^z_{m,e_i}}^2_{L^2}\\
				&\leq C \chi(\rho)3^{(d+2)m}\sum_{z\in\Z_m}\frac{1}{|\cu_m|}\sum_{x\in z+\cu_m} \vert \D_{e_i}g_t(x) \vert^2\\
				&\leq C 3^{2m}\norm{\Pi_1\bar{G}_t}^2_{\dH^1}.
			\end{split}
		\end{equation}
		The second line utilizes Jensen's inequality, and the third line uses \eqref{eq.GgH1}. Putting \eqref{eq.L2LinearStats_Step2} and \eqref{eq.L2LinearStats_Step3} back to \eqref{eq.L2LinearStats_Step1}, we obtain the desired result.
	\end{proof}
	
	The next result proves that $\tilde{G}_t^{(1)}$ is close to $G_t$ in $L^2$.
	
	\begin{proposition}\label{prop.TwoScaleLinear}
		There exists a constant $C(d, \lambda, \r,\rho) < +\infty$ such that the two-scale expansion in \eqref{eq.defTwoScaleLinear} satisfies 
		\begin{multline}\label{eq.TwoScaleLinearBound}
			\norm{\GG_t - G_t}_{L^2} + \Ll(\int_0^t \norm{\GG_s - G_s}^2_{\dH^1}  \, \d s \Rr)^{\frac{1}{2}}\\
			\leq C \Ll(\Ll(3^{-\alpha m} + 3^m t^{-\frac{5}{8}}\Rr)\norm{\Pi_1 G}_{L^2} + t^{-\frac{1}{8}}\norm{G}_{L^2} + 3^m \norm{\Pi_1 G}_{\dH^1} + 3^m t^{\frac{5}{8}}\norm{\Pi_1 G}_{\dH^2} + \norm{\Pi_{\geq 2} G}_{L^2} \Rr).
		\end{multline}
	\end{proposition}
	\begin{proof}		
		Using the identity
		\begin{align*}
			\Ll(\partial_s - \L \Rr) G_s = \Ll(\partial_s - \Lb \Rr) \Pi_1\bar G_s = 0,
		\end{align*}
		we deduce that
		\begin{align*}
			\Ll(\partial_s - \L \Rr)(\tilde G_s^{(1)} - G_s) 
			&= \Ll(\partial_s - \L\Rr)\tilde G_s^{(1)} - \Ll(\partial_s - \Lb \Rr) \Pi_1\bar G_s \\
			&= \partial_s(\tilde G_s^{(1)} -  \Pi_1\bar G_s) + (- \L \tilde G_s^{(1)} +  \Lb \Pi_1\bar G_s) .
		\end{align*}
		Then we test the two sides of equation above with $(\tilde G_s^{(1)} - G_s)$, and obtain that 
		\begin{equation}\label{eq.TwoScaleLinearBound_Decom}
			\begin{split}
				& \frac{1}{2}\bracket{(\tilde G_t^{(1)} - G_t)^2}_{\rho} + \int_0^t \bracket{(\tilde G_s^{(1)} - G_s)(-\L)(\tilde G_s^{(1)} - G_s)}_{\rho}  \, \d s \\
				& \qquad \leq \frac{1}{2}\bracket{(\tilde G_0^{(1)} - G_0)^2}_{\rho} +  \int_0^t \bracket{ (\tilde G_s^{(1)} - G_s)\partial_s(\tilde G_s^{(1)} -  \Pi_1\bar G_s)}_{\rho} \, \d s  \\
				& \qquad \qquad + \int_0^t \bracket{(\tilde G_s^{(1)} - G_s)(-\L\tilde G_s^{(1)} + \Lb \Pi_1\bar G_s)}_{\rho}  \, \d s. 
			\end{split}
		\end{equation}
		Using \eqref{eq.EnergyLbarUpper} in Corollary~\ref{cor.CovSEP} we have
		\begin{equation}\label{eq.CompensateH1}
			\int_0^t \bracket{(\tilde G_s^{(1)} - G_s)(-\L)(\tilde G_s^{(1)} - G_s)}_{\rho}\, \d s\geq \frac{1}{C_{\ref{cor.CovSEP}}}\int_0^t\norm{\tilde G_s^{(1)} - G_s}_{\dH^1}^2\, \d s, 
		\end{equation}
		We then treat the right-hand side term by term.
		
		\medskip
		
		\textit{Step~1: term $\bracket{(\tilde G_0^{(1)} - G_0)^2}_{\rho}$.} We use the result \eqref{eq.L2LinearStats} to obtain 
		\begin{align}\label{5.2}
			\norm{\tilde G_0^{(1)} - G_0}^2_{L^2} = \norm{\tilde G_0^{(1)} - \bar G_0}^2_{L^2} \leq C \Ll(3^{2m}\norm{\Pi_1 G}^2_{\dH^1} + \norm{\Pi_{\geq 2}G}^2_{L^2}\Rr).
		\end{align}
		
		\medskip
		
		\textit{Step~2: term $ \int_0^t \bracket{ (\tilde G_s^{(1)} - G_s)\partial_s(\tilde G_s^{(1)} -  \Pi_1\bar G_s)}_{\rho} \, \d s $.} We have the following estimate 
		\begin{align*}
			&\Ll\vert \int_0^t  \bracket{ (\tilde G_s^{(1)} - G_s)\partial_s(\tilde G_s^{(1)} - \Pi_1\bar G_s)}_{\rho} \, \d s  \Rr\vert \\
			&\leq \int_0^t\Ll(  \frac{t^{-\frac{5}{4}}}{2}\bracket{ (\tilde G_s^{(1)} - G_s)^2}_{\rho} + \frac{t^{\frac{5}{4}}}{2} \bracket{\Ll(\partial_s(\tilde G_s^{(1)} - \Pi_1 \bar G_s)\Rr)^2}_{\rho}  \Rr) \, \d s\\
			&\leq 3 \int_0^t\Ll(  \frac{t^{-\frac{5}{4}}}{2} \Ll(\bracket{ (\tilde G_s^{(1)} - \bar{G}_s)^2 + ( \bar{G}_s)^2 + ( G_s)^2}_{\rho} \Rr) + \frac{t^{\frac{5}{4}}}{2} \bracket{\Ll(\partial_s(\tilde G_s^{(1)} - \Pi_1 \bar G_s)\Rr)^2}_{\rho} \Rr) \, \d s.
		\end{align*}
		Using the $L^2$ decay, we have 
		\begin{align*}
			\int_0^t  \frac{t^{-\frac{5}{4}}}{2} \bracket{( \bar{G}_s)^2 + ( G_s)^2}_{\rho}     \, \d s \leq t^{-\frac{1}{4}} \bracket{(\bar{G}_0)^2}_{\rho}.
		\end{align*}
		We apply the estimate \eqref{eq.L2LinearStats} to the term $ (\tilde G_s^{(1)} - \bar{G}_s)^2$ 
		\begin{align*}
			&\int_0^t \frac{t^{-\frac{5}{4}}}{2} \bracket{ (\tilde G_s^{(1)} - \bar{G}_s)^2}_\rho   \, \d s \\
			&\leq C  \int_0^t \frac{t^{-\frac{5}{4}}}{2} \Ll(3^{2m}\norm{\Pi_1\bar{G}_{s}}^2_{\dH^1} + \norm{\Pi_{\geq 2}\bar G_s}^2_{L^2}\Rr)  \, \d s \\
			&\leq   Ct^{-\frac{5}{4}}3^{2m}\norm{\Pi_1\bar G_0}^2_{L^2} + C t^{-\frac{1}{4}}\norm{\Pi_{\geq 2}\bar G_0}^2_{L^2}.
		\end{align*}
		Here from the second line to the third line, we use the decay of Dirichlet form \eqref{eq.GtHkDecay}.

		Similar estimate also applies to $\Ll(\partial_s(\tilde G_s^{(1)} -  \Pi_1\bar G_s)\Rr)^2$
		\begin{align*}
			&\int_0^t \frac{t^{\frac{5}{4}}}{2} \bracket{\Ll(\partial_s(\tilde G_s^{(1)} - \Pi_1 \bar G_s)\Rr)^2}_{\rho}  \, \d s \\
			&= \int_0^t \frac{t^{\frac{5}{4}}}{2}  \bracket{\Ll( \sum_{i=1}^d\sum_{z \in \Z_m} (\D_{e_i}\partial_s g_s)_{z+\cu_m} \phi^z_{m, e_i}\Rr)^2}_{\rho}  \, \d s\\
			&\leq C  \int_0^t\frac{t^{\frac{5}{4}}}{2}   3^{2m}\norm{\Pi_1\bar G_s}^2_{\dH^3}  \, \d s \\
			&\leq   Ct^{\frac{5}{4}} 3^{2m}\norm{\Pi_1\bar{G}_0}^2_{\dH^2}.
		\end{align*}
		Here in the third line, we use the property $\partial_s \Pi_1\bar{G}_s = \Lb \Pi_1\bar{G}_s$ and the $L^2$-moment estimate \eqref{eq.L2phiz}. From the third line to the forth line, we use the decay of Dirichlet form \eqref{eq.GtHkDecay}.

		We conclude that 
		\begin{multline}\label{5.3}
			\Ll\vert \int_0^t \bracket{ (\tilde G_s^{(1)} - G_s)\partial_s(\tilde G_s^{(1)} - \Pi_1 \bar G_s)}_{\rho} \, \d s \Rr\vert \\
			\leq C\Ll(t^{-\frac{1}{4}}\norm{\bar{G}_0}^2_{L^2}+t^{-\frac{5}{4}}3^{2m}\norm{\Pi_1\bar{G}_0}^2_{L^2}+t^{\frac{5}{4}}3^{2m}\norm{\Pi_1 \bar{G}_0}^2_{\dH^2}\Rr).
		\end{multline}

		\medskip
		
		\textit{Step~3: term $\int_0^t \bracket{(\tilde G_s^{(1)} - G_s)(-\L\tilde G_s^{(1)} + \Lb \Pi_1 \bar G_s)}_{\rho}  \, \d s$.} This term involves the flux replacement, which is a key estimate in homogenization. We address it in Lemma~\ref{lem.TwoScale_Flux} below, and cite the estimate \eqref{eq.weaknormestimate}  there. Then we obtain
		\begin{equation}\label{5.4}
			\begin{aligned}
				&\Ll|\int_0^t \bracket{(\tilde G_s^{(1)} - G_s)(-\L\tilde G_s^{(1)} + \Lb \Pi_1 \bar G_s)}_{\rho}  \, \d s\Rr|\\
				&\qquad \leq C \int_{0}^{t}\norm{\tilde{G}_{s}^{(1)}-G_s}_{\dH^1}\Ll(3^{-\alpha m}\norm{\Pi_1\bar{G}_s}_{\dH^1}+3^m\norm{\Pi_1\bar{G}_s}_{\dH^2}\Rr)\d s\\
				&\qquad\leq \frac{1}{2C_{\ref{cor.CovSEP}}}\int_{0}^{t}\norm{\tilde{G}_s^{(1)}-G_s}^2_{\dH^1}\d s+C\int_{0}^{t}\Ll(3^{-2\alpha m}\norm{\Pi_1\bar{G}_s}^2_{\dH^1}+3^{2m}\norm{\Pi_1\bar{G}_s}^2_{\dH^2}\Rr)\d s\\
				&\qquad\leq \frac{1}{2C_{\ref{cor.CovSEP}}}\int_{0}^{t}\norm{\tilde{G}_s^{(1)}-G_s}^2_{\dH^1}\d s+C 3^{-2\alpha m}\norm{\Pi_1\bar{G}_0}^2_{L^2}+C 3^{2m}\norm{\Pi_1\bar{G}_0}^2_{\dH^1}
			\end{aligned}
		\end{equation}
		
		Finally, we put \eqref{eq.CompensateH1}, \eqref{5.2}, \eqref{5.3}, and \eqref{5.4} back to \eqref{eq.TwoScaleLinearBound_Decom}. The term $\frac{1}{2C_{\ref{cor.CovSEP}}}\int_{0}^{t}\norm{\tilde{G}_s^{(1)}-G_s}^2_{\dH^1}\d s$ above compensates part of \eqref{eq.CompensateH1}, and we obtain the desired result.
		
	\end{proof}

	We treat the technical estimates about the flux replacement in the following lemma.
	\begin{lemma}\label{lem.TwoScale_Flux}
		There exists a finite positive constant $C(d, \lambda, \r,\rho)$ such that the following estimate holds
		\begin{equation}\label{eq.weaknormestimate}
			\Ll|\bracket{V\Ll(-\L\tilde{G}_s^{(1)}+\Lb\Pi_1\bar{G}_s\Rr)}_\rho\Rr|\leq C\norm{V}_{\dH^1}\Ll(3^{-\alpha m}\norm{\Pi_1\bar{G}_s}_{\dH^1}+3^m\norm{\Pi_1\bar{G}_s}_{\dH^2}\Rr).
		\end{equation}
	\end{lemma}
	\begin{proof}
		The proof can be divided into 4 steps.
		
		\textit{Step~0: decomposition.} We start from the terms involving $\L$    
		\begin{align*}
			\bracket{V(-\L\tilde G_s^{(1)})}_\rho =  \frac{1}{2}\sum_{z \in \Z_m} \sum_{b \in \ov{(z + \cu_m)^*}} \bracket{c_b (\pi_b V) (\pi_b \tilde G_s^{(1)})}_\rho.
		\end{align*}
		Using the expression of two-scale expansion $\tilde G_s^{(1)}$ in \eqref{eq.defTwoScaleLinear}, we develop $\pi_b \tilde G_s^{(1)}$ for $b \in \ov{(z + \cu_m)^*}$ 
		\begin{equation}\label{eq.piTwoScaleLinear}
			\begin{split}
				\pi_b \tilde G_s^{(1)} &= \pi_b \Pi_1\bar{G}_s +   \sum_{i = 1}^d\sum_{z'\in\Z_m}(\D_{e_i}g_s)_{z'+\cu_m}(\pi_b\phi^{z'}_{m,e_i})\\
				&= \pi_b \Pi_1\bar{G}_s + \sum_{i = 1}^d (\D_{e_i}g_s)_{z+\cu_m}(\pi_b\phi^{z}_{m,e_i}).
			\end{split}
		\end{equation}
		In the first line, because $(\D_{e_i}g_s)_{z'+\cu_m}$ does not depend on the configuration, $\pi_b$ does not act on it. Afterwards, since all the functions $\phi^{z'}_{m,e_i}$ are local (see Proposition~\ref{prop.phiz}), the only non-vanishing term  is $\pi_b\phi^{z}_{m,e_i}$. 
		
		We then aim to make  the centered flux $\g_{m,e_i, b}^z$ (see \eqref{eq.defFlux} for its definition) appear in \eqref{eq.piTwoScaleLinear}, which requires the contribution of $c_b \pi_b \Pi_1\bar{G}_s$. 
		For every $b=\{x,x+e_j\} \in \ov{(z + \cu_m)^*}$, we have 
		\begin{align*}
			\pi_b\Pi_1\bar{G}_s=(g_s(x+e_j)-g_s(x))(\bar{\eta}_x-\bar{\eta}_{x+e_j})=\D_{e_j}g_s(x)(\eta_x-\eta_{x+e_j}).
		\end{align*}
		We notice that 
		\begin{align*}
			\pi_b \ell_{e_i} = e_i \cdot (x+e_j - x) (\eta_x-\eta_{x+e_j}) = \Ind{i=j}  (\eta_x-\eta_{x+e_j}),
		\end{align*}
		and this gives us the identity
		\begin{align*}
			\pi_b\Pi_1\bar{G}_s = \sum_{i=1}^d \D_{e_i}g_s(x)  \pi_b \ell_{e_i}.
		\end{align*} 
		We can further develop this expression by subtracting the local average of slope $ (\D_{e_i}g_s)_{z+\cu_m}$
		\begin{align*}
			\pi_b\Pi_1\bar{G}_s =  \sum_{i=1}^d \Ll(\D_{e_i}g_s(x) -  (\D_{e_i}g_s)_{z+\cu_m}\Rr)  \pi_b \ell_{e_i} +  \sum_{i=1}^d (\D_{e_i}g_s)_{z+\cu_m}  \pi_b \ell_{e_i}. 
		\end{align*}
		We put this formula back to \eqref{eq.piTwoScaleLinear}, and obtain that
		\begin{equation*}
			\pi_b \tilde G_s^{(1)} 	= \sum_{i = 1}^d (\D_{e_i}g_s)_{z+\cu_m}(\pi_b(\ell_{e_i} + \phi^{z}_{m,e_i})) + \sum_{i=1}^d \Ll(\D_{e_i}g_s(x) -  (\D_{e_i}g_s)_{z+\cu_m}\Rr)  \pi_b \ell_{e_i}.
		\end{equation*}
		We further apply $c_b$, which yields
		\begin{align*}
			c_b\pi_b \tilde G_s^{(1)} &	= \sum_{i = 1}^d (\D_{e_i}g_s)_{z+\cu_m}(c_b\pi_b(\ell_{e_i} + \phi^{z}_{m,e_i})) + \sum_{i=1}^d \Ll(\D_{e_i}g_s(x) -  (\D_{e_i}g_s)_{z+\cu_m}\Rr) c_b \pi_b \ell_{e_i} \\
			&= \sum_{i = 1}^d (\D_{e_i}g_s)_{z+\cu_m}\underbrace{\Ll(c_b  \pi_b (\ell_{e_i} + \phi^z_{m, e_i} ) - \pi_b \ell_{\DD(\rho) e_i}\Rr)}_{\g^z_{m,e_i, b}} \\
			& \qquad + \sum_{i=1}^d \Ll(\D_{e_i}g_s(x) -  (\D_{e_i}g_s)_{z+\cu_m}\Rr) c_b \pi_b \ell_{e_i} \\
			& \qquad + \sum_{i = 1}^d (\D_{e_i}g_s)_{z+\cu_m} \pi_b \ell_{\DD(\rho) e_i}. 
		\end{align*}
		Here we make appear $\g^z_{m,e_i, b}$ as desired. Therefore,  we conclude that
		\begin{align}\label{eq.fluxdecomshort}
			\bracket{V(-\L\tilde G_s^{(1)} + \Lb  \Pi_1\bar G_s)}_{\rho}  &= \mathbf{F.1} + \mathbf{F.2} + \mathbf{F.3},
		\end{align}
		where the three terms are
		\begin{equation}\label{eq.fluxdecomlong}
			\begin{split}
				\mathbf{F.1} &:= \frac{1}{2}\sum_{z \in \Z_m} \sum_{i = 1}^d   \sum_{b \in \ov{(z + \cu_m)^*}}(\D_{e_i}g_s)_{z+\cu_m}\bracket{ (\pi_{b} V)  \g^z_{m,e_i, b}}_{\rho},\\
				\mathbf{F.2} &:= \frac{1}{2}\sum_{z \in \Z_m} \sum_{i = 1}^d  \sum_{x \in z+\cu_m} (\D_{e_i}g_s(x)-(\D_{e_i}g_s)_{z+\cu_m}) \bracket{ (\pi_{x,x+e_i} V)  c_{x,x+e_i} (\eta_x - \eta_{x+e_i}) }_\rho, \\
				\mathbf{F.3} &:= \frac{1}{2}\sum_{z \in \Z_m} \sum_{i = 1}^d   \sum_{b \in \ov{(z + \cu_m)^*}} (\D_{e_i}g_s)_{z+\cu_m} \bracket{(\pi_{b} V)  \Ll(\pi_b \ell_{\DD(\rho) e_i}\Rr) }_{\rho} - \bracket{V(-\Lb \Pi_1\bar{G}_s)}_\rho.
			\end{split}
		\end{equation}
		These three terms have their own interpretations. The term $\mathbf{F.1}$ is the main part of the flux replacement. The term $\mathbf{F.2}$ is the error to fix the local slope. The term $\mathbf{F.3}$ is the error for discrete approximation. A similar decomposition of two-scale expansion can be found in the previous work \cite[eq.(4.10)]{gu2024quantitative}.
		
		In the following paragraphs, we treat the three terms separately.
		
		\medskip
		
		\textit{Step~1: term~$\mathbf{F.1}$ as the error in flux replacement.}
		For this term, we make appear the centered flux $\g_{m,e_i,b}^z$. Moreover, as the averaged slope $(\D_{e_i}g_s)_{z+\cu_m}$ does not depend on the configuration in $(z+\cu_m)^+$, we apply the flux cancellation \eqref{eq.FluxCancel} to obtain
		\begin{align*}
			\Ll\vert \bracket{\sum_{b \in \ov{(z + \cu_m)^*}} (\pi_{b} V) \g_{m,e_i,b}^z}_{\rho} \Rr\vert \leq C 3^{-\alpha m} \vert \cu_m \vert^{\frac{1}{2}}\Ll(\sum_{b \in \ov{(z+\cu_m)^*} }\bracket{(\pi_b V)^2}_{\rho}\Rr)^{\frac{1}{2}},
		\end{align*}
		and Jensen's inequality for $(\D_{e_i}g_s)_{z+\cu_m}$ to obtain
		\begin{align*}
			\Ll\vert (\D_{e_i}g_s)_{z+\cu_m} \Rr\vert \leq C \vert \cu_m \vert^{-\frac{1}{2}} \Ll(\sum_{x \in z +\cu_m}\bracket{ (\pi_{x, x+e_i} \Pi_1\bar{G}_s)^2}_{\rho}\Rr)^{\frac{1}{2}}. 
		\end{align*} 
		The volume factor $\vert \cu_m \vert^{\frac{1}{2}}$ compensates in the product of two estimates above. We then apply Jensen's inequality and obtain
		\begin{align}\label{eq.F.1}
			\vert \mathbf{F.1} \vert &\leq C 3^{-\alpha m} \norm{V}_{\dH^1} \norm{\Pi_1\bar{G}_s}_{\dH^1}.
		\end{align}

		\medskip
		
		\textit{Step~2: term~$\mathbf{F.2}$ as the error to fix the slope.}
		We apply at first the Cauchy--Schwarz inequality and obtain
		\begin{align*}
			\vert \bracket{ (\pi_{x,x+e_i} V)  c_{x,x+e_i} (\eta_x - \eta_{x+e_i}) }_\rho \vert \leq \lambda \chi(\rho)^{\frac{1}{2}} \bracket{ (\pi_{x,x+e_i} V)^2}^{\frac{1}{2}}_\rho .
		\end{align*}
		We then insert this estimate in $\mathbf{F.2}$ 
		\begin{align*}
			\vert \mathbf{F.2} \vert & \leq \sum_{z \in \Z_m} \sum_{i=1}^{d}\sum_{x\in z+\cu_m} \Ll\vert \D_{e_i}g_s(x)-(\D_{e_i}g_s)_{z+\cu_m}\Rr\vert  \lambda \chi(\rho)^{\frac{1}{2}} \bracket{ (\pi_{x,x+e_i} V)^2}^{\frac{1}{2}}_\rho \\
			&\leq \lambda \chi(\rho)^{\frac{1}{2}}\Ll(\sum_{b \in (\Zd)^*}\bracket{(\pi_{b}V)^2}_\rho\Rr)^{\frac{1}{2}}\Ll( \sum_{z \in \Z_m}\sum_{i=1}^{d}\sum_{x\in z+\cu_m}\Ll(\D_{e_i}g_s(x)-(\D_{e_i}g_s)_{z+\cu_m}\Rr)^2\Rr)^{\frac{1}{2}}\\
			&\leq 3^m \lambda \chi(\rho)^{\frac{1}{2}} \Ll(\sum_{b \in (\Zd)^*}\bracket{(\pi_{b}V)^2}_\rho\Rr)^{\frac{1}{2}}\Ll( \sum_{z \in \Z_m}\sum_{x\in z+\cu_m}\sum_{i,j=1}^{d}\Ll(\D_{e_j}\D_{e_i}g_s(x)\Rr)^2\Rr)^{\frac{1}{2}}.
		\end{align*}
		The passage from the second line to the third line makes use of Poincar\'e's inequality. Using \eqref{eq.GgH2}, we have 
		\begin{align}\label{eq.F.2}
			\vert \mathbf{F.2} \vert 
			\leq C 3^m \norm{V}_{\dH^1}\norm{\Pi_1\bar{G}_s}_{\dH^2}.
		\end{align}
		
		\medskip
		
		\textit{Step~3: term~$\mathbf{F.3}$ as the error for discrete approximation.} We study the two terms respectively in $\mathbf{F.3}$. As a preparation, we notice the Kawasaki operator $x,y \mapsto \pi_{x,y}$ is symmetric, and we can use $\overrightarrow{\pi}_{x, y}$ below to measure the change when particle moves from $x$ to $y$: for every $F: \X \to \R$, we define
		\begin{align}\label{eq.defapi}
			\api_{x, y}F := (\pi_{x, y} F) (\eta_x - \eta_y).
		\end{align}   
		Then the mapping $x,y \mapsto \api_{x,y}$ is anti-symmetric. 	Moreover, the operator $\api$ satisfies the following chain rule: given $\{x_i\}_{0 \leq i \leq n} \subset \Lambda \subset \Zd$, we have
		\begin{align}\label{eq.chain}
			\bracket{\api_{x_0, x_n} F}_{\rho, \Lambda} = \sum_{i=0}^{n-1} \bracket{\api_{x_i, x_{i+1}} F}_{\rho, \Lambda}.
		\end{align}
		See \cite[Lemma~3.3]{funaki2024quantitative} for its proof.
		
		\textit{Step~3.1: the first term in $\mathbf{F.3}$.} Recall Lemma~\ref{lem.CovRW} and Corollary~\ref{cor.CovSEP} that
		\begin{align*}
			\DD(\rho)=\frac{1}{2}\sum_{y\in\Zd}Q_y yy^{\T}. 
		\end{align*}
		Thus, for every $x\in\Zd$ and $j\in\{1,\cdots,d\}$, we have the identity
		\begin{equation*}
			\pi_{x,x+e_j}\ell_{\DD(\rho)e_i}=e_j^{\T}\DD(\rho)e_i(\eta_x-\eta_{x+e_j})=\frac{1}{2}\sum_{y\in\Zd}Q_y\Ll(y^{\T} e_j\Rr)\Ll(y^{\T} e_i\Rr)(\eta_x-\eta_{x+e_j}).
		\end{equation*}
		Then we can simplify the first term in $\mathbf{F.3}$:
		\begin{multline}\label{eq.FixSlope3}
			\frac{1}{2}\sum_{z \in \Z_m} \sum_{i = 1}^d   \sum_{b \in \ov{(z + \cu_m)^*}}  \bracket{(\pi_{b} V)  \Ll(\pi_b \ell_{\DD(\rho) e_i}\Rr)(\D_{e_i}g_s)_{z+\cu_m} }_{\rho}\\
			=\frac{1}{4}\sum_{y\in\Zd}\sum_{z\in\Z_m}\sum_{x\in z+\cu_m}Q_y\Ll(\sum_{i=1}^{d}(y^{\T} e_i)(\D_{e_i}g_s)_{z+\cu_m}\Rr)\Ll(\sum_{j=1}^{d}(y^{\T} e_j)\bracket{\api_{x,x+e_j}V}_\rho\Rr).
		\end{multline}
		
		\smallskip
		
		\textit{Step~3.2: the second term in $\mathbf{F.3}$.}  For any $x,y\in\Zd$, we first define the canonical path from $x$ to $x+y$. We fix the order of the coordinate $e_1, e_2, \cdots e_d$, then connect $x$ to $(x+y)$ along a geodesic path in $\ell^1$ distance on lattice  $\Zd$
		\begin{align*}
			x\rightarrow x+y&:=x\rightarrow x+\mathrm{sgn}(y^{\T}e_1)e_1\rightarrow\cdots\rightarrow x+(y^{\T}e_1)e_1\rightarrow x+(y^{\T}e_1)e_1+\mathrm{sgn}(y^{\T}e_2)e_2\\
			&\qquad\rightarrow\cdots\rightarrow x+(y^{\T}e_1)e_1+(y^{\T}e_2)e_2\rightarrow\cdots\rightarrow x+y.
		\end{align*}
		To be convenient we denote the vertex at $n$-th step by 
		$x_n$. We consider the terms involving $\Lb$
		\begin{equation}\label{eq.weeknormLbar}
			\bracket{V(-\Lb \Pi_1\bar{G}_s}_\rho=\frac{1}{4}\sum_{x\in\Zd}\sum_{y\in\Zd}Q_{y}\bracket{(\pi_{x,x+y}V)(\pi_{x,x+y}\Pi_1\bar{G}_s)}_\rho.
		\end{equation}
		Using the identity $\pi_{x,x+y}\Pi_1\bar{G}_s= \D_{y}g_s(x)(\eta_x-\eta_{x+y})$ and the definition of $\api$ in \eqref{eq.defapi}, we have
		\begin{align*}
			\bracket{(\pi_{x,x+y}V)(\pi_{x,x+y}\Pi_1\bar{G}_s)}_\rho
			&=\D_{y}g_s(x)\bracket{\api_{x,x+y}V}_\rho\\
			&=\Ll(\sum_{n=1}^{\vert y \vert_{1}}\D_{x_n-x_{n-1}}g_s(x_{n-1})\Rr)\Ll(\sum_{n=1}^{\vert y \vert_{1}}\bracket{\api_{x_{n-1},x_n }V}_\rho\Rr),
		\end{align*}
		Here in the second line, we use the property $\D_{y}g_s(x)=\sum_{n=1}^{\vert y \vert_{1}}\D_{x_n-x_{n-1}}g_s(x_{n-1})$ and the chain rule \eqref{eq.chain} to decompose terms through the canonical paths.
		
		We aim to shift the finite difference and the Kawasaki operators above to $x$. We illustrate the error in this procedure using the following calculation. For every $i,j\in\{1,\cdots,d\}$, $w_1,w_2\in\Zd$ such that $|w_1-w_2|=1$, we have that
		\begin{equation}\label{eq.NearBondEqiv}
			\begin{aligned}
				&\Ll|\sum_{x\in\Zd}\Ll(\D_{e_j}g_s(x)\bracket{\api_{x+w_1,x+w_2}V}_\rho\Rr)-\sum_{x\in\Zd}\Ll(\D_{e_j}g_s(x)\bracket{\api_{x+w_1+e_i,x+w_2+e_i}V}_\rho\Rr)\Rr|\\
				&\quad =\Ll|\sum_{x\in\Zd}\Ll(\D_{e_j}g_s(x)-\D_{e_j}g_s(x-e_i)\Rr)\bracket{\api_{x+w_1,x+w_2}V}_\rho\Rr|\\
				&\quad\leq\Ll(\sum_{x\in\Zd}\Ll(\D_{e_i}\D_{e_j}g_{s}(x)\Rr)^2\Rr)^{\frac{1}{2}}\Ll(\sum_{x\in\Zd}\bracket{\Ll(\api_{x+w_1,x+w_2}V\Rr)^2}_\rho\Rr)^{\frac{1}{2}}\\
				&\quad\leq C\norm{\Pi_1\bar{G}_s}_{\dH^2}\norm{V}_{\dH^1}.
			\end{aligned}
		\end{equation}
		The last line makes use of \eqref{eq.GgH2}. Since the support of $Q$ is a finite, we use \eqref{eq.NearBondEqiv}, $\D_{y}g_s(x)=\D_{-y}g_s(x+y)$ and  $\api$ to approximate \eqref{eq.weeknormLbar} in the sense 
		\begin{align}\label{eq.FixSlope2_Rough}
			\bracket{V(-\Lb \Pi_1\bar{G}_s}_\rho \simeq \frac{1}{4}\sum_{y\in\Zd}\sum_{x\in\Zd}Q_y\Ll(\sum_{i=1}^{d}(y^{\T} e_i)\D_{e_i}g_s(x)\Rr)\Ll(\sum_{j=1}^{d}(y^{\T} e_j)\bracket{\api_{x,x+e_j}V}_\rho\Rr),
		\end{align}
		with an error estimate
		\begin{multline}\label{eq.FixSlope2}
			\Ll|\bracket{V(-\Lb \Pi_1\bar{G}_s}_\rho-\frac{1}{4}\sum_{y\in\Zd}\sum_{x\in\Zd}Q_y\Ll(\sum_{i=1}^{d}(y^{\T} e_i)\D_{e_i}g_s(x)\Rr)\Ll(\sum_{j=1}^{d}(y^{\T} e_j)\bracket{\api_{x,x+e_j}V}_\rho\Rr)\Rr|\\
			\leq C\norm{\Pi_1\bar{G}_s}_{\dH^2}\norm{V}_{\dH^1}.
		\end{multline}
		Here the constant $C$ only depends on $d, \lambda,$ and $\rho$, as indicated in (2) of Lemma~\ref{lem.CovRW}.
		\smallskip
		
		\textit{Step~3.3: comparison between two terms in $\mathbf{F.3}$.} The two terms in $\mathbf{F.3}$ are now simplified in similar expression in \eqref{eq.FixSlope3} and \eqref{eq.FixSlope2_Rough}. 
		We then compare their difference in one cube via the Cauchy--Schwarz inequality and Poincar\'e's inequality
		\begin{align*}
			&\Ll|\sum_{x\in z+\cu_m}\Ll(\sum_{i=1}^{d}(y^{\T} e_i)\Ll(\D_{e_i}g_s(x)-(\D_{e_i}g_s)_{z+\cu_m}\Rr)\Rr)\Ll(\sum_{j=1}^{d}(y^{\T} e_j)\bracket{\api_{x,x+e_j}V}_\rho\Rr)\Rr|^2\\
			&\leq d \Ll(\sum_{x\in z+\cu_m}\sum_{i=1}^{d}(y^{\T} e_i)^2\Ll(\D_{e_i}g_s(x)-(\D_{e_i}g_s)_{z+\cu_m}\Rr)^2\Rr)\Ll(\sum_{x\in z+\cu_m}\sum_{j=1}^{d}(y^{\T} e_j)^2\bracket{\Ll(\api_{x,x+e_j}V\Rr)^2}_{\rho}\Rr)\\
			&\leq C 3^{2m}|y|^4\Ll(\sum_{i=1}^{d}\sum_{x\in z+\cu_m}\sum_{j=1}^{d}\Ll(\D_{e_j}\D_{e_i}g_s(x)\Rr)^2\Rr)\Ll(\sum_{j=1}^{d}\sum_{x\in z+\cu_m}\bracket{\Ll(\api_{x,x+e_j}V\Rr)^2}_\rho\Rr).
		\end{align*}
		Since the support of $Q$ is a finite set, we can use Jensen's inequality to obtain:
		\begin{equation}\label{eq.FixSlope4}
			\begin{aligned}
				&\Ll|\frac{1}{4}\sum_{y\in\Zd}\sum_{z\in\Z_m}\sum_{x\in z+\cu_m}Q_y\Ll(\sum_{i=1}^{d}(y^{\T} e_i)(\D_{e_i}g_s)_{z+\cu_m}\Rr)\Ll(\sum_{j=1}^{d}(y^{\T} e_j)\bracket{\api_{x,x+e_j}V}_\rho\Rr)\Rr.\\
				&\Ll.-\frac{1}{4}\sum_{y\in\Zd}\sum_{x\in\Zd}Q_y\Ll(\sum_{i=1}^{d}(y^{\T} e_i)\D_{e_i}g_s(x)\Rr)\Ll(\sum_{j=1}^{d}(y^{\T} e_j)\bracket{\api_{x,x+e_j}V}_\rho\Rr)\Rr|\\ 
				&\quad \leq C 3^m \norm{V}_{\dH^1}\norm{\Pi_1\bar{G}_s}_{\dH^2}.
			\end{aligned}
		\end{equation}
		Combining \eqref{eq.FixSlope3}, \eqref{eq.FixSlope2}, \eqref{eq.FixSlope4} we obtain
		\begin{equation}\label{eq.F.3}
			|\mathbf{F.3}|\leq C 3^m \norm{V}_{\dH^1}\norm{\Pi_1\bar{G}_s}_{\dH^2}.
		\end{equation}
		Finally, we combine \eqref{eq.F.1}, \eqref{eq.F.2} and \eqref{eq.F.3} to obtain the desired result \eqref{eq.weaknormestimate}.
	\end{proof}

	\subsection{An elementary regularization}\label{subsec.reg}
	We will implement a step of regularization, so that all the terms on the {\rhs} of \eqref{eq.TwoScaleLinearBound} show a more explicit decay. 
	
	\begin{lemma}\label{lem.Regularization}
		The following estimates hold for all $\tau,t>0$:
		\begin{equation}\label{eq.regularization_basic}
			\norm{\bar{P}_\tau P_t-P_t}_{L^2\rightarrow L^2}\leq \sqrt{\frac{C_{\ref{cor.CovSEP}}\tau}{t}}\qquad and\qquad \norm{\bar{P}_\tau\bar{P}_t-\bar{P}_t}_{L^2\rightarrow L^2}\leq \sqrt{\frac{\tau}{t}}.
		\end{equation}
	\end{lemma}
	\begin{proof}
		For any $\tilde{F}\in L^2(\X, \fil, \Pr)$, testing the equation $(\partial_s-\bar{\L})\bar{P}_s\tilde{F}=0$ with $\bar{P}_s\tilde{F}$ over $[0,\tau]$, we obtain that
		\begin{equation*}
			\norm{\bar{P}_\tau\tilde{F}}^2_{L^2}-\norm{\tilde{F}}^2_{L^2}+2\int_{0}^{\tau}\norm{\bar{P}_s\tilde{F}}^2_{\dH^1}\d s=0,
		\end{equation*}
		which implies
		\begin{align*}
			\norm{\bar{P}_{\tau}\tilde{F}}^2_{L^2}\geq \norm{\tilde{F}}^2_{L^2}-2\tau\norm{\tilde{F}}^2_{\dH^1}.
		\end{align*}
		
		
		Then we insert $\tilde{F}=P_t F$, and use \eqref{eq.EnergyLbarUpper} in Corollary~\ref{cor.CovSEP} to obtain
		\begin{equation}\label{eq.PtbarLowBd}
			\begin{aligned}
				\norm{\bar{P}_\tau P_t F}^2_{L^2}&\geq \norm{P_t F}^2_{L^2}-2C_{\ref{cor.CovSEP}}\tau\expec{P_tF(-\L)P_t F}_{\rho}\\
				&\geq \norm{P_t F}^2_{L^2}-\frac{C_{\ref{cor.CovSEP}}\tau}{t}\norm{F}^2_{L^2}.
			\end{aligned}
		\end{equation}
		In the last line, we also use the semigroup property associated with $P_t$, which can be derived similar to \eqref{eq.GtHkDecayd}. We can use this result to obtain the setimate of $\norm{\bar{P_\tau}P_t-P_t}_{L^2\rightarrow L^2}$:
		\begin{align*}
			\norm{(\bar{P}_\tau P_t-P_t)F}^2_{L^2}&=\norm{\bar{P}_\tau P_tF}^2_{L^2}+\norm{P_t F}^{2}_{L^2}-2\norm{\bar{P}_{\tau/2}P_t F}^2_{L^2}\\
			&\leq 2\norm{P_t F}^2_{L^2} - 2\Ll(\norm{P_t F}^2_{L^2}-\frac{C_{\ref{cor.CovSEP}}\tau}{2t}\norm{F}^2_{L^2}\Rr)\\
			&\leq \frac{C_{\ref{cor.CovSEP}}\tau}{t}\norm{F}^2_{L^2}.
		\end{align*}
		Here in the second line, we apply the estimate \eqref{eq.PtbarLowBd} to $\norm{\bar{P}_{\tau/2}P_t F}^2_{L^2}$, and the fact $\norm{\bar{P}_\tau P_t F}^2_{L^2}\leq \norm{P_t F}^2_{L^2}$. A similar argument also works for $\norm{\bar{P}_\tau \bar{P}_t-\bar{P}_t}_{L^2\rightarrow L^2}$
	\end{proof}
	
	Now we prove the main result of this section.
	\begin{proof}[Proof of Proposition~\ref{prop.MixL2HkEs}]
		Recall the dual characterization of $L^2$-norm
		\begin{align}\label{eq.dual_L2}
			\norm{(P_t-\bar{P}_t)F}_{L^2} = \sup_{\norm{G}_{L^2} = 1} \bracket{(P_t-\bar{P}_t)F, G}_{\rho}.
		\end{align}
		Thus, we consider $\expec{(P_t-\bar{P}_t)F,G}_\rho$ for a $G\in L^2(\X, \fil, \Pr)$. We use the reversibility of $P_t$ and $\Pb_t$  under the measure $\P_\rho$, then rearrange the expression, and apply the regularization Lemma \ref{lem.Regularization} to obtain that
		\begin{align*}
			\Ll|\expec{(P_t-\bar{P}_t)F,G}_\rho\Rr|&=\Ll|\expec{F,(P_t-\bar{P}_t)G}_\rho\Rr|\\
			&\leq \Ll|\expec{F,\bar{P}_\tau(P_t-\bar{P}_t)G}_\rho\Rr|+\norm{F}_{L^2}\norm{\bar{P}_\tau(P_t-\bar{P}_t)G-(P_t-\bar{P}_t)G}_{L^2}\\
			&= \Ll|\expec{F,\bar{P}_\tau(P_t-\bar{P}_t)G}_\rho\Rr|+\norm{F}_{L^2}\norm{(\bar{P}_\tau P_t - P_t)G-(\bar{P}_\tau \Pb_t - \Pb_t)G}_{L^2}  \\ 
			&\leq \Ll|\expec{F,\bar{P}_\tau(P_t-\bar{P}_t)G}_\rho\Rr|+2\sqrt{\frac{C\tau}{t}}\norm{F}_{L^2}\norm{G}_{L^2}.
		\end{align*}
		When $\tau\ll t$, the error paid in the second term is very small, and it suffices to consider the term ${\Ll|\expec{F,\bar{P}_\tau(P_t-\bar{P}_t)G}_\rho\Rr|}$. Using the reversibility once again, we get that
		\begin{equation*}
			\expec{F,\bar{P}_\tau(P_t-\bar{P}_t)G}_{\rho}=\expec{(P_t-\bar{P}_t)(\bar{P}_\tau F),G}_\rho.
		\end{equation*}
		We then apply Lemma~\ref{lem.L2linearStats} and Proposition \ref{prop.TwoScaleLinear} to obtain that 
		\begin{align*}
			\Ll|\expec{F,\bar{P}_\tau(P_t-\bar{P}_t)G}_\rho\Rr|
			&\leq \norm{(P_t-\bar{P}_t)(\bar{P}_\tau F)}_{L^2}\norm{G}_{L^2}\\
			&\leq C\left(\Ll(3^{-\alpha m} + 3^m t^{-\frac{5}{8}}\Rr)\norm{\Pi_1 \bar{P}_\tau F}_{L^2}+t^{-\frac{1}{8}}\norm{\bar{P}_\tau F}_{L^2}\right.\\
			&\quad\left.+3^m\norm{\Pi_{1}\bar{P}_\tau F}_{\dH^1} + 3^mt^{\frac{5}{8}}\norm{\Pi_1\bar{P}_\tau F}_{\dH^2}+\norm{\Pi_{\geq 2}\bar{P}_\tau F}_{L^2}\right)\norm{G}_{L^2}.
		\end{align*}
		The operators $\Pi_1$ and $\bar{P}_\tau$ communicate thanks to Lemma~\ref{lem.eqKolmogrov}. Especially, the operator brings an extra decay. By the decay property of the semigroup \eqref{eq.GtHkDecayd}, we obtain that
		\begin{equation*}
			\norm{\Pi_1\bar{P}_\tau F}_{L^2}\leq \norm{\bar{P}_\tau F}_{L^2}\leq \norm{F}_{L^2},
		\end{equation*}
		and
		\begin{align*}
			\norm{\Pi_1\bar{P}_\tau F}_{\dH^1} &=\norm{\bar{P}_{\tau}(\Pi_1 F)}_{\dH^1}\leq C_1 \tau^{-1/2}\norm{\Pi_1 F}_{L^2}\leq C_1 \tau^{-1/2}\norm{F}_{L^2}, \\
			\norm{\Pi_1\bar{P}_\tau F}_{\dH^2} &=\norm{\bar{P}_\tau(\Pi_1 F)}_{\dH^2}\leq C_2\tau^{-1}\norm{\Pi_1 F}_{L^2}\leq C_2\tau^{-1}\norm{F}_{L^2}.
		\end{align*}
		The remaining term can be treated using the faster decay in Proposition \ref{prop.fasterDecay}:
		\begin{equation*}
			\norm{\Pi_{\geq 2}\bar{P}_{\tau}F}^2_{L^2}\leq C \sum_{k=2}^{N}\tau^{-\frac{kd}{2}}||| \Pi_{k}F|||^2_k.
		\end{equation*}
		Combining the above estimates we obtain the following result:
		\begin{align*}
			\Ll|\expec{(P_t-\bar{P}_t)F,G}_\rho\Rr|
			&\leq C\left(\Ll(3^{-\alpha m}+ 3^m t^{-\frac{5}{8}}+t^{-\frac{1}{8}}+3^m \tau^{-\frac{1}{2}}+3^m t^{\frac{5}{8}}\tau^{-1}+2\sqrt{\frac{C\tau}{t}}\Rr)\norm{F}_{L^2}\right.\\
			&\qquad \qquad \left.+\sum_{k=2}^{N}\tau^{-\frac{kd}{4}}||| \Pi_{k} F|||_k\right)\norm{G}_{L^2}.
		\end{align*}
		We put this result back to \eqref{eq.dual_L2}, with a choice of mesoscopic scales 
		\begin{align*}
			1\ll 3^m\ll t^{\frac{5}{8}}\ll \tau\ll t.
		\end{align*}
		For example, we can choose $\tau=t^{\frac{3}{4}}, 3^m \simeq t^{\frac{1}{16}}$, then we obtain the desired result \eqref{eq.MixL2HkEs} with a parameter $\beta :=\frac{\min\{\alpha,1\}}{16}$.
	\end{proof}

	\section{Regularization via spatial mixing}\label{sec.Reg}
	Proposition~\ref{prop.MixL2HkEs} only brings us an extra factor $t^{-2\beta}$ for $\norm{F}_{L^2}$, which is not sufficient to match the result in Proposition~\ref{prop.AvePbarDecay}. To achieve a rate   $o(t^{-\frac{d}{2}})$ in homogenization, we will implement another step of regularization (different from that in Section~\ref{subsec.reg}). This regularization is actually contained in the previous work \cite[Proposition~2.2]{jlqy}.
	
	Let us explain the basic setting of this regularization via spatial mixing. Recall the translation operator $\tau_x$ defined in \eqref{eq.translation2}. For every local function $u$ and non-negative integer $L$, we consider an average over the spatial translations
	\begin{equation*}
		R_L u:= \frac{1}{|\La_L|}\sum_{x\in\La_L}\tau_x u.
	\end{equation*}
	We also have a sequence of parameters $\theta, \delta, \varepsilon$ fixed throughout the section
	\begin{align}\label{eq.parameters_reg}
		\theta > 100, \qquad \delta > 0,  \qquad \varepsilon > 0.
	\end{align}
	Roughly, the parameters $\delta, \varepsilon$ are small, and $\theta$ comes from the spectral gap inequality. Their explicit values are given and explained in \eqref{eq.parameters_choice} and \eqref{eq.conditionTheta}. Given a local function $u$, we then consider a sequence of time and scales (see \eqref{eq.deflu} for the definition of $\ell_u$)
	\begin{align}\label{eq.tn}
		t_0:= \max\{10(1+\ell_u), 2(d+2)\theta\}, \qquad t_n := \theta^n t_0, \qquad K_n:=\lfloor t_n^{(1-\varepsilon)/2} \rfloor,
	\end{align} 
	where $\lfloor \cdot \rfloor$ stands for the integer part. Finally, we denote by $K(t)$ the scale of regularization, which is constant on every interval
	\begin{align}\label{eq.nt}
		\forall t \in [t_n, t_{n+1}), \qquad K(t) := K_n.
	\end{align} 
	Notice that the scale of regularization $K(t)$ depends on the local function $u$, as \eqref{eq.tn} indicated. In order to study $P_t u$ in Theorem~\ref{thm.main}, a better object is its regularized version 
	\begin{align}\label{eq.defPtRtu}
		P_t R_{K(t)}u =   R_{K(t)} P_t u,
	\end{align}
	when the parameters in \eqref{eq.parameters_reg} are well chosen. One can verify easily the equality in \eqref{eq.defPtRtu} as the translation commutes with the semigroup.
	
	This section consists of two main estimates. Firstly, we show that, thanks to the spatial mixing in regularization,  Propositions~\ref{prop.MixL2HkEs} and \ref{prop.AvePbarDecay} can yield \eqref{eq.main} for the regularized function $P_t R_{K(t)}u$.
	\begin{proposition}\label{prop.DecayPRegul}
		There exist well-chosen parameters $(\theta, \delta, \varepsilon)$ depending on $d, \lambda, \r$, such that  the following estimate holds for every local function $u$:
		\begin{equation}\label{eq.DecayPRegul}
			\var_\rho[P_t R_{K(t)}u]=\frac{\tilde u'(\rho)^2 \chi(\rho)}{\sqrt{(8\pi t)^{d} \det [\DD(\rho)]}} + O(t^{-\frac{d+4\delta}{2}}).
		\end{equation}
		Here  $\tilde u' $ follows its definition in Theorem~\ref{thm.main}.
	\end{proposition}
	Secondly, we follow \cite[Sections~3,4]{jlqy} to show that $P_t R_{K(t)}u$ is a proper regularization which provides a good approximation for $P_t u$. A heuristic explanation is that, the range of translation $K_n =\lfloor t_n^{(1-\varepsilon)/2} \rfloor$ is much smaller than the diffusive scale $t^{1/2}$ when $t \in [t_n, t_{n+1})$. 
	\begin{proposition}\label{prop.RegulEff}
		Under the same setting as Proposition~\ref{prop.DecayPRegul}, we have the following estimate for every local function $u$ 
		\begin{equation}\label{eq.RegulEff}
			\var_\rho[P_t(u-R_{K(t)}u)]=O(t^{-\frac{d+2\delta}{2}}).
		\end{equation}
	\end{proposition}
	
	These two propositions will immediately yield Theorem~\ref{thm.main}.
	\begin{proof}[Proof of Theorem~\ref{thm.main}]
		For every two real numbers $x, y$, one can verify the following elementary inequality for every deterministic constant $A>0$ via Cauchy--Schwarz inequality
		\begin{align*}
			(1-A)x^2 + \Ll(1-A^{-1}\Rr)y^2 	\leq (x+y)^2 \leq (1+A)x^2 + \Ll(1+A^{-1}\Rr)y^2,
		\end{align*}
		which implies
		\begin{align*}
			\Ll\vert (x+y)^2 - x^2\Rr\vert \leq A x^2 + \Ll(1+A^{-1}\Rr)y^2.
		\end{align*}
		We then set 
		\begin{align*}
			x = P_t R_{K(t)}u-\expec{u}_\rho, \qquad y= P_t u - P_t R_{K(t)}u, \qquad A=t^{-\frac{\delta}{2}},
		\end{align*}
		and obtain that
		\begin{align*}
			\Ll\vert \var[P_t u] - \var[P_t R_{K(t)}u] \Rr\vert &\leq t^{-\frac{\delta}{2}}\var[P_t R_{K(t)}u]+(1+t^{\frac{\delta}{2}})\var[P_t(u-R_{K(t)})u] \\
			& \leq C \Ll(t^{-\frac{\delta}{2}} \cdot t^{-\frac{d}{2}} + t^{\frac{\delta}{2}}\cdot t^{-\frac{d+2\delta}{2}}\Rr)\\
			&= 2 C  t^{-\frac{d+\delta}{2}}.
		\end{align*}
		In the second line above, we insert the result in Propositions~\ref{prop.DecayPRegul} and ~\ref{prop.RegulEff}. Here the constant $C$ depends on $u$. Taking the estimate  \eqref{eq.DecayPRegul} for $\var[P_t R_{K(t)}u]$, we conclude Theorem~\ref{thm.main}.
	\end{proof}
	
	Propositions~\ref{prop.DecayPRegul} and ~\ref{prop.RegulEff} will be proved respectively in Section~\ref{subsec.DecayPRegul} and Section~\ref{subsec.RegulEff}. Section~\ref{subsec.recapJLQY} justifies a technical estimate from \cite{jlqy} in the setting of non-gradient exclusion.
	
	\subsection{Decay of semigroup after regularization}\label{subsec.DecayPRegul} 
	This subsection is devoted to Proposition~\ref{prop.DecayPRegul}. We first show some basic results for the regularized function $R_{K(t)}u$. Recall that $u \in \F_0(\Lambda_{\ell_u})$ for the local function $u$. 
	\begin{proposition}\label{prop.AveInitPro}
		For the norm $||| \cdot |||$ defined in \eqref{eq.defTripleNorm}, we have that
		\begin{equation}\label{eq.AveInitHk}
			\forall j\in \N, \qquad |||I_j(T_j R_{K(t)}u)|||_j\leq |||I_j(T_j u)|||_j.
		\end{equation}
		Concerning the support, we have the estimate that
		\begin{equation}\label{eq.AveInitSupport}
			{R_{K(t)}u} \in  \F_0\Ll(\La_{\ell_u + t^{\frac{1-\varepsilon}{2}}}\Rr), 
		\end{equation}
		and
		\begin{equation}\label{eq.AveInitSuppL1}				
			\sum_{x\in \Zd}|x||T_1 R_{K(t) }u(x)| \leq \Ll(\ell_u + t^{\frac{1-\varepsilon}{2}}\Rr) |||I_1(T_1 u) |||_{1}.
		\end{equation}
		$ R_{K(t)}u$ also satisfies the estimate of variance
		\begin{equation}\label{eq.AveInitL2}
			\var_{\rho}[R_{K(t)}u]	\leq \theta^{\frac{d}{2}}(2\ell_u)^d t^{-\frac{(1-\varepsilon)d}{2}} \var_{\rho}[u]. 
		\end{equation}
	\end{proposition}
	\begin{proof}
		Since $I_k(T_k\cdot)$ is a linear operator, \eqref{eq.AveInitHk} can be derived directly using triangle inequality. The support estimate of \eqref{eq.AveInitSupport} is obvious. The $\ell^1$ estimate \eqref{eq.AveInitSuppL1} can be derived from case $j=1$ in \eqref{eq.AveInitHk} and the diameter of the support  \eqref{eq.AveInitSupport}. 
		
		The estimate \eqref{eq.AveInitL2} is a result of the spatial cancellation since $u$ is a local function. Recall that $K(t) = K_n$ for $t \in [t_n, t_{n+1})$, we develop the variance as
		\begin{align*}
			\var_{\rho}[R_{K(t)}u] &=\norm{R_{K_n}u-\expec{u}_\rho}^2_{L^2}\\
			&=\frac{1}{|\La_{K_{n}}|^2}\sum_{x,y\in\La_{K_n}}\expec{\Ll(\tau_x u-\expec{u}_\rho\Rr)\Ll(\tau_y u-\expec{u}_\rho\Rr)}_\rho.
		\end{align*}
		Notice that
		\begin{equation*}
			x,y \in \La_{K_{n}}, |x-y|>\ell_u \Longrightarrow \supp(\tau_x u)\cap\supp(\tau_y u)=\emptyset.
		\end{equation*} 
		Then the independence implies that
		\begin{align*}
			\expec{\Ll(\tau_x u-\expec{u}_\rho\Rr)\Ll(\tau_y u-\expec{u}_\rho\Rr)}_\rho = \bracket{\tau_x u-\expec{u}_\rho}_\rho \bracket{\tau_y u-\expec{u}_\rho}_\rho=0.
		\end{align*}
		Therefore, we can make a restriction of covariance on $|x-y| \leq \ell_u$. Then we have
		\begin{align*}
			\norm{R_{K_n}u-\expec{u}_\rho}^2_{L^2} &=\frac{1}{|\La_{K_{n}}|^2}\sum_{x,y\in\La_{K_n},|x-y| \leq \ell_u}\expec{\Ll(\tau_x u-\expec{u}_\rho\Rr)\Ll(\tau_y u-\expec{u}_\rho\Rr)}_\rho\\
			&\leq \frac{1}{|\La_{K_{n}}|^2}\sum_{x,y\in\La_{K_n},|x-y| \leq \ell_u}\norm{\tau_x u-\expec{u}_\rho}_{L^2}\norm{\tau_y u-\expec{u}_\rho}_{L^2}\\
			&\leq  \theta^{\frac{d}{2}}(\ell_u)^d t^{-\frac{(1-\varepsilon)d}{2}}\norm{u-\expec{u}_\rho}^2_{L^2}.
		\end{align*}
		The second line makes use of Cauchy--Schwarz inequality, and the third line is due to the translation invariance of the measure and $t_{n}\in[t/\theta,t]$.
	\end{proof}

	The result \eqref{eq.AveInitL2} is quite important. It says the regularized version can nearly attain a diffusive decay. Actually, using the following choice of parameters
	\begin{align}\label{eq.parameters_choice}
		\varepsilon:=\frac{2\beta}{d} = \frac{\min\{1,\alpha\}}{8d}, \qquad \delta := \frac{\varepsilon}{8}=\frac{\min\{1,\alpha\}}{64d},
	\end{align}
	the homogenization result in Proposition~\ref{prop.MixL2HkEs} can improve for $R_{K(t)}u$.
	
	\begin{lemma}\label{lem.HomoAve}
		Using the choice of parameters \eqref{eq.parameters_choice}, for every local function $u$, we have
		\begin{equation}\label{eq.HomoAve}
			\var_\rho[(P_t-\bar{P}_t)R_{K(t)}u]=O(t^{-\frac{d+8\delta}{2}}).
		\end{equation}
	\end{lemma}
	\begin{proof}
		Taking $F=R_{K(t)}u-\expec{u}_\rho$ in Proposition~\ref{prop.MixL2HkEs}, we obtain
		\begin{align*}
			\var_\rho[(P_t-\bar{P}_t)R_{K(t)}u] & \leq 	C\Ll(t^{-2\beta}\var_\rho[R_{K(t)}u]+\sum_{j=2}^{N}t^{-\frac{3jd}{8}}||| I_j(T_j R_{K(t)}u) |||^2_j\Rr) \\
			&\leq C\Ll((\ell_u)^d t^{-2\beta-\frac{(1-\varepsilon)d}{2}} \var_{\rho}[u]+\sum_{j=2}^{N}t^{-\frac{3jd}{8}}||| I_j(T_j u) |||^2_j\Rr)\\
			&= C\Ll((\ell_u)^d t^{-\frac{d+8\delta}{2}} \var_{\rho}[u]+\sum_{j=2}^{N}t^{-\frac{3jd}{8}}||| I_j(T_j u) |||^2_j\Rr).
		\end{align*}
		From first line to the second line, we apply \eqref{eq.AveInitHk} to the higher-order terms. From the second line to the third line, we make use of the choice of parameters \eqref{eq.parameters_choice} that
		\begin{align*}
			\beta = \frac{d\varepsilon}{2} \geq \frac{\epsilon}{2} = 4 \delta.
		\end{align*}
		Moreover, for $t$ very large, the higher-order terms $j \geq 2$ have a decay of order at least $O(t^{-\frac{3d}{4}})$, which concludes \eqref{eq.HomoAve}.
	\end{proof}

	Proposition~\ref{prop.AvePbarDecay} can also be extended to $R_{K(t)}u$. This requires a slightly more careful treatment as $R_{K(t)}u$ has a growing support.
	
	\begin{lemma}\label{lem.DecayPbarRegul}
		For every local function $u$, we have 
		\begin{align}
			\var_{\rho}[\Pb_t R_{K(t)}u] = \frac{\tilde u' (\rho)^2 \chi(\rho)}{\sqrt{(8\pi t)^{d} \det [\DD(\rho)]}} + O(t^{-\frac{d}{2}-\frac{\varepsilon}{2}}).
		\end{align}
		Here the function $\tilde u' $ has the same definition in Theorem~\ref{thm.main}.
	\end{lemma}
	
	\begin{proof}
		The proof is similar to that of Proposition~\ref{prop.AvePbarDecay}. The only difference is that $R_{K(t)}u$ has a growth support in function of $t$, thus we need to verify carefully the detail of integrability.
		
		\textit{Step~1: chaos expansion.}
		Since $u$ is a local function, there exists $N\in\N$ such that 
		\begin{equation*}
			u\in\bigoplus_{j=0}^{N}\HH_j.
		\end{equation*}
		As the regularization operator keeps the structure of Fock space, we further obtain that
		\begin{equation*}
			R_{K(t)}u\in\bigoplus_{j=0}^{N}\HH_j.
		\end{equation*}
		Using the chaos expansion in Lemma~$\ref{lem.Chaos}$, we have the following identity
		\begin{equation*}
			R_{K(t)}u=\sum_{j=0}^{N}I_j(T_j R_{K(t)}u),
		\end{equation*}
		where $I_j(T_j R_{K(t)}u)\in\HH_j$ is the projection of $R_{K(t)}u$ on $\HH_j$. The linearity of the semigroup $\bar{P}_t$ implies that
		\begin{equation*}
			\bar{P}_t R_{K(t)}u=\sum_{j=0}^{N}\bar{P}_t I_j(T_j R_{K(t)} u).
		\end{equation*}
		The semigroup $\bar{P}_t$ is closed in every $\HH_j$, so $\bar{P}_t I_j(T_j R_{K(t)} u)\in\HH_j$ and we have
		\begin{equation*}
			\bracket{\bar{P}_t R_{K(t)}u}_\rho =\bar{P}_t I_0(T_0 R_{K(t)}u),
		\end{equation*}
		and the orthogonal decomposition over $(\HH_j)_{j \in \N}$ yields that
		\begin{equation}\label{eq.decomL2Pbt}
			\var_\rho[\bar{P}_t R_{K(t)}u]= \bracket{ \left(\sum_{j=1}^{N}\bar{P}_t I_j(T_j R_{K(t)} u)\right)^2}_\rho =\sum_{j=1}^{N} \bracket{\left(\bar{P}_t I_j(T_j R_{K(t)}u)\right)^2}_\rho.
		\end{equation}
		\smallskip
		
		\textit{Step~2: faster decay of higher-order terms.}
		The faster decay of heat kernel in higher dimension Proposition~\ref{prop.fasterDecay} implies
		\begin{equation}\label{eq.fasterL2Pbt}
			\begin{split}
				\forall j \geq 2, \qquad \bracket{\left(\bar{P}_t I_j (T_j R_{K(t)} u)\right)^2}_\rho &\leq  Ct^{-\frac{jd}{2}}|||I_j(T_j R_{K(t)}u)|||^2_{j} \\
				&\leq Ct^{-\frac{jd}{2}}|||I_j(T_j u)|||^2_j.
			\end{split}
		\end{equation}
		The second line comes from  \eqref{eq.AveInitHk} in Propsition \ref{prop.AveInitPro}. 
		Therefore, when $t$ is very large, the main contribution in \eqref{eq.decomL2Pbt} is the term $j=1$.
		\smallskip
		
		\textit{Step~3: identification of the leading order.}
		For the case $j=1$ in \eqref{eq.decomL2Pbt}, we follow the same argument in Step~3 of the proof to Proposition~\ref{prop.AvePbarDecay}, and obtain the explicit solution
		\begin{equation*}
			\bar{P}_t I_1(T_1 R_{K(t)}u)=I_1(e^{t\Delta_Q/2}T_1 R_{K(t)}u)=I_1(\bar{p}_t\ast(T_1 R_{K(t)}u)).
		\end{equation*}
		Here $\ast$ is the discrete convolution defined in \eqref{eq.defConvolution}, and the function $\bar{p}_t = e^{t\Delta_Q/2}$ is defined from Corollary~\ref{cor.CovSEP}. 
		
		The isometric property \eqref{eq.ItoIso} then yields
		\begin{align}\label{eq.ItoL2Pbt}
			\bracket{\left(\bar{P}_t I_1(T_1 R_{K(t)} u)\right)^2}_\rho = \chi(\rho) \norm{\bar{p}_t\ast(T_1 R_{K(t)}u)}_{\ell^2(\Zd)}^2.
		\end{align}
		Concerning the last term, Lemma~\ref{lem.NashZd} applies and \eqref{eq.AveInitSuppL1} in Proposition \ref{prop.AveInitPro} ensures the integrability 
		\begin{equation}\label{eq.GaussionRu}
			\begin{split}
				&\Ll\vert \norm{\bar{p}_t\ast(T_1 R_{K(t)}u)}_{\ell^2(\Zd)} - \frac{\vert \sum_{x \in \Zd} T_1 R_{K(t)}u(x) \vert}{\Ll((4\pi t)^d \det [\cm(\rho)]\Rr)^{\frac{1}{4}}}\Rr\vert \\
				&\leq C t^{-\frac{d+2}{4}} \sum_{x\in \Zd}|x||T_1 R_{K(t) }u(x)| \\
				&\leq  C t^{-\frac{d+2\epsilon}{4}} |||I_1(T_1 u)|||_1.
			\end{split}
		\end{equation}
		Since $R_{K(t)}u$ is a local function, we can show that $T_1 R_{K(t)}u$ is also a local function for any fixed time $t$. Then from the perturbation formula in Lemma~\ref{lem.PertFormu}, we have
		\begin{equation}\label{eq.coeffL2Pbt}
			\sum_{x\in\Zd}T_1 R_{K(t)}u(x)=\frac{\d}{\d\rho} \E_\rho[R_{K(t)}u]=\frac{\d}{\d\rho} \E_\rho[u] =\tilde{u}'(\rho).
		\end{equation}
		Here we use the notation $\tilde{u}(\rho) = \bracket{u}_\rho$ defined in Theorem~\ref{thm.main}.
		
		Combing \eqref{eq.ItoL2Pbt}, \eqref{eq.GaussionRu} and \eqref{eq.coeffL2Pbt} and the definition $\cm(\rho) = 2 \DD(\rho)$, we conclude
		\begin{align}\label{eq.leadingL2Pbt}
			\bracket{\left(\bar{P}_t I_1(T_1 R_{K(t)} u)\right)^2}_\rho =  \frac{\tilde{u}'(\rho)^2  \chi(\rho)}{\sqrt{(8\pi t)^d \det [\DD(\rho)]}}+ O( t^{-\frac{d}{2}-\frac{\varepsilon}{2}}).
		\end{align}
		The estimates \eqref{eq.decomL2Pbt}, \eqref{eq.fasterL2Pbt} and \eqref{eq.leadingL2Pbt} complete the proof of Lemma~\ref{lem.DecayPbarRegul}.
		
	\end{proof}

	\begin{proof}[Proof of Proposition~\ref{prop.DecayPRegul}]
		Combining Lemma~\ref{lem.DecayPbarRegul} and Lemma~\ref{lem.HomoAve}, we can obtain Proposition~\ref{prop.DecayPRegul} under the choice of parameters \eqref{eq.parameters_choice}.
	\end{proof}
	
	We finish this subsection with the following corollary, which will be used in the next subsection. Its proof is quite close to Proposition~\ref{prop.DecayPRegul}.
	\begin{corollary}\label{cor.AveGapDecay}
		Using the choice of parameters \eqref{eq.parameters_choice}, for every local function $u$, the following estimate holds for all $n \in \N$
		\begin{equation}\label{eq.AveGapDecay}
			\var_\rho[P_{t_{n+1}}(R_{K_{n+1}}-R_{K_n})u] = O\Ll((t_{n+1})^{-\frac{d+4\delta}{2}}\Rr).
		\end{equation}
	\end{corollary}
	\begin{proof}
		We apply at first Lemma~\ref{lem.HomoAve}. Then with the price of order $O\Ll( (t_{n+1})^{-\frac{d+4\delta}{2}}\Rr)$, it suffices to study $\var_\rho[\Pb_{t_{n+1}}(R_{K_{n+1}}-R_{K_n})u]$. One can then repeat the proof of Lemma~\ref{lem.DecayPbarRegul}. Notice that the leading term should be 
		\begin{align*}
			\frac{\d}{\d \rho} \Er[(R_{K_{n+1}}-R_{K_n})u],
		\end{align*}
		which disappears because $\expec{(R_{K_{n+1}}-R_{K_n})u}_\rho=0$ for all $\rho \in (0,1)$. This gives us the desired result \eqref{eq.AveGapDecay}.
	\end{proof}

	\subsection{Approximation rate in regularization}\label{subsec.RegulEff}
	This subsection is devoted to Proposition~\ref{prop.RegulEff}, which is the counterpart of \cite[Proposition~2.2]{jlqy}. Its proof relies on Proposition~\ref{prop.OnestepDecay}. We define at first the shorthand notation 
	\begin{align*}
		u_t := P_t u,
	\end{align*}
	and define the gap between the original process and the regularized process by $v_t$
	\begin{align}\label{eq.defvt}
		v_t := u_t-R_{K(t)}u_t.
	\end{align}
	\begin{proposition}\label{prop.OnestepDecay}
		For the gap process $v_t$ and every $t\in[t_n,t_{n+1}),n \in \N_+$, the following inequality holds:
		\begin{equation}\label{eq.OnestepDecay}
			(1+t)^{\frac{d+2}{2}}\expec{v_t^2}_\rho-(1+t_{n})^{\frac{d+2}{2}}\expec{v_{t_{n}}^2}_\rho\leq Ct_n^{1-\varepsilon}\Ll(\log t_{n}\Rr)^{d+1},
		\end{equation}
		where the constant $C$ depends only on $u,d$ and $\theta$.
	\end{proposition}
	We postpone the proof of Proposition~\ref{prop.OnestepDecay} to Section~\ref{subsec.recapJLQY}, which follows that of \cite[Proposition~2.2]{jlqy}. However, as $K(t)$ jumps at $(t_n)_{n \in \N}$, one should pay attention when iterating \eqref{eq.OnestepDecay} for $t \in [0,\infty)$. The previous work omitted it, so we fix this tiny gap in the following paragraph. 
	\begin{proof}[Proof of Proposition~\ref{prop.RegulEff}]
		We aim to establish \eqref{eq.RegulEff} for $(t_n)_{n \in \N}$ at first. The proof can be divided into 3 steps and we use the choice of parameters \eqref{eq.parameters_choice}.
		
		\textit{Step~1: a preliminary version of iteration.}
		For $t \in [t_n, t_{n+1})$, we recall the definition \eqref{eq.nt} that $K(t)=K_n$, then \eqref{eq.OnestepDecay} yields
		\begin{align*}
			(1+t)^{\frac{d+2}{2}}\expec{(u_{t}-R_{K_n}u_{t})^2}_\rho-(1+t_{n})^{\frac{d+2}{2}}\expec{v_{t_{n}}^2}_\rho\leq Ct_n^{1-\varepsilon}\Ll(\log t_{n}\Rr)^{d+1}.
		\end{align*}
		Since the mapping $t \mapsto \norm{u_t}_{L^2}$ is continuous, we send $t$ to $t_{n+1}$ and obtain
		\begin{equation}\label{eq.Onestep}
			(1+t_{n+1})^{\frac{d+2}{2}}\big\langle(\underbrace{u_{t_{n+1}}-R_{K_n}u_{t_{n+1}}}_{=:\tilde{v}_{t_{n+1}}})^2\big\rangle_\rho-(1+t_{n})^{\frac{d+2}{2}}\expec{v_{t_{n}}^2}_\rho\leq Ct_n^{1-\varepsilon}\Ll(\log t_{n}\Rr)^{d+1}.
		\end{equation}
		We also define a variant function $\tilde{v}_{t_{n+1}}$ above. It is different from $v_{t_{n+1}}$, whose regularization has a scale $K_{n+1}$. We aim to estimate the difference between the two functions.
		
		\smallskip
		
		\textit{Step~2: a priori bound.} We show there exists a constant $C$ depending on $u$, such that
		\begin{equation}\label{eq.priorvardecayaux}
			\forall t \geq 0, \qquad (1+t)^{\frac{d}{2}}\expec{v_t^2}_\rho\leq C.
		\end{equation}
		By \eqref{eq.Onestep}, we obtain
		\begin{equation}\label{eq.OnestepTilde}
			(1+t_{n+1})^{\frac{d+2}{2}}\expec{(\tilde{v}_{t_{n+1}})^2}_\rho-(1+t_n)^{\frac{d+2}{2}}\expec{(v_{t_n})^2}_\rho\leq C t_n.
		\end{equation}
		By Cauchy--Schwarz inequality, we can give an upper bound for $v_{t_{n+1}}$:
		\begin{align*}
			\expec{(v_{t_{n+1}})^2}_\rho&=\expec{(u_{t_{n+1}}-R_{K_{n+1}}u_{t_{n+1}})^2}_\rho\\
			&=\expec{\Ll( (u_{t_{n+1}}-R_{K_{n}}u_{t_{n+1}}) + (R_{K_{n}}u_{t_{n+1}}-\bracket{u}_\rho) + (\bracket{u}_\rho- R_{K_{n+1}}u_{t_{n+1}})\Rr)^2}_\rho\\ 
			&\leq 3\expec{(\tilde{v}_{t_{n+1}})^2}_\rho+3\expec{(R_{K_n}u_{t_{n+1}}-\expec{u}_\rho)^2}_\rho + 3\expec{(R_{K_{n+1}}u_{t_{n+1}}-\expec{u}_\rho)^2}_\rho.
		\end{align*}
		By Proposition~\ref{prop.DecayPRegul}, we can derive
		\begin{align*}
			\expec{(R_{K_{n+1}}u_{t_{n+1}}-\expec{u}_\rho)^2}_\rho &=\var_{\rho}[P_{t_{n+1}}R_{K_{n+1}}u]\leq C (t_{n})^{-\frac{d}{2}}.
		\end{align*}
		A similar estimate applies to $R_{K_n}u_{t_{n+1}}$. We use the decay property of the semigroup $P_t$ at first, and then apply Proposition~\ref{prop.DecayPRegul} to obtain 
		\begin{align*}
			\expec{(R_{K_n}u_{t_{n+1}}-\expec{u}_\rho)^2}_\rho &=\var_\rho[P_{t_{n+1}}R_{K_n}u]\\
			&\leq \var_\rho[P_{t_{n}}R_{K_n}u] \\
			&\leq C (t_{n})^{-\frac{d}{2}}.
		\end{align*}
		Combining the above estimates, we obtain
		\begin{align*}
			\expec{(v_{t_{n+1}})^2}_\rho \leq 3\expec{(\tilde{v}_{t_{n+1}})^2}_\rho + C (t_{n})^{-\frac{d}{2}}.
		\end{align*}
		Insert \eqref{eq.OnestepTilde} in this estimate, we get 
		\begin{equation}
			(1+t_{n+1})^{\frac{d+2}{2}}\expec{(v_{t_{n+1}})^2}_{\rho}\leq 3(1+t_{n})^{\frac{d+2}{2}}\expec{(v_{t_n})^2}_\rho+C t_{n+1},
		\end{equation}
		which implies
		\begin{align}\label{eq.onestepiter}
			\expec{(v_{t_{n+1}})^2}_{\rho}\leq 3\Ll(\frac{1+t_{n}}{1+t_{n+1}}\Rr)^{\frac{d+2}{2}}\expec{(v_{t_n})^2}_\rho+C (t_{n+1})^{-\frac{d}{2}}.
		\end{align}
		Since we choose $\theta > 100$ (see \eqref{eq.parameters_reg}), an iteration of \eqref{eq.onestepiter} yields estimate \eqref{eq.priorvardecayaux}. 
		
		\smallskip
		\textit{Step~3: a refined bound.}
		The difference between $v_{t_{n+1}}$ and $\tilde{v}_{t_{n+1}}$ actually has a better estimate. Recall that we have the gap estimate \eqref{eq.AveGapDecay} in Corollary~\ref{cor.AveGapDecay}, which gives us 
		\begin{equation}\label{eq.vvsquare}
			\bracket{(v_{t_{n+1}}-\tilde{v}_{t_{n+1}})^2}_\rho = \expec{(R_{K_{n}}u_{t_{n+1}}-R_{K_{n+1}}u_{t_{n+1}})^2}_\rho\leq C (t_{n+1})^{-\frac{d+4\delta}{2}}.
		\end{equation}
		We can then refine the estimate of second moment
		\begin{equation}\label{eq.accugapes}
			\begin{split}
				\Ll|\expec{(v_{t_{n+1}})^2}_\rho-\expec{(\tilde{v}_{t_{n+1}})^2}_\rho\Rr|			&=\Ll|\bracket{(v_{t_{n+1}}-\tilde{v}_{t_{n+1}})(v_{t_{n+1}}+\tilde{v}_{t_{n+1}})}_\rho\Rr|\\
				&\leq \Ll(\bracket{(v_{t_{n+1}}-\tilde{v}_{t_{n+1}})^2}_\rho\bracket{(v_{t_{n+1}}+\tilde{v}_{t_{n+1}})^2}_\rho\Rr)^{\frac{1}{2}}\\
				&\leq C (t_{n+1})^{-\frac{d+2\delta}{2}}.
			\end{split}
		\end{equation}
		In the second line, the first term relies on the estimate \eqref{eq.vvsquare}, while the second term still needs the \textit{a priori} bound \eqref{eq.priorvardecayaux}.
		
		Combining \eqref{eq.OnestepTilde} and \eqref{eq.accugapes}, we obtain the relation
		\begin{equation}\label{eq.accuonestepiter}
			(1+t_{n+1})^{\frac{d+2}{2}}\expec{(v_{t_{n+1}})^2}_\rho\leq (1+t_{n})^{\frac{d+2}{2}}\expec{(v_{t_n})^2}_\rho+C(t_{n+1})^{1-\delta}
		\end{equation}
		An iteration of \eqref{eq.accuonestepiter} yields estimate \eqref{eq.RegulEff} at $(t_n)_{n \in \N}$. Then with the help of  \eqref{eq.OnestepDecay}, we obtain the estimate \eqref{eq.RegulEff} for all $t \in \R_+$.
	\end{proof}
	
	\subsection{Proof of Proposition~\ref{prop.OnestepDecay}}\label{subsec.recapJLQY}
	In this subsection, we follow the steps in \cite[Sections~3,4]{jlqy} to justify Proposition~\ref{prop.OnestepDecay}, which relies on several lemmas. Since these lemmas are generally robust, we omit their proofs, but recall the statement and reference.
	
	The first lemma is a cutoff estimate \cite[Proposition~3.1]{jlqy}. Roughly, its says the information in $c \log t \sqrt{t}$ is enough to capture the diffusive behavior in $[0,t]$. One can also find its proof in \cite[Section~6]{jlqy}. Its adaptation in \cite[Proposition 3.1]{ccr} and \cite[Theorem~5.1]{gu2020decay} cover the setting in this paper. In the statement, we denote by $A_L F$ the conditional expectation of a function in $L^1(\X, \fil, \Pr)$  given $\fil_{\La_L}$ with $L\in \N_+$ 
	\begin{equation}\label{eq.defAL}
		A_L F:=\E_\rho[F|\fil_{\La_L}].
	\end{equation}
	\begin{lemma}[Cutoff estimate]\label{lem.Cutoff}
		There exists a finite positive constant $C(d,\lambda, \r)$ such that, for all local functions $F$, for every $t\geq 1$ satisfying $F \in \F_0(\Lambda_{\lfloor 3\sqrt{t} \rfloor})$ and every $L\in \N_+$, we have
		\begin{equation*}
			\bracket{(P_t F - A_L P_t F)^2}\leq C e^{-\frac{L}{\sqrt{t}}}\expec{F^2}_\rho.
		\end{equation*}
	\end{lemma}

	The second lemma is a spectral gap inequality. We call $(\ell, L)$ a ``good pair'' if $\ell$ is a mesoscopical  scale such that $q := \frac{2L+1}{2\ell+1} \in \N_+$. Then we consider an enumeration of the set 
	\begin{align}\label{eq.enumeration}
		(2\ell+1)\Zd \cap \Lambda_L :=\{x_1,x_2,\cdots, x_q\},
	\end{align}
	such that $|x_j|\leq |x_k|$ for $j\leq k$. We let the random variable $\mathbf{M}_j(\eta)$ stand for the total number of particles in $\La_\ell(x_j) = x_j+\La_\ell$ 
	\begin{equation}\label{eq.defNj}
		\mathbf{M}_j(\eta):=\sum_{x\in\La_\ell(x_j)}\eta_x,
	\end{equation}
	and also let $\mathbf{M}$ stand for the vector
	\begin{align}\label{eq.defN}
		\mathbf{M} := (\mathbf{M}_1,\cdots,\mathbf{M}_q).
	\end{align}
	Given a function $F$ in $L^1(\X, \fil_{\Lambda_L}, \Pr)$, we denote by $B_{\ell,L} F$ its conditional expectation  given $\mathbf{M}$:
	\begin{equation}\label{eq.defBL}
		B_{\ell,L}F=\E_\rho[F|\mathbf{M}].
	\end{equation}
	\begin{lemma}[Spectral gap inequality]\label{lem.SpectGap}
		There exists a finite positive constant $C_{\ref{lem.SpectGap}}(d)$ such that for every good pair $(\ell, L)$ and $F\in L^2(\X, \fil, \Pr)$, the following estimate holds:
		\begin{equation*}
			\E_{\rho,\La_L}[(F-B_{\ell,L}F)^2|\mathbf{M}]\leq C_{\ref{lem.SpectGap}} \ell^2 \sum_{b \in \Lambda_L^*}\E_{\rho,\La_L}\Ll[ (\pi_b F)^2|\mathbf{M}\Rr].
		\end{equation*}
	\end{lemma}
	Its proof is just a tensorization of the spectral gap inequality in \cite{LuYau} and the constant $C_{\ref{lem.SpectGap}}$ inherits; see \cite[Theorem~3.2]{jlqy}.
	


	The last lemma is a $H^{-1}$ type estimate for a local centered function and its proof can be found in \cite[Lemma 4.3]{jlqy}. For an oriented edge $e=(x,y)$ and a function $F$ on configuration space, we define that
	\begin{equation}\label{eq.defDelta_e}
		\Delta_e F:=\tau_{y}F-\tau_{x}F.
	\end{equation}
	\begin{lemma}\label{lem.H^-1}
		Given a local function $u$. There exists a constant $C=C(u,\rho)$, such that the following estimate holds for every nearest oriented edge $e=(x,y)$ in $\Zd$ and every non-negative function $f$
		\begin{equation*}
			\expec{\Delta_e u\  f}^2_\rho\leq C\expec{\sE_{\rho,\La_{\ell_u+1}(e)}\Ll(\sqrt{f}\Rr)}_\rho\expec{f}_\rho.
		\end{equation*}
		Here we keep the convention $\La_{\ell_u}(e):=\La_{\ell_u}(x)$ for $e=(x,y)$.
	\end{lemma}

	Now we present the proof. Throughout the proof, we will also use a Dirichlet form associated to $\L$ on $\Lambda \subset \Zd$ 
	\begin{align}\label{eq.Dirichlet}
		\sE_{\rho, \Lambda}(F) := \frac{1}{2} \sum_{b \in \Lambda^*} \bracket{c_b (\pi_b F)^2}_{\rho, \Lambda}.
	\end{align}
	Especially, we will use $\sE_{\rho}(F) \equiv \sE_{\rho, \Zd}(F)$ as a shorthand notation, which coincides with $\bracket{F(-\L F)}_\rho$.  We assume that the parameter $\theta$ satisfies
	\begin{align}\label{eq.conditionTheta}
		\theta := \max\{100, C_{\ref{lem.SpectGap}}\}.
	\end{align} 	
	
	\begin{proof}[Proof of Proposition~\ref{prop.OnestepDecay}]
		Notice that $\expec{v_t}_\rho=0$, we have $\var_\rho[v_t]=\expec{v_t^2}_\rho$. For every $t\in [t_n,t_{n+1})$, by differentiation we obtain
		\begin{multline}\label{eq.OnestepDecay_Derivative}
			(1+t)^{\frac{d+2}{2}}\expec{v_t^2}_\rho-(1+t_n)^{\frac{d+2}{2}}\expec{v_{t_n}^2}_\rho\\
			=-2\int_{t_n}^{t}(1+s)^{\frac{2+d}{2}}\sE_{\rho}(v_s)\ \d s+\frac{d+2}{2}\int_{t_n}^{t}(1+s)^{\frac{d}{2}}\expec{v_s^2}_\rho \, \d s.
		\end{multline}
		We use the fact $v_t = P_t(u - R_{K(t)u})$ in \eqref{eq.defvt}, \eqref{eq.defPtRtu}, and $K(t)\equiv K_n$ is constant in the interval $[t_n,t_{n+1})$ as defined in \eqref{eq.nt}.
		
		Since the dynamics is translation invariant, we may replace $v_s$ by $\tau_x v_s$ on the {\rhs} of \eqref{eq.OnestepDecay_Derivative}, which yields
		\begin{equation}\label{eq.diffaux}
			-2\int_{t_n}^{t}(1+s)^{\frac{2+d}{2}}\sE_{\rho}(\tau_x v_s)\ \d s+\frac{d+2}{2}\int_{t_n}^{t}(1+s)^{\frac{d}{2}}\expec{(\tau_x v_s)^2}_\rho\, \d s.
		\end{equation}
		We set $\vert x \vert \leq \sqrt{t_n}$ and the remaining is to develop the {\rhs} in several steps.
		
		\smallskip
		
		\emph{Step~1: cutoff}. 
		For every $L\geq 1$, the second term in \eqref{eq.diffaux} is equal to
		\begin{equation}\label{eq.secdiffaux}
			\frac{d+2}{2}\int_{t_n}^{t}(1+s)^{\frac{d}{2}}\expec{(\tau_x v_s-A_L \tau_x v_s)^2}_\rho\ \d s+\frac{d+2}{2}\int_{t_n}^{t}(1+s)^{\frac{d}{2}}\expec{(A_L \tau_x v_s)^2}_\rho \, \d s.
		\end{equation}
		The cutoff estimate in Lemma~\ref{lem.Cutoff} applies to the first term above because $t_0$ satisfies $\supp(u)\subset \La_{\lfloor\sqrt{t_0}\rfloor}$ and $\supp({\tau_x u})\subset\La_{3\lfloor\sqrt{t}\rfloor}$. For every integer $L\geq \frac{d+2}{2}\sqrt{t_{n+1}}\log t_{n+1}$, we obtain that
		\begin{align*}
			\forall s\in[t_{n},t_{n+1}), \qquad \expec{(\tau_x v_s-A_L \tau_x v_s)^2}_\rho &\leq C s^{-\frac{d+2}{2}}\expec{\Ll(\tau_x (u-R_{K_n}u)\Rr)^2}_\rho, \\
			&\leq 4  C s^{-\frac{d+2}{2}} \expec{u^2}_\rho.
		\end{align*}
		In the last line, we use the translation invariant property and the Cauchy--Schwarz inequality.
		
		Viewing \eqref{eq.secdiffaux} and the previous cutoff estimate, \eqref{eq.diffaux} is bounded above by
		\begin{equation}\label{eq.cutoffes}
			-2\int_{t_n}^{t}(1+s)^{\frac{2+d}{2}}\sE_{\rho}(\tau_x v_s)\ \d s+\frac{d+2}{2}\int_{t_n}^{t}(1+s)^{\frac{d}{2}}\expec{(A_L \tau_x v_s)^2}_\rho \, \d s +C \expec{u^2}_\rho,
		\end{equation}
		for every integer $L$ satisfying $L\geq \frac{d+2}{2}\sqrt{t_{n+1}}\log t_{n+1}$.
		
		\smallskip
		
		\emph{Step~2: spectral gap.}
		Notice the fact that $B_{\ell,L} A_L F=B_{\ell,L}F$ as defined in \eqref{eq.defBL}, we further develop $\expec{(A_L \tau_x v_s)^2}_\rho$ using the spectral gap inequality in Lemma~\ref{lem.SpectGap}, with a choice of parameter 
		\begin{align}\label{eq.choice_ell}
			\ell := \Ll\lfloor \sqrt{\frac{2(1+t_n)}{(d+2)\theta}}\Rr\rfloor.
		\end{align}
		We also choose $L$ to be the smallest one satisfying $L\geq \frac{d+2}{2}\sqrt{t_{n+1}}\log (t_{n+1})$ such that $(2L+1)/(2\ell+1)$ is an integer. Hence, we obtain that
		\begin{equation}\label{eq.specdecom}
			\begin{split}
				&-2\int_{t_n}^{t}(1+s)^{\frac{2+d}{2}}\sE_{\rho}(\tau_x v_s) \, \d s+\frac{d+2}{2}\int_{t_n}^{t}(1+s)^{\frac{d}{2}}\expec{(A_L \tau_x v_s)^2}_\rho \, \d s +C \expec{u^2}_\rho\\
				&\leq -2\int_{t_n}^{t}(1+s)^{\frac{2+d}{2}}\sE_{\rho}(\tau_x v_s) \, \d s+\frac{d+2}{2}\int_{t_n}^{t}(1+s)^{\frac{d}{2}}\Ll(\expec{(B_{\ell,L} \tau_x v_s)^2}_\rho + \theta \ell^2 \sE_{\rho}(\tau_x v_s)\Rr)  \, \d s \\
				& \qquad + C \expec{u^2}_\rho  \\ 
				&\leq \frac{d+2}{2}\int_{t_n}^{t}(1+s)^{\frac{d}{2}} \expec{(B_{\ell,L} \tau_x v_s)^2}_\rho  \, \d s + C \expec{u^2}_\rho.
			\end{split}
		\end{equation}
		In the second line, the choice of $\ell$ in \eqref{eq.choice_ell} ensures that $\frac{d+2}{2} \theta \ell^2 (1+s)^{\frac{d}{2}} \leq (1+s)^{\frac{d+2}{2}}$ for all $s \in [t_n, t_{n+1})$, so the sum of Dirichlet energy is negative.
		
		Combining \eqref{eq.cutoffes} and \eqref{eq.specdecom}, we obtain the following estimate for every local function $u$, (with $C$ depending on $u$)
		\begin{multline}\label{eq.cutspaces}
			(1+t)^{\frac{d+2}{2}}\expec{v_t^2}_\rho-(1+t_n)^{\frac{d+2}{2}}\expec{v_{t_n}^2}_\rho\\
			\leq \frac{d+2}{2}\int_{t_n}^{t}(1+s)^{\frac{d}{2}}\expec{(B_{\ell,L}\tau_{x}v_s)^2}_\rho\ \d s+C.
		\end{multline}
		
		\smallskip
		\emph{Step~3: spatial average.} Since the previous formula holds for all $|x|\leq\sqrt{t_n}$, we may average it in space to obtain another bound for the left-hand side of \eqref{eq.cutspaces} 
		\begin{equation*}
			\frac{d+2}{2}\int_{t_n}^{t}\frac{1}{|\La_{\ell}|}\sum_{x\in\La_\ell}(1+s)^{\frac{d}{2}}\expec{(B_{\ell,L}\tau_x v_s)^2}_\rho\ \d s+C.
		\end{equation*}
		
		We then develop $\expec{(B_{\ell,L}\tau_x v_s)^2}_\rho$. Recall that $s\in [t_n,t_{n+1})$ and the expression of $v_s$ in \eqref{eq.defvt}
		\begin{equation*}
			v_s=u_s-R_{K_n}u_s=\frac{1}{|\La_{K_n}|}\sum_{y\in\La_{K_n}}(u_s-\tau_y u_s).
		\end{equation*}
		Using Cauchy--Schwarz inequality, we obtain
		\begin{equation*}
			\expec{(B_{\ell,L}\tau_x v_s)^2}_\rho\leq\frac{1}{|\La_{K_n}|}\sum_{y\in\La_{K_n}}\expec{(B_{\ell,L}\tau_x(u_s-\tau_y u_s))^2}_\rho.
		\end{equation*}
		For every $y\in\La_{K_n}$, there exists a canonical path from $0$ to $y$ which consists of $\vert y \vert_1$ nearest neighbor steps. Hence, we can define $\gamma_y:=(e_1,\cdots,e_{\vert y \vert_1})$ where each $e_i=(y_{i-1},y_i)$ is an nearest oriented edge in $\La_{K_n}$. Using the definition \eqref{eq.defDelta_e} and the Cauchy--Schwarz inequality, we have that
		\begin{equation*}
			\expec{(B_{\ell,L}\tau_{x}(u_s-\tau_{y}u_s))^2}_\rho\leq|\gamma_y|\sum_{e\in\gamma_y}\expec{(B_{\ell,L}\tau_x(\Delta_e u_s))^2}_\rho.
		\end{equation*}
		
		Combining the above computation, we obtain the following estimate:
		\begin{multline}\label{eq.spaceaver}
			(1+t)^{\frac{d+2}{2}}\expec{v_t^2}_\rho-(1+t_n)^{\frac{d+2}{2}}\expec{v_{t_n}^2}_\rho\\
			\leq \frac{d+2}{2}\int_{t_n}^{t}\Ll((1+s)^{\frac{d}{2}}\sum_{x\in\La_{\ell},y\in\La_{K_n}}\frac{1}{|\La_{\ell}|}\frac{1}{|\La_{K_n}|}|\gamma_y|\sum_{e\in\gamma_y}\expec{(B_{\ell,L}\tau_x(\Delta_e u_s))^2}_\rho \Rr)\, \d s + C.
		\end{multline}
		
		\smallskip
		\emph{Step~4: entropy estimates.} We aim to give an upper bound for the {\rhs} of \eqref{eq.spaceaver}. Let $\mathbb{M}_q:=\{0,1,\cdots,|\La_\ell|\}^q$ stand for the value space of $\mathbf{M}$. Then we calculate the Radon--Nikodym derivative for all $M\in\mathbb{M}_q$
		\begin{align}\label{eq.derRN}
			f^{M}(\eta):=\frac{\1_{\{\mathbf{M}(\eta) = M\}}}{\P_\rho(\mathbf{M}(\eta) = M)}.
		\end{align} 
		We also define the following shorthand notation 		
		\begin{equation*}
			\forall x\in\Zd, t \geq 0, \qquad f_{x,t}^{M}:=\tau_x P_t f^{M}.
		\end{equation*}
		Then, by the reversibility and the translation invariance, we have
		\begin{align*}
			\forall M\in\mathbb{M}_q, \qquad (B_{\ell,L}\tau_x P_s u)(M) &=\E_{\rho}[\tau_x P_s u|\mathbf{M} = M]\\
			&=\int_{\X} \tau_x P_s u f^{M} \P_\rho(\d \eta)\\
			&=\expec{u f_{-x,s}^{M}}_\rho.
		\end{align*}
		Therefore, we obtain the following expression
		\begin{equation}\label{eq.compucond}
			\expec{(B_{\ell,L}\tau_{x}(\Delta_e u_s))^2}_\rho=\sum_{M\in\mathbb{M}_q}\P_\rho(\mathbf{M}=M)\expec{\Delta_e u\ f_{-x,s}^{M}}^2_\rho.
		\end{equation}
		We insert \eqref{eq.compucond} in \eqref{eq.spaceaver} given $M\in\mathbb{M}_q$, and obtain that
		\begin{equation}\label{eq.comcanens}
			\begin{aligned}
				& \sum_{x\in\La_{\ell}}\sum_{y\in\La_{K_n}}\sum_{e\in\gamma_y}\frac{|\gamma_y|}{|\La_{\ell}||\La_{K_n}|}\expec{\Delta_e u\ f_{-x,s}^{M}}^2_\rho \\
				&\leq C\sum_{x\in\La_{\ell}}\sum_{y\in\La_{K_n}}\sum_{e\in\gamma_y}\frac{d K_n}{|\La_{\ell}||\La_{K_n}|}\expec{\sE_{\rho,\La_{\ell_u+1}(e)}\Ll(\sqrt{f_{-x,s}^{M}}\Rr)}_\rho\\
				&\leq C\sum_{x\in\La_{\ell}}\sum_{y\in\La_{K_n}}\sum_{e\in\gamma_y}\frac{d K_n}{|\La_{\ell}||\La_{K_n}|}\expec{\sE_{\rho,x+\La_{\ell_u+1}(e)}\Ll(\sqrt{f_{s}^{M}}\Rr)}_\rho\\
				&\leq \frac{C K_n^2}{\ell^d}\sE_{\rho}\Ll(\sqrt{f_s^{M}}\Rr).
			\end{aligned}
		\end{equation}
		The first line applies Lemma~\ref{lem.H^-1} and the fact that $|\gamma_y|=\vert y \vert_1\leq d K_n$. 	
		In the second inequality, we use the fact that $\expec{\sE_{\rho,\La}(\tau_x F)}_\rho=\expec{\sE_{\rho,-x+\La}(F)}_\rho$ for every $x$ and $F$, due to the translation invariant property. The last inequality follows from an explicit computation: the denominator $|\La_{K_n}|$ cancels with the summation in $y$, and one edge in the energy appears at most $|\La_{\ell_u+1}(e)|d K_n=C K_n$ times in the summation of $x$ and $e$.

		The last line of \eqref{eq.comcanens} reminds us of the entropy. For every density function $f$ in $(\X, \fil, \Pr)$, we define its entropy as 
		\begin{align*}
			\operatorname{Ent}[f] := \Er[f \log f].
		\end{align*} 
		Then it is well known that 
		\begin{align*}
			\partial_s \operatorname{Ent}[P_s f] \leq-4\sE_{\rho}\Ll(\sqrt{P_s f}\Rr).
		\end{align*}
		Hence, $\operatorname{Ent}[P_s f]$ is decreasing. We apply the formula above to $f_{s}^{M} = P_s f^M$ and get
		\begin{multline}\label{eq.comint}
			\int_{t_n}^{t}4\sE_{\rho}\Ll(\sqrt{f_{s}^{M}}\Rr)\ \d s\leq\operatorname{Ent}\Ll[f_{t_n}^{M}\Rr]-\operatorname{Ent}\Ll[f_{t}^{M}\Rr]\leq\operatorname{Ent}\Ll[f^{M}\Rr]\\
			= \int_{\X}   \frac{\1_{\{\mathbf{M}(\eta) = M\}}}{\P_\rho(\mathbf{M} = M)} \log\Ll(\frac{1}{\P_{\rho}\Ll(\mathbf{M}=M\Rr)}\Rr) \, \Pr(\eta) = \log\Ll(\frac{1}{\P_{\rho}\Ll(\mathbf{M}=M\Rr)}\Rr).
		\end{multline}
		Since $x \mapsto x\log (1/x)$ is concave, by Jensen's inequality and \eqref{eq.choice_ell}, we have
		\begin{multline}\label{eq.coment}
			\sum_{M\in\mathbb{M}_q}\P_{\rho}(\mathbf{M}=M)\log \Ll(\frac{1}{\P_{\rho}(\mathbf{M}=M)}\Rr) \\
			\leq \log(|\mathbb{M}_q|)\leq C \Ll(\frac{L}{\ell}\Rr)^d \log(\ell)\leq C (\log t_n)^{d+1},
		\end{multline}
		where $C$ is a constant that depends only on $d$ and $\theta$.
		
		Combining  \eqref{eq.compucond}, \eqref{eq.comcanens}, \eqref{eq.comint}, \eqref{eq.coment}, and the fact $s \simeq \ell^2 \simeq t_n$ when $s\in[t_n,t_{n+1})$  together with $K_n\leq t_n^{(1-\varepsilon)/2}$ (see \eqref{eq.choice_ell} and \eqref{eq.tn}), we obtain the following estimate
		\begin{equation}\label{eq.canenes}
			\begin{split}
				&\int_{t_n}^{t}(1+s)^{\frac{d}{2}}\sum_{x\in\La_{\ell},y\in\La_{K_n}}\frac{1}{|\La_{\ell}|}\frac{1}{|\La_{K_n}|}|\gamma_y|\sum_{e\in\gamma_y}\expec{(B_{\ell,L}\tau_x(\Delta_e u_s))^2}_\rho\ \d s\\
				&\leq\sum_{M\in\mathbb{M}_q}\P_\rho(\mathbf{M}=M)\int_{t_{n}}^{t}(1+s)^{\frac{d}{2}}\frac{C K_n^2}{\ell^d}\sE_{\rho}\Ll(\sqrt{f_s^{M}}\Rr)\ \d s\\
				&\leq Ct_n^{1-\varepsilon}\Ll(\log t_n\Rr)^{d+1}.
			\end{split}
		\end{equation}
		The desired estimate \eqref{eq.OnestepDecay} then follows from \eqref{eq.spaceaver} and \eqref{eq.canenes}.
	\end{proof}

	\appendix

	\section{Sobolev norms}\label{sec.Sobolev}
	\begin{proof}[Proof of Lemma~\ref{lem.elementary}]
		\begin{align*}
			\partial_s \bracket{\bar{G}_s (-\Lb)^{k} \bar{G}_s}_\rho = -2 \bracket{\bar{G}_s (-\Lb)^{k+1} \bar{G}_s}_\rho.
		\end{align*}
		We do integration from $\tau$ to $t$ to obtain
		\begin{align*}
			\bracket{\bar{G}_t (-\Lb)^{k} \bar{G}_t}_\rho - \bracket{\bar{G}_\tau (-\Lb)^{k} \bar{G}_\tau}_\rho = -2 \int_\tau^t \bracket{\bar{G}_s (-\Lb)^{k+1} \bar{G}_s}_\rho \, \d s.
		\end{align*}
		This concludes \eqref{eq.GtHkDecay}.
		
		To study \eqref{eq.GtHkDecayd}, we use the fact that the mapping $t \mapsto \norm{\bar{G}_t}_{\dH^k}$ is decreasing via spectral analysis; see \cite[Lemma~1.3.2]{fukushima2010dirichlet}. Then \eqref{eq.GtHkDecay} leads to 
		\begin{equation*}
			(t-\tau)\norm{\bar{G}_t}^2_{\dH^{k+1}}\leq \int_{\tau}^{t}\norm{\bar{G}_s}^2_{\dH^{k+1}}\d s\leq \frac{1}{2}\norm{\bar{G}_\tau}^2_{\dH^{k}},
		\end{equation*}
		and we thus have 
		\begin{equation*}
			\norm{\bar{G}_t}_{\dH^{k+1}} \leq \frac{1}{(t - \tau)^{\frac{1}{2}}} \norm{\bar{G}_\tau}_{\dH^{k+1}}.
		\end{equation*}
		We thus conclude \eqref{eq.GtHkDecayd}.

		Concerning \eqref{eq.GgH1}, we have 
		\begin{align*}
			\norm{\Pi_{1} \bar{G}_t}^2_{\dH^{1}} & = \frac{1}{4} \sum_{x \in \Zd} \sum_{y \in \Zd} Q_{y} \bracket{\Ll(\D_y g_t(x) (\eta_x - \eta_{x+y})\Rr)^2}_{\rho} \\
			& \geq \frac{\chi(\rho)}{16} \sum_{x \in \Zd} \sum_{i=1}^d \vert \D_{e_i} g_t(x) \vert^2 \\
			&=\frac{\chi(\rho)}{16}   \sum_{i=1}^d \norm{ \D_{e_i} g_t}^2_{\ell^2(\Zd)}.
		\end{align*}
		In the second line, we use the property (1) in  Lemma~\ref{lem.CovRW}.
		
		The estimate  \eqref{eq.GgH2} strongly relies on the fact that we can close the generator in $\HH_1$ like \eqref{eq.Delta_Q_Closed}
		\begin{align*}
			\Lb \Ll(\Pi_{1} \bar{G}_t \Rr) = I_1 \Ll(\frac{1}{2} \Delta_Q g_t \Rr).
		\end{align*}
		Then we calculate the $\dH^2$-norm via the isometric property
		\begin{align*}
			\norm{\Pi_{1} \bar{G}_t}^2_{\dH^{2}} &= \chi(\rho) \norm{\frac{1}{2} \Delta_Q g_t}^2_{\ell^2(\Zd)} \\
			&= \frac{\chi(\rho)}{4}\bracket{\sum_{h \in \Zd} Q_h \D^*_h \D_h g_t, \sum_{h \in \Zd} Q_h \D^*_h \D_h g_t}_{\ell^2(\Zd)} \\
			&= \frac{\chi(\rho)}{4} \sum_{h,h'\in \Zd} Q_h Q_{h'} \norm{\D_{h'} \D_h g_t}^2_{\ell^2(\Zd)}\\
			&\geq \frac{\chi(\rho)}{64} \sum_{i,j=1}^d  \norm{\D_{e_i} \D_{e_j} g_t}^2_{\ell^2(\Zd)}.
		\end{align*}
		We also use the property (1) of Lemma~\ref{lem.CovRW} in the last line.
	\end{proof}
	
	\section{Generalized Nash inequality}\label{sec.Nash}
	For the convenience of readers, we reformulate the proof of the generalized Nash inequality in \cite[Section 6]{berzeg}.
	\begin{proposition}[Generalized Nash inequality]
		For every function $f\in\tilde{\HH}_n$, the following estimate holds
		\begin{equation*}
			\bracket{f^2}_{\rho}\leq C\chi(\rho)^{n(1-\alpha_n)}\sEb_\rho(f)^{\alpha_n}|||f|||^{2(1-\alpha_n)}_n,
		\end{equation*}
		where $\alpha_n=\frac{nd}{2+nd}$ and $C$ is a constant which depends only on $n$ and $d$.
	\end{proposition}
	\begin{proof}
		In the proof, we make use of the notations in Section~\ref{subsec.Fock}, and we will  introduce some more notations. By \eqref{eq.EnergyLbarLower} in Corollary~\ref{cor.CovSEP} we may consider $\sEb_\rho$ here is the Dirichlet energy with respect to a SSEP. Recall that $\La_\ell=\{-\ell,\cdots,\ell\}^d$ is a cube of side length $2\ell+1$. We consider an enumeration of the set $(2\ell+1)\Zd$: $(2\ell+1)\Zd=\{x_1,x_2,\cdots\}$ such that $|x_j|\leq |x_k|$ for $j\leq k$. We also define the shorthand notation
		\begin{align}\label{eq.defLambdak}
			\La^{k} \equiv \La_{\ell}(x_k)=x_k+\La_\ell.
		\end{align}

		We define a family $\{E^{k}\}^{\infty}_{0}$ of conditional expectations:
		\begin{equation}\label{eq.defEk}
			\forall k \in \N, \qquad E^{k}f:=\E_\rho[f|\fil_{(\cup_{j\leq k}\La^j)^c}] =\E_{\rho, \cup_{j\leq k}\La^j}[f].
		\end{equation}
		It yields a structure of martingale, so we define the martingale difference $\{\Delta^{k}\}_{1}^{\infty}$ as 
		\begin{equation}\label{eq.defDk}
			\Delta^{k} f:=E^{k-1}f-E^{k} f.
		\end{equation}
		We denote by $\mathbf{M}^k$ the number of particles in the cube $\Lambda^k$
		\begin{align*}
			\mathbf{M}^k := \sum_{x\in\La^k}\eta_x .
		\end{align*}
		By the spectral gap inequality under canonical ensemble \cite{LuYau}, we have:
		\begin{equation}\label{ae2}
			\bracket{\Ll(f-\E_{\rho, \La^k}[f|\mathbf{M}^k]\Rr)^2}_{\rho,\La^k}\leq C \ell^2 \sEb_{\rho,\La^k}(f),
		\end{equation}
		with a constant $C$ depending on $d$, and the finite volume Dirichlet form associated to SSEP defined as follows
		\begin{equation}\label{eq.DiriSSEP}
			\sEb_{\rho,\La}(f):=\frac{1}{2}\sum_{b \in\La^*}\bracket{(\pi_bf)^2}_{\rho,\La}.
		\end{equation}

		\textit{Step~1: the spectral gap inequality.}
		For every $f\in\HH_{n}$ with $n\geq 1$, we have $\bracket{f}_\rho=0$. Using the orthogonal decomposition given by \eqref{eq.defEk} and \eqref{eq.defDk}, we have
		\begin{equation}\label{ae3}
			\begin{split}
				\bracket{f^2}_\rho &=\sum_{k=1}^{\infty}\bracket{\Ll(\Delta^{k} f\Rr)^2}_\rho\\
				&=\sum_{k=1}^{\infty}\bracket{\Ll(\Delta^{k} f-E^k[\Delta^{k}f|\mathbf{M}^k ]\Rr)^2}_\rho+\sum_{k=1}^{\infty}\bracket{\Ll(E^k[\Delta^{k}f|\mathbf{M}^k ]\Rr)^2}_\rho\\
				&\leq  \sum_{k=1}^{\infty} C \ell^2 \bracket{\sEb_{\rho,\La^k}\Ll(\Delta^{k} f\Rr)}_\rho + \sum_{k=1}^{\infty}\bracket{\Ll(E^k[\Delta^{k}f|\mathbf{M}^k ]\Rr)^2}_\rho.
			\end{split}
		\end{equation}
		In the last line, we make use of the spectral gap inequality \eqref{ae2}.
		
		We then estimate $\bracket{\Ll(E^k[\Delta^{k}f|\mathbf{M}^k ]\Rr)^2}_\rho$. Notice the following projection
		\begin{align*}
			E^k[\Delta^{k}f|\mathbf{M}^k] = \sum_{s=0}^{\vert \La^k \vert} a_s \Ll(\sum_{Y\in\K_s(\La^k)}\bar{\eta}_{Y}\Rr),
		\end{align*}
		because $\Ll(\sum_{Y\in\K_s(\La^k)}\bar{\eta}_{Y}\Rr)$ are eigenvectors under $	E^k[\cdot|\mathbf{M}^k]$. Moreover, because these vectors are orthogonal under $E^k$, one can derive the following expansion
		\begin{align}\label{ae5}
			\bracket{\Ll(E^k[\Delta^{k}f|\mathbf{M}^k ]\Rr)^2}_\rho =\sum_{s=1}^{n}\frac{1}{|\mcl{K}_s(\La_\ell)|\chi(\rho)^s}\bracket{E^k\Ll[\sum_{Y\in\K_s(\La^k)}\bar{\eta}_{Y}\cdot \Delta^{k} f\Rr]^2}_\rho.
		\end{align}
		In the expression above, the case $s=0$ does not contribute, and the truncation at level $n$ is due to the fact $f \in \HH_{n}$. Using the Glauber derivatives defined in \eqref{eq.Glauber} and \eqref{eq.IPP_Ber}, we can reformulate $E^k\Ll[\sum_{Y\in\K_s(\La^k)}\bar{\eta}_{Y}\cdot \Delta^{k} f\Rr]$ as 
		\begin{align*}
			E^k\Ll[\sum_{Y\in\K_s(\La^k)}\bar{\eta}_{Y}\cdot \Delta^{k} f\Rr]&=\chi(\rho)^{s}\sum_{Y\in\K_s(\La^k)}E^k\Ll[D_Y \Delta^{k} f\Rr]\\
			&=\chi(\rho)^{s}\sum_{Y\in\K_s(\La^k)}E^k\Ll[D_Y f\Rr].
		\end{align*}
		The second line comes from the fact that $E^k\Ll[D_Y E^{k-1}f\Rr] = E^k\Ll[D_Y f\Rr]$ and ${E^k\Ll[D_Y E^{k}f\Rr]=0}$. For convenience, we define a mapping $\varphi^{k}_s : \HH_n \to \HH_{n-s}$ for all $1 \leq s \leq n$ that
		\begin{equation}\label{eq.defphi_s_k}
			\varphi^{k}_s f:= \sum_{Y\in\K_s(\La^k)} E^{k}[D_Y f].
		\end{equation}
		One should keep in mind that the superscript ``$k$'' indicates the $k$-th cube for the expectation and the Glauber derivative, while the subscript ``$s$'' indicates the order of the Glauber derivative.  Using this notation and \eqref{ae5}, then \eqref{ae3} can be reformulated as 
		\begin{equation}\label{ae6}
			\bracket{f^2}_\rho\leq  \sum_{k=1}^{\infty} C \ell^2 \bracket{\sEb_{\rho,\La^k}\Ll(\Delta^{k} f\Rr)}_\rho +\sum_{k=1}^{\infty}\sum_{s=1}^{n}|\K_s(\La_\ell)|^{-1}\chi(\rho)^{s}\bracket{\Ll(\varphi^{k}_s f\Rr)^2}_\rho.
		\end{equation}
		
		\smallskip
		
		\textit{Step~2: iteration.}
		Notice that $E^{k} \varphi^{k}_{s}f=\varphi^{k}_{s}f$, we obtain a decomposition using \eqref{eq.defDk} 
		\begin{equation*}
			\varphi^{k}_s f=\sum_{k'=k+1}^{\infty}\Delta^{k'}\varphi^{k}_s f.
		\end{equation*}
		We then apply the same argument in inequality \eqref{ae6} to $\bracket{\Ll(\varphi^{k}_s f\Rr)^2}_\rho$ for all $1 \leq s \leq n$: 
		\begin{multline*}
			\bracket{\Ll(\varphi^{k}_s f\Rr)^2}_\rho\leq C\ell^2\sum_{k'>k}\bracket{\sEb_{\rho,\La^{k'}}\Ll(\Delta^{k'}\varphi^{k}_s f\Rr)}_\rho\\
			+\sum_{k'>k}\sum_{s'=1}^{n-s}|\K_{s'}(\La_\ell)|^{-1}\chi(\rho)^{s'}\bracket{\Ll(\varphi^{k'}_{s'}\varphi^{k}_s f\Rr)^2}_\rho.
		\end{multline*}
		We iterate the above process in  \eqref{ae6} and obtain the following estimate:
		\begin{equation}\label{ae7}
			\bracket{f^2}_\rho\leq C \ell^2(\mathbf{I})+(\mathbf{II}),
		\end{equation}
		where
		\begin{align*}
			(\mathbf{I}) &:=\sum_{m=1}^{n}\sum_{k_m>\cdots>k_1}\sum_{s_1+\cdots+s_{m-1}<n}\prod_{i=1}^{m-1} \Ll(\frac{\chi(\rho)^{s_i}}{|\K_{s_{i}}(\La_\ell)|}\Rr)
			\bracket{\sEb_{\rho,\La^{k_m}}\Ll(\Delta^{k_m} \varphi^{k_{m-1}}_{s_{m-1}}\cdots \varphi^{k_1}_{s_1}f\Rr)}_\rho, \\
			(\mathbf{II}) &:=\sum_{m=1}^{n}\sum_{k_m>\cdots>k_1}\sum_{s_1+\cdots+s_m=n}\prod_{i=1}^{m}\Ll(\frac{\chi(\rho)^{s_i}}{|\K_{s_{i}}(\La_\ell)|}\Rr)\bracket{\Ll(\varphi^{k_m}_{s_m}\cdots \varphi^{k_1}_{s_1}f\Rr)^2}_\rho.
		\end{align*}
		Actually, the iteration will stop once the term meets the Dirichlet energy operator, or the operators $\varphi^{k}_{s}$ project it into $\HH_0$.
		
		\smallskip
		\textit{Step~3: closing the equation.}
		By Cauchy--Schwarz inequality, Jensen's inequality, and the definition of $\varphi^{k}_{s}$ in \eqref{eq.defphi_s_k}, we have
		\begin{multline*}
			\bracket{\sEb_{\rho,\La^{k_m}}\Ll(\Delta^{k_m} \varphi^{k_{m-1}}_{s_{m-1}}\cdots \varphi^{k_1}_{s_1}f\Rr)}_\rho\\
			\leq \prod_{i=1}^{m-1}\Ll|\K_{s_i}(\La_\ell)\Rr|\cdot\sum_{\substack{ Y_i\in\K_{s_i}(\Lambda^{k_i}) \\ 1 \leq i \leq m-1}}\expec{\sEb_{\rho,\La^{k_m}}\Ll(\Delta^{k_m}D_{Y_{m-1}}\cdots D_{Y_1}f\Rr)}_\rho.
		\end{multline*}
		Therefore, the volume factor $\Ll|\K_{s_i}(\La_\ell)\Rr|$ actually will cancel. We arrange $(\mathbf{I})$ with chaos expansion formula, which will yield the following bound
		\begin{equation}\label{eq.AppL^2}
			(\mathbf{I})\leq \sum_{k=1}^{\infty}\sum_{s=0}^{n-1}\sum_{Y\in\K_{s}(\cup_{j\leq k-1}\La^{j})}\chi(\rho)^s\expec{\sEb_{\rho,\La^k}\Ll(\Delta^{k} D_{Y}f\Rr)}_\rho\leq \sEb_{\rho}(f).
		\end{equation}
		
		Concerning $(\mathbf{II})$, we notice that there exists a constant $C$ such that, the following inequality holds for all  $\ell\geq n$: 
		\begin{equation*}
			\forall s_1+\cdots+s_m=n, \qquad \prod_{i=1}^{m}\Ll(\frac{\chi(\rho)^{s_i}}{|\K_{s_{i}}(\La_\ell)|}\Rr)\leq \frac{C n! \chi(\rho)^n}{|\La_\ell|^{n}} \leq C \ell^{-nd}n! \chi(\rho)^n.
		\end{equation*}
		
		Since $\varphi^{k_m}_{s_m}\cdots \varphi^{k_1}_{s_1}f\in\HH_0$ is a constant when $\sum_{i=1}^{m}s_i=n$, we can estimate
		\begin{align*}
			\bracket{\Ll(\varphi^{k_m}_{s_m}\cdots \varphi^{k_1}_{s_1}f\Rr)^2}_\rho
			&=\left(\prod_{i=1}^{m}\Ll(\sum_{Y_i\in\K_{s_i}(\La^{k_i})}D_{Y_i}\Rr)f\right)^2\\
			&\leq\left(\sum_{Y_i\in\K_{s_i}(\La^{k_i})}|D_{Y_m}\cdots D_{Y_1}f|\right)^2.
		\end{align*}
		Notice that we have the following inequality
		\begin{multline*}
			\sum_{m=1}^{n}\sum_{k_m>\cdots>k_1}\sum_{s_1+\cdots+s_m=n}\left(\sum_{Y_i\in\K_{s_i}(\La^{k_i})}|D_{Y_m}\cdots D_{Y_1}f|\right)^2\\
			\leq \left(\sum_{m=1}^{n}\sum_{k_m>\cdots>k_1}\sum_{s_1+\cdots+s_m=n}\sum_{Y_i\in\K_{s_i}(\La^{k_i})}|D_{Y_m}\cdots D_{Y_1}f|\right)^2,
		\end{multline*}
		and the {\rhs} is actually $|||f|||^2_n$ by \eqref{eq.defTripleNorm}. Thus we obtain an upper bound for $(\mathbf{II})$:
		\begin{equation}\label{eq.AppL^1}
			(\mathbf{II})\leq C\ell^{-nd}n!\chi(\rho)^n|||f|||^2_n.
		\end{equation}
		
		Using \eqref{eq.AppL^2} and \eqref{eq.AppL^1}, the inequality \eqref{ae7} becomes
		\begin{equation}\label{eq.AppL^1L^2}
			\bracket{f^2}_\rho\leq C \ell^2 \sEb_{\rho}(f)+C \ell^{-nd} n! \chi(\rho)^n |||f|||_n^2.
		\end{equation}
		By a direct computation and Cauchy-Schwarz inequality, we can show 
		\begin{equation}\label{eq.Diri|||ratio}
			\forall f\in\tilde{H}_n, \qquad \sEb_{\rho}(f)\leq 8dn \chi(\rho)^n|||f|||_{n}^2.
		\end{equation}
		We complete the proof by using \eqref{eq.AppL^1L^2} and \eqref{eq.Diri|||ratio} and choosing an appropriate $\ell$:
		\begin{equation*}
			\ell= \max\Ll\{\left\lfloor \left(\frac{n!\chi(\rho)^n|||f|||_n^2}{\sEb_{\rho}(f)}\right)^{\frac{1}{2+nd}}\right\rfloor, n\Rr\}.
		\end{equation*}
	\end{proof}
	
	\section*{Acknowledgements}
	This research is supported by the National Key R\&D Program of China (No. 2023YFA1010400) and NSFC (No. 12301166). We thank Claudio Landim and Jean-Christophe Mourrat for the comments on the preliminary version of the manuscript.

	\bibliographystyle{plain}
	\bibliography{KawasakiRef}
	
\end{document}